\documentclass{cedram-aif}

% ...
\usepackage{amsmath,amssymb,amscd,amsfonts}
\numberwithin{equation}{section}
\theoremstyle{plain}

% ...

\newtheorem{lem}{Lemma}
\newtheorem{cor}{Corollary}
\theoremstyle{definition}
\newtheorem{defn}{Definition}
\newtheorem{rmk}{Remark}

\title
[$K_1$ and $K_2$ of an elliptic curve]
{First and second $K$-groups of an elliptic curve over 
a global field of positive characteristic}

\alttitle
{Sur les premier et second $K$-groupes 
d'une courbe elliptique sur un corps global 
de caract\'eristique positive}

\author{\firstname{Satoshi} \lastname{Kondo}}

\address{
National Research University\\
Higher School of Economics\\
Usacheva St., 7, Moscow 119048 (Russia)\\
Kavli Institute for the Physics and Mathematics \\of the Universe\\
University of Tokyo\\
5-1-5 Kashiwanoha\\
Kashiwa, 277-8583 (Japan)
}

\email{satoshi.kondo@gmail.com}

\thanks{
The first author is grateful to Professor Bloch for numerous discussions 
and to the University of Chicago for their hospitality.     
He also thanks Takao Yamazaki for his interest,
Masaki Hanamura for his comments, and Ramesh Sreekantan for discussions.
The second author would like to thank Kazuya Kato
for his comments regarding the logarithmic de Rham-Witt complex.
Some work related to this paper was done during our stay
at University of M\"unster;  we thank Professor Schneider
who made our visit possible.
We also thank the referee whose suggestions significantly improved this 
paper.
During this research, the first author was supported as a 
Twenty-First Century COE Kyoto Mathematics Fellow, and was 
partially supported by the JSPS Grant-in-Aid for Scientific
Research 17740016
and by the World Premier International Research Center Initiative (WPI Initiative), MEXT, Japan.
The second author was partially supported by the JSPS Grant-in-Aid for Scientific
Research 15H03610, 24540018, 21540013, 16244120.
}

\author{\firstname{Seidai}  \lastname{Yasuda}}
\address{Department of Mathematics,\\
Osaka University\\
Toyonaka, Osaka 560-0043 (Japan)
}
\email{s-yasuda@math.sci.osaka-u.ac.jp}

  
\altkeywords{K-th\'eorie, corps de fonctions, 
courbe elliptique, cohomologie motivique}


\newcommand{\A}{\mathbb{A}}

\newcommand{\F}{\mathbb{F}}
\newcommand{\Fbar}{\overline{\F}}

\newcommand{\Gm}{\mathbb{G}_m}

\newcommand{\Q}{\mathbb{Q}}

\newcommand{\Z}{\mathbb{Z}}
\newcommand{\cA}{\mathcal{A}}

\newcommand{\cF}{\mathcal{F}}

\newcommand{\cE}{\mathcal{E}}
\newcommand{\cEU}{\mathcal{E}_U}

\newcommand{\cK}{\mathcal{K}}
\newcommand{\cL}{\mathcal{L}}
\newcommand{\cM}{\mathcal{M}}
\newcommand{\cN}{\mathcal{N}}
\newcommand{\cO}{\mathcal{O}}

\newcommand{\cX}{\mathcal{X}}
\newcommand{\cZ}{\mathcal{Z}}
\newcommand{\Cbar}{\overline{C}}

\newcommand{\cEbar}{\overline{\cE}}
\newcommand{\cEsupUbar}{\overline{\cE^U}}
\newcommand{\Dbar}{\overline{D}}

\newcommand{\Ubar}{\overline{U}}

\newcommand{\Xbar}{\overline{X}}

\newcommand{\ab}{\mathrm{ab}}
\newcommand{\alg}{\mathrm{alg}}

\newcommand{\chara}{\mathrm{char}}
\newcommand{\Cone}{\mathrm{Cone}}

\newcommand{\Div}{\mathrm{Div}}

\newcommand{\Ext}{\mathrm{Ext}}

\newcommand{\Frac}{\mathrm{Frac}}
\newcommand{\Frob}{\mathrm{Frob}}

\newcommand{\Gal}{\mathrm{Gal}}
\newcommand{\GFq}{G_{\mathbb{F}_q}}

\newcommand{\gon}{\mathrm{gon}}

\newcommand{\Hom}{\mathrm{Hom}}
\newcommand{\id}{\mathrm{id}}
\newcommand{\Image}{\mathrm{Im}}
\newcommand{\Ker}{\mathrm{Ker}}
\newcommand{\Lie}{\mathrm{Lie}}
\newcommand{\mot}{\mathrm{mot}}

\newcommand{\Pic}{\mathrm{Pic}}

\newcommand{\Spec}{\mathrm{Spec}}

\newcommand{\tors}{\mathrm{tors}}
\newcommand{\vep}{\mathrm{\varepsilon}}

\newcommand{\inj}{\hookrightarrow}
\newcommand{\resp}{resp.\ }
\newcommand{\xto}[1]{\xrightarrow{#1}}
\newcommand{\wt}[1]{\widetilde{#1}}

\newcommand{\chern}{\mathrm{ch}}

\newcommand{\etabar}{{\overline{\eta}}}
\newcommand{\etale}{{\mathrm{et}}}

\newcommand{\sm}{\mathrm{sm}}
\newcommand{\sing}{\mathrm{sing}}

\newcommand{\cf}{cf.\,}

\newcommand{\red}{\mathrm{red}}
\newcommand{\Irr}{\mathrm{Irr}}

\newcommand{\cotors}{\mathrm{red}}
\newcommand{\NS}{\mathrm{NS}}
\newcommand{\CH}{\mathrm{CH}}

\newcommand{\cl}{\mathrm{cl}}
\newcommand{\divi}{\mathrm{div}}
\newcommand{\Zar}{\mathrm{Zar}}
\newcommand{\Map}{\mathrm{Map}}
\newcommand{\Coker}{\mathrm{Coker}}

\newcommand{\cont}{\mathrm{cont}}
\newcommand{\cris}{{\mathrm{crys}}}
\newcommand{\Tr}{\mathrm{Tr}}

\newcommand{\dRCW}[1]{CW\Omega_X^{#1}}
\newcommand{\dRCWlog}[1]{CW\Omega_{X,\log}^{#1}}
\newcommand{\dRW}[2]{W_{#1}\Omega_{X}^{#2}}
\newcommand{\dRWlog}[2]{W_{#1}\Omega_{X,\log}^{#2}}
\newcommand{\dRWt}[2]{\wt{W}_{#1}\Omega_{X}^{#2}}
\newcommand{\dRWtlog}[2]{\wt{W}_{#1}\Omega_{X,\log}^{#2}}

\newcommand{\Jac}{\mathrm{Jac}}

\newcommand{\kwp}{{\kappa(\wp)}}

\newcommand{\surj}{\twoheadrightarrow}

\newcommand{\naif}{\mathit{naif}}

\begin{document}
\begin{abstract}
In this paper, we show that the maximal divisible subgroup
of 
groups $K_1$ and $K_2$ of
an elliptic curve $E$ over a function field
is uniquely divisible.
Further 
those $K$-groups modulo this uniquely divisible 
subgroup are explicitly computed.
We also 
calculate the motivic cohomology groups of the minimal
regular model of $E$, which is an
elliptic surface over a finite field.
\end{abstract}

\begin{altabstract}
On d\'emontre que les plus grands 
sous-groupes divisibles des groupes
$K_1$ et $K_2$ 
d'une courbe elliptique $E$ sur un corps global de
caract\'eristique positive 
sont uniquement divisibles
et on d\'ecrit explicitement 
les $K$-groupes modulo leurs plus 
grands sous-groupes divisibles.
On calcule \'egalement 
la cohomologie motivique 
du mod\`ele minimal de $E$ 
qui est une surface elliptique sur un corps fini.
\end{altabstract}

\maketitle

\section{Introduction}\label{sec:introduction}
In this paper, we study the kernel and cokernel of boundary maps
in the localization sequence of $G$-theory of the triple:
an elliptic curve $E$ over a global function field $k$ 
of positive characteristic,
the regular minimal model $\cE$ of $E$ that is proper flat over the curve $C$ 
associated with $k$, and the fibers of $\cE \to C$.
The aim of this initial section is to provide our motivation,
describe some necessary background,
and present known results more generally.
Precise mathematical statements regarding our results are presented 
in Section~\ref{sec:intro K}.  Among these statements, 
we suggest that our most interesting point is represented in 
assertion (2) of
Theorem~\ref{thm:intro2b}, which relates the kernel of a boundary map
to the special value $L(E,0)$ of the $L$-function of $E$.

\subsection{Background and conjectures}
\label{sec:intro background}
From Grothendieck's theory of motives, Beilinson envisioned the abelian category of mixed motives over a base (see \cite{Levine4}).     
The existence of such a category remains a conjecture, 
but we now have various constructions of the triangulated category of (mixed) motives, 
which serves as the derived category of the sought-after category.   
Here, motivic cohomology groups are extension groups in the category of mixed motives, or, unconditionally, homomorphism 
groups in the triangulated category of motives.   
Algebraic $K$-theory also enters the picture here
via the Atiyah-Hirzebruch-type spectral sequence (AHSS) of Grayson-Suslin or Levine  \cite{Grayson2}, \cite{Levine3}, \cite{Suslin} 
(see also \cite[(1.8), p.211]{Ge}, \cite{Grayson}), i.e.,
\[
E_2^{s,t}=H_\cM^{s-t}(X, \Z(-t)) \Rightarrow K'_{-s-t}(X)
\]
or via formula \cite[9.1]{Bloch} (see also \cite[p.59]{Grayson})
\[
K_n'(X) \otimes_\Z \Q \cong \bigoplus_i H^{2i-n}_\cM(X, \Q(i))
\]
of Bloch.   
Also applicable here are 
the Riemann-Roch theorem for higher $K$-theory and 
the motivic cohomology theory 
by Gillet~\cite{Gillet}, Levine~\cite{Levine}, Riou~\cite{Riou}, and 
Kondo-Yasuda~\cite{RRwoD}.

Arithmetic geometers are individuals who study motives (or varieties) over number fields, 
over the function fields of curves $C$ over finite fields (denoted function fields for short), 
or over finite fields.   For such a variety, the $L$-function, also known as the zeta function, 
one of the fundamental invariants in arithmetic geometry, is defined.   
The motivation for our work stems from 
a range of conjectures that describe motivic cohomology in terms of special values of these $L$-functions.   
Conjectures over $\Q$ were formulated by Lichtenbaum, Beilinson, and Bloch-Kato (see \cite{Kahn} for a summary of these conjectures).    Using the analogy between number fields and function fields, parts of these 
conjectures can be translated into conjectures over function fields.   

The case over function fields and over finite fields are related as follows. 
Given a variety over a function field, 
take a model, i.e., a variety over the base finite field fibered over $C$ 
whose generic fiber is the given variety.   
Then, the motivic cohomology groups 
and the $L$-functions
(over finite fields, i.e., congruence zeta functions) 
are related,
thus, we obtain some conjectural statements regarding 
motivic cohomology groups of the model and its congruence zeta function from these conjectures over function fields concerning $L$-functions and vice versa. 
Note that Geisser \cite{Ge2} presents a description of $K$-theory and motivic cohomology groups of varieties over finite fields assuming both Parshin's conjecture and Tate's conjecture.

Finally, let us mention the finite generation conjecture of Bass \cite[p.386, Conjecture~36]{Kahn}, 
which states that the $K$-groups of a regular scheme of finite type over $\Spec\, \Z$ are finitely generated.

\subsection{Known results: one-dimensional cases}
Unconditionally verifying the conjectures above is difficult, as is verifying the consequences of the conjectures.    
The key 
theorem here is the Rost-Voevodsky theorem (i.e., the Bloch-Kato conjecture) and, as a consequence,  
the Geisser-Levine theorem 
(i.e., a part of the Beilinson-Lichtenbaum conjecture) \cite[Theorem 1.1, Corollary 1.2]{Ge-Le2}.    
These theorems
state that the motivic cohomology with torsion coefficients and 
\'etale cohomology are isomorphic in certain range of bidegrees.   

These theorems and the aforementioned AHSS can be used as the 
primary ingredients for computing
the $K$-groups and motivic cohomology 
groups of the ring of integers of global fields.   
In these ``one-dimensional" cases, 
the Bass conjecture on finite generation 
is known to hold based on work by 
Quillen~\cite{Quillen2} and Grayson-Quillen~\cite{Grayson}. 
In positive characteristic, 
these ingredients are sufficient 
for computing 
the motivic cohomology groups, and 
a substantial amount of computation 
can also be performed in the number field case.  
We refer to \cite{Weibel} for details on these unconditional results.

Unfortunately, much less is known regarding higher dimensions.   
Our results may be regarded as the next step
in determining this in that we 
treat elliptic curves over a function field and elliptic surfaces over a finite field.    The theorems of Rost-Voevodsky and Geisser-Levine 
do indeed cover many of the bidegrees, but there remain
lower bidegrees that require additional work.   
%

\subsection{Computing motivic cohomology groups}
\label{sec:intro CSS}
To compute motivic cohomology with $\Z$-coefficient, 
one first 
divides the problem into 
$\Q$-coefficient, torsion coefficient,
and the divisible part in the cohomology with $\Z$-coefficient.
Then, for the prime-to-the-characteristic coefficient, 
one uses the \'{e}tale realization, i.e., a map to the \'{e}tale cohomology.   
For the $p$-part, one uses the map to the cohomology of de Rham-Witt complexes.
Here, the target groups are presumably easier to compute than the motivic cohomology groups, and these 
realization maps are now known to be isomorphisms in many cases,
due to Rost-Voevodsky and Geisser-Levine; however,    
no general method for computing the $\Q$-coefficient and
divisible parts is known. 
In one-dimensional cases, we have finite generation theorems, which then 
imply that the divisible part is zero and the $\Q$-coefficient part is finite dimensional, and the $\Q$-coefficient part is 
zero except for bidegrees $(0,0)$, $(1,1)$, and $(2,1)$.
For varieties over finite fields, the divisible part is conjecturally always zero, 
but this has not been shown in higher dimensions. 

In this paper, we therefore study the weaker problem as to 
whether the divisible part is uniquely divisible for a smooth surface 
$X$ over a finite field. 
In Section~\ref{sec:Motcoh}, we prove that the divisible part of motivic cohomology groups, which may still contain torsion, 
is actually uniquely divisible except for bidegree $(3, 2)$, 
and we explicitly compute the quotient group.
Further, by using the main result of our work in \cite{KY},
we achieve a similar statement for the exceptional bidegree when $X$ is a model of an elliptic curve over a function field.

When the coefficient in 
the motivic cohomology groups is the second Tate twist, especially when
the bidegree is $(3, 2)$, 
many results closely related to the computation of the torsion coefficients were obtained in a series of papers by Raskind, Colliot-Th\'{e}l\`{e}ne, Sansuc, Soul\'{e}, Gros, and Suwa \cite{Co-Ra}, \cite{CSS}, \cite{Gr-Su}. 
The general strategy is described by Colliot-Th\'{e}l\`{e}ne-Sansuc-Soul\'{e} in \cite{CSS}, in which the primary focus was on the 
prime-to-$p$ part.  
The $p$-part was treated by Gros-Suwa \cite{Gr-Su}, though we 
need a supplementary statement on the $p$-part,
which we provide in the Appendix. 
We follow their outline closely in Section~\ref{sec:Motcoh} below. 

We note here that the series of results above are written in terms of $\cK$-cohomology groups and the associated realization maps.   
The $\cK$-cohomology groups were used as a substitute for motivic cohomology groups and are now known to be isomorphic to motivic cohomology groups in particular bidegrees, as shown 
by Bloch, Jannsen, Landsburg, M\"{u}ller-Stach, and Rost (see \cite[p.297, Bloch's formula, Corollary 5.3, Theorem 5.4]{Muller-Stach1}).    Some earlier results do imply some of the results in our paper via this comparison isomorphism; 
however we present our exposition independent of earlier results and ensure that
our work is self-contained for two reasons.
First, if we used earlier results, in possible applications of our result,
some compatibility checks would have become necessary.  
For example, in earlier papers, a certain map from $\cK$-cohomology to \'etale cohomology groups was used.  Conversely, we use a cycle map from higher Chow groups given by Geisser-Levine~\cite{Ge-Le2}.   
It is nontrivial to check if these two maps are compatible under the comparison isomorphism.   
Therefore, while we do refer to earlier papers, 
we use only the non-motivic statements, 
reproducing the motivic statements in our paper where applicable. 
Second, some earlier results (e.g., those of \cite{Gr-Su}) are usually stated for projective schemes. We also need similar results for not necessarily projective schemes, since we treat a curve over a function field as the limit of surfaces fibered over (affine) curves over a finite field. 
When we remove the condition that the scheme considered is projective,
it becomes much more difficult to show that the divisible part of motivic cohomology is uniquely divisible. 
 
\subsection{Motivic cohomology and $K$-theory}
\label{sec:intro Chern}
In our paper, we use Chern classes to relate $K$-theory and motivic cohomology.
We use the Riemann-Roch theorem without denominators, and we perform some direct
computations for singular curves over finite fields.   
We further explain these two independent issues below.   
Finally, in the last paragraph of this subsection, we describe the application.

The first issue is the comparison isomorphism between 
Levine's motivic cohomology and Bloch's higher Chow groups, an issue we detail 
in 
Section~\ref{sec:Compatibility}.   
The general problem is summarized as follows.  We have various constructions of motivic cohomology theory (or groups), i.e., Suslin-Voevodsky-Friedlander~\cite{Voevodsky}, Hanamura~\cite{Hanamura}, Levine~\cite{Levine}, Bloch~\cite{Bloch} (higher Chow groups), Morel-Voevodsky~\cite{MorelVoevodsky}, and the graded pieces of rational $K$-groups.
The motivic cohomology groups defined in various ways are generally known to be isomorphic, however, it is nontrivial to verify whether the comparison isomorphisms respect other structures, such as functoriality, Chern classes (characters), localization sequences, and product structures, etc.   
The Riemann-Roch theorem without denominators is known to hold 
in the following three cases: by Gillet for cohomology theories satisfying his axioms \cite{Gillet}; by Levine for his motivic cohomology theory \cite{Levine}; and by Kondo-Yasuda in the context of motivic homotopy theory of Morel and Voevodsky \cite{RRwoD}.   
None of these directly apply to the 
motivic cohomology groups in our paper, 
namely the higher Chow groups of Bloch,
hence we use the comparison isomorphism between higher Chow groups and Levine's motivic cohomology groups, 
importing the Riemann-Roch theorem of Levine into our setting by checking 
various compatibilities.  The compatibility proof for 
Gysin maps and localization sequences under the comparison isomorphism are 
thus the focus of Section~\ref{sec:Compatibility}.

The second issue is the integral construction of Chern characters for singular curves over finite fields, 
which we cover more fully 
in Section~\ref{sec:MotChern}.   
The Chern characters from $K$-theory (or $G$-theory) to motivic cohomology are usually defined for nonsingular varieties and come with denominators in higher degrees.   In our paper, 
we focus on elliptic fibrations; 
singular curves over a finite field appear naturally as fibers, 
and we must study their $G$-theory.   We do so by constructing a generalization of Chern characters in an ad hoc manner, generalizing those in the smooth case.   To be able to use such an approach, 
we make substantial use of the fact that we consider only curves (however singular) over finite fields.   
We use the known computation of motivic cohomology and $K$-groups of nonsingular curves over finite fields, and, as one additional ingredient, we use a result from our previous work (Lemma~\ref{prop:AppB_main}).    
The principal result of Section~\ref{sec:MotChern} is Proposition~\ref{prop:Z} which states that the 0-th and 1-st 
$G$-groups are integrally isomorphic 
via these Chern characters to the direct sum of relevant motivic cohomology groups with a $\Z$-coefficient.

Finally, in Section~\ref{sec:KgpsofCurves},
we apply results from Sections~\ref{sec:Motcoh},~\ref{sec:Compatibility}, and~\ref{sec:MotChern},
to compute the lower $K$-groups of a curve over a function field in terms of motivic cohomology groups.  An additional ingredient here 
is the computation of $K_3$ of fields given by Nesterenko-Suslin~\cite{Ne-Su} and Totaro~\cite{To}.

\subsection{Statement of results on $K$-groups}
\label{sec:intro K}
The aim of our paper is to explicitly compute 
the $K_1$ and  $K_2$ groups, 
the motivic cohomology groups of an elliptic curve over a
function field, and the motivic cohomology
groups of an elliptic surface over a finite field.
In this subsection, we provide the precise statements of 
our results 
(i.e., Theorems~\ref{thm:intro2a} and~\ref{thm:intro2b}) 
concerning 
$K_1$ and $K_2$ groups of an elliptic curve.
We refer to 
Theorems~~\ref{thm:conseq2},~\ref{thm:conseq3},
~\ref{thm:conseq5}, and~\ref{thm:conseq6} below 
for more detailed results concerning motivic cohomology groups.

Let $k$ be a global field of positive characteristic $p$,
and let $E$ be an elliptic curve over $\Spec\, k$.
Let $C$ be the proper smooth irreducible curve over
a finite field whose function field is $k$.
We regard a place $\wp$ of $k$ as a closed point of $C$, 
and vice versa.
Let $\kappa(\wp)$ denote the residue field at $\wp$ of $C$.
Let $f: \cE \to C$
denote the minimal regular model of 
the elliptic curve $E \to \Spec\, k$.  
This $f$ is a
proper, flat, generically smooth 
morphism such that for almost all closed points $\wp$ of $C$,
the fiber $\cE_\wp=\cE \times_C \Spec\, \kwp$ at $\wp$ is a genus one curve,
and 
such that the generic fiber is the elliptic curve $E \to \Spec\, k$.


Let us identify the $K$-theory and $G$-theory of 
regular Noetherian schemes.  
First, there is a localization sequence of $G$-theory,
i.e.,
\[
K_i(\cE) \to K_i(E) \xto{\oplus \partial^i_\wp} 
\bigoplus_{\wp} G_{i-1}(\cE_\wp) \to K_{i-1}(\cE),
\]
in which $\wp$ runs over all primes of $k$.

For a scheme $X$, we let $X_0$
be 
the set of closed points
of $X$.
%
Let us use $\partial_2$ to denote the boundary map
$\oplus \partial^2_\wp$,
$\partial: K_2(E)^\cotors \to \bigoplus_{\wp \in C_0}G_0(\cE_\wp)$
to denote the boundary map induced by $\partial_2$,
and 
$\partial_1: K_1(E)^\cotors \to \bigoplus_{\wp \in C_0}G_0(\cE_\wp)$
to denote the boundary map induced by $\oplus\partial^1_\wp$.
Here, for an abelian group $M$, we let $M_\divi$ denote
the maximal divisible subgroup of $M$,
and set $M^\cotors = M/M_\divi$.

In Theorems~\ref{thm:intro2a} and~\ref{thm:intro2b} below,
we use the following notation.
We use subscript $-_\Q$ to mean $-\otimes_\Z \Q$.  
%
For a scheme $X$, let 
$\Irr(X)$ be 
the set of irreducible components 
of $X$.
Let $\F_q$ be the field of constants of $C$.
For a scheme $X$ of finite type over $\Spec\, \F_q$ and
for $i \in \Z$, choose a prime number $\ell \neq p$
and set 
$$L(h^i(X),s) = 
\det(1-\Frob\cdot q^{-s};H^i_\etale(X\times_{\Spec\,\F_q}\Spec\, \Fbar_q, \Q_\ell)),$$
where $\Fbar_q$ is an algebraic closure of $\F_q$ and 
$\Frob \in \Gal(\Fbar_q/\F_q)$ is the geometric Frobenius element.
In all cases considered in Theorems~\ref{thm:intro2a} and 
\ref{thm:intro2b}, the function $L(h^i(X),s)$ does not 
depend on the choice of $\ell$.
Let $T'_{(1)}$ denote what we call the twisted Mordell-Weil group 
\[
T'_{(1)}= \bigoplus_{\ell \neq p} 
(E(k\otimes_{\F_q}\Fbar_q)_\tors 
\otimes_{\Z}\Z_{\ell}(1))^{\Gal(\Fbar_q/\F_q)}.
\]
We write $S_0$ (resp. $S_2$) for the set of 
primes of $k$ at which $E$ has split multiplicative (resp. bad) reduction;
we also regard it as a closed subscheme of $C$ with the reduced structure; note that the set of primes $S_1$ will be introduced later.
Further, for a set $M$, we denote its cardinality by $|M|$.  Finally, 
we let $r=|S_0|$.
\begin{thm}[see Theorem \ref{thm:conseq1}(1)(2) and 
Theorem~\ref{thm:conseq6}(2)]
\label{thm:intro2a}
Suppose that $S_2$ is non-empty, or equivalently 
that $f$ is not smooth.
\begin{enumerate}
\item The dimension of the $\Q$-vector space $(K_2(E)^\cotors)_\Q$
is $r$.
\item The cokernel of the boundary map
$\partial_2 : K_2(E) \to \bigoplus_{\wp \in C_0}
G_1(\cE_\wp)$ is a finite group of order
$$
\frac{(q-1)^2 |L(h^0(\Irr(\cE_{S_2})),-1)|}{
|T'_{(1)}|\cdot |L(h^0(S_2),-1)|}.
$$
\item 
The group $K_2(E)_\divi$ is uniquely divisible, and 
the kernel of the boundary map
$\partial: K_2(E)^\red \to \bigoplus_{\wp \in C_0}
G_1(\cE_\wp)$ is a finite group
of order 
$|L(h^1(C),-1)|^2$
\end{enumerate}
\end{thm}
%
Let $L(E,s)$ denote the $L$-function of $E$ 
(see Section~\ref{sec: L-function}).
We write $\mathrm{Jac}(C)$ for the Jacobian of $C$.
\begin{thm}[see Theorem \ref{thm:conseq1}(3)(4)]
\label{thm:intro2b}
Suppose that $S_2$ is non-empty, or equivalently that $f$ is not smooth.
\begin{enumerate}
\item The group $K_1(E)_\divi$ is uniquely divisible.
\item The kernel of  the boundary map
$\partial_1: K_1(E)^\cotors \to \bigoplus_{\wp \in C_0}G_0(\cE_\wp)$
is a finite group of order $(q-1)^2 |T'_{(1)}|\cdot |L(E,0)|$.
The cokernel of $\partial_1$ is a
finitely generated abelian group of rank
$2+|\Irr(\cE_{S_2})|-|S_2|$ whose torsion
subgroup is isomorphic to $\mathrm{Jac}(C)(\F_q)^{\oplus 2}$.
\end{enumerate}
\end{thm}

\subsection{Similarities with the Birch-Tate conjecture}
\label{sec:intro BirchTate}
In this subsection, we describe similarities 
between 
our statements and the Birch-Tate 
conjecture
\cite[pp.206--207]{BirchTate}.
Although there is no direct connection between them,
we hope this subsection provides some 
justification 
for the way  
our statements are formulated.   
The similarities were pointed out by Takao Yamazaki.

Our results for $K$-theory are stated in 
terms of the boundary map in the localization sequence.   
For example, in Theorem~\ref{thm:intro2b}(2), in which 
we describe 
the $K_1$ group 
of an elliptic curve over a function field, 
we consider the boundary map to the $G_0$-groups of the fibers.    Since the $G$-groups of curves over finite fields (i.e., the fibers) are known, the statement 
gives a description of the $K_1$ group.    
This type of description is not directly related to the conjectures of Lichtenbaum and Beilinson.

The Birch-Tate conjecture
focuses on the
$K_2$ group of a global field in any characteristic.
They study the boundary map, 
which they denote $\lambda$,
from the $K_2$ group to the direct sum of
the $K_1$ group of the residue fields.
A part of their
conjecture is that 
the order of the 
kernel
of the boundary map $\lambda$
is
expressed using 
the special value $\zeta_F(-1)$
and the invariant $w_F$, 
which is expressed in terms of
the number of roots of unity.

In our theorem,
we consider the $K_1$ group 
instead of the $K_2$ group,
an elliptic curve over a function field
instead of the function field (i.e., a global field with positive 
characteristic), 
the Hasse-Weil $L$-function $L(E,s)$
instead of the zeta function of the global field.
Note that the value $|T_{(1)}'|$ in our setting plays the role of
$w_F$.

Next, we note that there is no counterpart 
for the factor $(q-1)^2$ 
in their conjecture,
because there is more than one degree 
for
which the cohomology groups
are conjecturally nonzero in our setting,
 in turn because the variety is one-dimensional
as opposed to zero-dimensional; 
the global field itself is regarded 
as a zero-dimensional 
variety
over the global field.

Finally, while it may be interesting to degenerate 
the elliptic 
curve $E$  
to compare our results with the conjecture,
we have not pursued
this point.   

\subsection{Higher genus speculation}
\label{sec:intro high genus}
In this subsection, we speculate
on higher genus cases,
which in turn points to the reason why
we restrict our efforts 
to elliptic curves and 
to where the difficulty 
lies in our results.
More specifically, the results of Sections~\ref{sec:Motcoh},~\ref{sec:Compatibility},~\ref{sec:MotChern}, and~\ref{sec:KgpsofCurves} are valid and proved for curves of arbitrary genus.    Only in Sections~\ref{conseq} and~\ref{sec:BlKa} do we use the fact that $E$ is of genus 1, which we use in two ways,
one being through Theorem~\ref{surjectivity}, the other through Lemma~\ref{lem:appl_pi1}.
The former is more motivic, whereas the second is primarily 
used in the 
computation of the 
\'{e}tale cohomology of elliptic surfaces and is not 
so motivic in nature.

Note that we separately treat cases $j \ge 3$ and $j \le 2$ in Sections~\ref{sec:BlKa} and~\ref{conseq}, respectively.  
The former case is less involved in that it does not use Theorem~\ref{surjectivity}.   
We use some consequences of Lemma~\ref{lem:appl_pi1} 
that express the cohomology of an elliptic surface in terms of 
the base curve, 
hence our results appear as if there are  little contributions 
from the cohomology of~$\cE$.  
The second author suspects that a result analogous 
to Lemma~\ref{lem:appl_pi1}
should also hold true in higher genus cases,
thus we may formulate statements 
that look similar to 
those given in our paper.

The results of Section~\ref{conseq}, i.e., 
for the $j\le 2$ case, 
use the following theorem as an additional ingredient:
\begin{thm}[{\cite[Theorem 1.1]{KY}}]
\label{surjectivity}
Let the notations be as in \mbox{Section}~\ref{sec:intro K}. 
For an arbitrary set $S$ 
of closed points of $C$, the homomorphism 
induced by the boundary map $\partial_2$, i.e.,
\[
K_2(E)_\Q \xto{\oplus \partial_{\wp \Q}} \bigoplus_{\wp\in S}G_1(\cE_\wp)_\Q
\]
is surjective.
\end{thm}
\noindent 
Theorem~\ref{surjectivity} is a consequence of Parshin's conjecture, 
hence a similar statement is expected to hold even if we replace $E$ by a curve of higher genus, 
but the proof is not known.  
 In \cite{CKY}, Chida et al.  
constructed curves other than elliptic curves 
for which surjectivity similar to 
that of Theorem~\ref{surjectivity} holds true.   
We might be able to draw some consequences  
similar to the theorems in our present paper 
for those curves.

\subsection{Organization}
Finally, in this subsection, 
we  provide 
a short explanation of the content of each section.
As seen above, 
the motivation and brief explanation regarding the results of
Section~\ref{sec:Motcoh} 
and Sections~\ref{sec:Compatibility}-\ref{sec:KgpsofCurves} 
are given
in Sections~\ref{sec:intro CSS} and~\ref{sec:intro Chern}, respectively.
Each of these sections is fairly independent.

In Section~\ref{sec:Motcoh},  we compute the motivic 
cohomology groups of an arbitrary smooth surface $X$ over finite fields.
The difficult case is that of $H_\cM^3(X, \Z(2))$, 
which 
is therefore the focus of 
Section~\ref{subsec:criterion(3,2)}.
In Section~\ref{sec:Compatibility}, we prove the compatibility 
of Chern characters and the localization sequence of motivic cohomology.
%
In Section~\ref{sec:MotChern},  we define Chern 
characters for singular curves over finite fields,
though
the treatment is quite ad hoc.
In Section~\ref{sec:KgpsofCurves},  we then give 
the relation via 
the Chern class map between the $K_1$ and $K_2$ groups 
of curves 
over function fields
and motivic cohomology groups.
We present our main 
result in Section~\ref{conseq}.
Then, applying results of Sections~\ref{sec:Motcoh}, 
\ref{sec:MotChern}, and \ref{sec:KgpsofCurves},
and using special features of elliptic surfaces, 
including Theorem~\ref{surjectivity}, 
we explicitly compute the orders of 
certain torsion groups.
We treat 
the $p$-part separately in Appendix~\ref{Appendix p-part};
see its introduction for more technical details.  
In Section~\ref{sec:BlKa}, we use the Bloch-Kato 
conjecture, 
as proved by Rost and Voevodsky 
(see Theorem~\ref{conj:BK}), 
and generalize our results in Section~\ref{conseq}. 

\section{Motivic cohomology groups of smooth surfaces}
\label{sec:Motcoh}

We begin this section by 
referring to Section~\ref{sec:intro CSS},
which provides a general overview of the contents of this section.

Aside from the uniquely divisible part, we understand the motivic cohomology 
groups of smooth surfaces over finite fields fairly well.  It follows from the Bass conjecture 
that the divisible part is zero \cite[Conjecture 37, p.387]{Kahn}.
 
The main goal of this section is to prove Theorem~\ref{thm:motfin},
therefore let us provide 
a brief description of the statement.
Let $X$ be a smooth surface over a finite field.
We find that 
$H^i_\cM(X, \Z(j))$ 
is an extension of a finitely generated abelian group
 by a uniquely divisible group except when $(i,j)=(3,2)$.
Aside from cases
 $(i,j)=(0,0), (1,1),(3,2)$, and $(4,2)$,
the finitely generated abelian group 
is a finite group, which is either zero or
is written in terms of \'etale cohomology groups of $X$.
We refer to Theorem~\ref{thm:motfin},
the following table, and the following paragraph
for further details.
 Note that for a prime number $\ell$, we let
$|\ |_\ell :\Q_\ell \to \Q$ denote the $\ell$-adic
absolute value normalized such 
that $|\ell|_\ell = \ell^{-1}$.

\subsection{Motivic cohomology of surfaces over a finite field}
\label{subsec:Motcoh}
Let $\F_q$ be a field of cardinal $q$ of characteristic $p$.
For a separated scheme $X$ which is essentially 
of finite type over $\Spec\, \F_q$, we define the motivic
cohomology group $H^i_\cM(X,\Z(j))$ as
the homology group $H^i_\cM(X,\Z(j))=H_{2j-i}(z^j(X,\bullet))$ 
of Bloch's cycle complex $z^j(X,\bullet)$ 
\cite[Introduction, p.~267]{Bloch} 
(see also \cite[2.5, p.~60]{Ge-Le2} to remove the
condition that $X$ is quasi-projective). 
When $X$ is quasi-projective,
we have $H_\cM^i(X, \Z(j))=\CH^j(X,2j-i)$,
where the right-hand side is the higher 
Chow group of Bloch \cite{Bloch}.
We say that a scheme $X$ is essentially smooth over a field $k$
if $X$ is a localization of a smooth scheme of finite type over $k$.
When $X$ is essentially smooth
over $\Spec\, \F_q$, it coincides with the motivic cohomology
group defined in \cite[Part I, Chapter I, 2.2.7, p.~21]{Levine} or
\cite{Voevodsky} (\cf \cite[Theorem 1.2, p.~300]{Levine2}, 
\cite[Corollary 2, p.~351]{Voevodsky2}).
For a discrete abelian group $M$, we set
$H^i_\cM(X, M(j))=H_{2j-i}(z^j(X,\bullet)\otimes_\Z M)$.

We note here that this notation is inappropriate if  $X$ is not 
essentially smooth,
for in that case it would be a Borel-Moore homology
group.    
A reason for using this notation is that 
in Section~\ref{sec:def chern sing}, we define 
Chern classes for higher Chow groups of low degrees
as if higher Chow groups 
were forming a cohomology theory.

First, let us recall that 
the groups $H^i_\cM(X,\Z(j))$ have been known for $j \le 1$.
By definition, $H^i_\cM(X,\Z(j))=0$ for 
$j \le 0$ and $(i,j) \neq (0,0)$, and $H^0_\cM(X,\Z(0))=
H^0_\Zar(X,\Z)$. We have $H^i_\cM(X,\Z(1))=0$
for $i \neq 1,2$.  By \cite[Theorem 6.1, p.~287]{Bloch}, we
have $H^1_\cM(X,\Z(1))=H^0_\Zar(X,\Gm)$,
and $H^2_\cM(X,\Z(1))=\Pic(X)$.

Below is 
a conjecture of Bloch-Kato
\cite[\S1, Conjecture 1, p.~608]{Kato},
which has been 
proved by Rost, Voevodsky, Haesemeyer, and Weibel.
\begin{thm}[Rost-Voevodsky; the Bloch-Kato conjecture]
\label{conj:BK}
Let $j \ge 1$ be an integer.
Then, for any finitely generated field $K$
over $\F_q$ and any positive integer 
$\ell \neq p$,
the symbol map 
$K_j^M(K) \to H^j_\etale(\Spec\, K, \Z/\ell(j))$
is surjective.
\end{thm}

\begin{defn}
Let $M$ be an abelian group. We say that
$M$ is {\em finite} {\em modulo} 
{\em a uniquely divisible
subgroup}
(\resp {\em finitely generated modulo} 
{\em a uniquely} {\em divisible subgroup}) 
if $M_\divi$ is uniquely divisible and $M^\cotors$ is finite
(\resp $M_\divi$
is uniquely divisible and $M^\cotors$ is finitely generated).
\end{defn}
We note that if $M$ is finite modulo a uniquely
divisible subgroup, then $M_\tors$ is a finite group and 
$M = M_\divi \oplus M_\tors$.

Recall that for a scheme $X$, we let $\Irr(X)$ denote
the set of irreducible components of $X$.
The aim of Section~\ref{subsec:Motcoh} is to prove
the below theorem.
\begin{thm}\label{thm:motfin}
Let $X$ be a smooth surface over $\F_q$.
Let $R \subset \mathrm{Irr}(X)$ denote the subset of
irreducible components of $X$ that are projective 
over $\Spec\, \F_q$. For $X' \in \mathrm{Irr}(X)$, let
$q_{X'}$ denote the cardinality of the field
of constants of $X'$.
\begin{enumerate}
\item $H^i_\cM(X,\Z(2))$ is
finitely generated modulo a uniquely divisible subgroup
if $i \neq 3$ or if $X$ is projective. More precisely,
\begin{enumerate}
\item $H^i_\cM(X,\Z(2))$ is zero for $i \ge 5$.
\item $H^4_\cM(X,\Z(2))$ is a finitely generated abelian 
group of rank $|R|$.
\item If $i \le 1$ or if $X$ is projective and $i \le 3$,
the group $H^i_\cM(X,\Z(2))$ is
finite modulo a uniquely divisible subgroup.
\item $H^2_\cM(X,\Z(2))$ is finitely generated
modulo a uniquely divisible subgroup.
\item For $i \le 2$, the group $H^i_\cM(X,\Z(2))_\tors$ 
is canonically isomorphic to the direct sum
$\bigoplus_{\ell \neq p}H^{i-1}_\etale(X,\Q_\ell/\Z_\ell(2))$.
In particular, $H^i_\cM(X,\Z(2))$ 
is uniquely divisible for $i \le 0$.
\item If $X$ is projective, then $H^3_\cM(X,\Z(2))_\tors$ 
is isomorphic to the direct sum of the group 
$\bigoplus_{\ell \neq p}H^{2}_\etale(X,\Q_\ell/\Z_\ell(2))$
and a finite $p$-group of order 
$|\Hom(\Pic^o_{X/\F_q},\Gm)| \cdot |L(h^2(X),0)|_p^{-1}$.
Here, we let $\Hom(\Pic^o_{X/\F_q},\Gm)$ 
denote the set of morphisms 
$\Pic^o_{X/\F_q}\to \Gm$ of group schemes
over $\Spec\, \F_q$.
\end{enumerate}
\item Let $j \ge 3$ be an integer. 
Then, for any integer $i$, the group
$H^i_\cM(X,\Z(j))$ is finite modulo 
a uniquely divisible subgroup. 
More precisely,
\begin{enumerate}
\item $H^i_\cM(X,\Z(j))$ is
zero for $i\ge \max(6,j+1)$, 
is isomorphic to 
$\bigoplus_{X' \in R} \Z/(q_{X'}^{j-2}-1)$
for $(i,j)=(5,3)$, $(5,4)$,
and is finite for $(i,j)=(4,3)$.
\item $H^i_\cM(X,\Z(j))_\tors$ is canonically 
isomorphic to the direct sum
$\bigoplus_{\ell \neq p}H^{i-1}_\etale(X,\Q_\ell/\Z_\ell(j))$.
In particular, $H^i_\cM(X,\Z(j))$ 
is uniquely divisible for 
$i \le 0$ or $6 \le i \le j$, and
$H^1_\cM(X,\Z(j))_\tors$ is 
isomorphic to the direct sum 
$\bigoplus_{X' \in \mathrm{Irr}(X)} \Z/(q_{X'}^{j}-1)$.
\end{enumerate}
\end{enumerate}
\end{thm}

In the table below, we summarize the description  
of the 
groups
$H^i_\cM(X,\Z(j))$ stated in Theorem~\ref{thm:motfin}.  
Here, we write u.d, f./u.d., f.g./u.d., f., and f.g. for uniquely divisible, finite modulo a uniquely divisible subgroup, 
finite generated modulo a uniquely divisible 
subgroup, finite, and finitely generated,
respectively.
$$
\begin{array}{|c||c|c|c|c|c|c|c|}
\hline
\, j \setminus i \, &  \, < 0 \, & 0 &
\, 0 < i <j\, & j (\neq 0) & j+1 & \, j+2\, & \,\ge j+3\, \\
\hline
\hline
0 & 0 & \, H^0(\Z)\,  & \multicolumn{2}{c|}{\text{-}} & 
\multicolumn{3}{c|}{0} \\
\hline
1 & \multicolumn{2}{c|}{0}
& \text{-} & H^0(\Gm) & \Pic(X)
& \multicolumn{2}{c|}{0} \\
\hline
\raisebox{-1.6ex}[0pt][0pt]{$2$} & 
\multicolumn{2}{c|}{\raisebox{-1.6ex}[0pt][0pt]{\text{u.\ d.}}}
& \raisebox{-1.6ex}[0pt][0pt]{\text{f./u.\ d.}}
& \text{\,f.\ g./u.\ d.\,} & \text{?}
& \raisebox{-1.6ex}[0pt][0pt]{\text{f.\ g.}} 
& \raisebox{-1.6ex}[0pt][0pt]{$0$} \\
\cline{5-6}
& \multicolumn{2}{c|}{} & &
\multicolumn{2}{c|}{\text{\,f./u.\! d.\! if projective\,}}
& & \\
\hline
3 & \multicolumn{2}{c|}{\text{u.\ d.}}
& \multicolumn{2}{c|}{\text{f./u.\ d.}}
& \multicolumn{2}{c|}{\text{f.}} & 0 \\
\hline
4 & \multicolumn{2}{c|}{\text{u.\ d.}}
& \multicolumn{2}{c|}{\text{f./u.\ d.}}
& \text{f.} & \multicolumn{2}{c|}{0} \\
\hline
\raisebox{-1.6ex}[0pt][0pt]{$\ge 5$} & 
\multicolumn{2}{c|}{\raisebox{-1.6ex}[0pt][0pt]{\text{u.\ d.}}}
& \multicolumn{2}{c|}{\text{f./u.\ d.}}
& \multicolumn{3}{c|}{\raisebox{-1.6ex}[0pt][0pt]{$0$}} \\
\cline{4-5}
& \multicolumn{2}{c|}{} 
& \multicolumn{2}{c|}{\text{u.\ d.\ if }6 \le i \le j}
& \multicolumn{3}{c|}{} \\
\hline
\end{array}
$$

\begin{lem}\label{lem:toritori}
Let $X$ be a separated scheme essentially 
of finite type over $\Spec\, \F_q$.  Let $i$ and $j$ be integers.
If both $H^{i-1}_\cM(X,\Q/\Z(j))$ and
$\varprojlim_m H^i_\cM(X,\Z/m(j))$ are finite,
then $H^i_\cM(X,\Z(j))$ 
is finite modulo a uniquely
divisible subgroup and its torsion subgroup is
isomorphic to $H^{i-1}_\cM(X,\Q/\Z(j))$.
\end{lem}
\begin{proof}
Let us consider the exact sequence
\begin{equation} \label{eqn:shortexact_motcoh}
\begin{array}{l}
0 \to H^{i-1}_\cM(X,\Z(j))\otimes_\Z \Q/\Z
\to H^{i-1}_\cM(X,\Q/\Z(j))  \\ \to
H^{i}_\cM(X,\Z(j))_\tors \to 0.
\end{array}
\end{equation}
Since $H^{i-1}_\cM(X,\Q/\Z(j))$
is a finite group, all groups in the above
exact sequence are finite groups.
Then, $H^{i-1}_\cM(X,\Z(j))\otimes_\Z \Q/\Z$
must be zero, 
since it is finite and divisible,
hence we have a canonical isomorphism
$H^{i-1}_\cM(X,\Q/\Z(j)) \to H^{i}_\cM(X,\Z(j))_\tors$. 
The finiteness of $H^{i}_\cM(X,\Z(j))_\tors$ implies that
the divisible group $H^{i}_\cM(X,\Z(j))_\divi$ is uniquely divisible
and 
the canonical homomorphism
$$
H^{i}_\cM(X,\Z(j))^\cotors \to 
\varprojlim_{m} H^{i}_\cM(X,\Z(j))/m
$$
is injective. 
The latter group
$\varprojlim_{m} H^{i}_\cM(X,\Z(j))/m$ is
canonically embedded in the finite group
$\varprojlim_{m} H^{i}_\cM(X,\Z/m(j))$,
hence we conclude that $H^{i}_\cM(X,\Z(j))^\cotors$ 
is finite. This proves the claim.
\end{proof}

\begin{lem}\label{lem:MSGL_1}
Let $X$ be a smooth projective surface over $\F_q$.
Let $j$ be an integer. 
Then, $H^i_\cM(X,\Q/\Z(j))$
and $\varprojlim_m H^i_\cM(X,\Z/m(j))$
are finite if $i \neq 2j$ or $j \ge 3$.
\end{lem}
\begin{proof}
The claim for $j \le 1$ is clear.
Suppose that $j=2$. Then, the claim for $i \ge 5$ holds 
since the groups are zero by dimension reasons. 
If $p \nmid m$, from the theorem of Geisser and Levine
\cite[Corollary 1.2, p.~56]{Ge-Le2} 
(see also [Corollary 1.4]) 
and the theorem of Merkurjev and Suslin~\cite[(11.5), Theorem, p.~328]{Me-Su1}, 
it follows that the cycle class map 
$H^i_\cM(X,\Z/m(2)) \to H^i_\etale(X,\Z/m(2))$
is an isomorphism for $i \le 2$
and is an embedding for $i =3$.
By \cite[Th\'{e}or\`{e}me 2, p.~780]{CSS}
and the exact sequence \cite[2.1 (29) p.~781]{CSS},
for $i \le 3$, the group $\varinjlim_{m,\, p\nmid m} 
H^{i}_{\etale}(X,\Z/m(2))$ 
and the group $\varprojlim_{m,\, p\nmid m} 
H^{i}_{\etale}(X,\Z/m(2))$ are finite.

%
Let $W_n \Omega^\bullet_{X,\log}$ denote the 
logarithmic de Rham-Witt sheaf (\cf \cite[I, 5.7, p.~596]{Illusie}),
which was introduced by Milne in \cite{Milne1}.
There is an  
isomorphism $H^i_\cM(X,\Z/p^n(2)) 
\cong H^{i-2}_\Zar(X,W_n \Omega^2_{X,\log})$
(\cf \cite[Theorem 8.4, p.~491]{Ge-Le1}). 
In particular,  we have $H^i_\cM(X,\Q_p/\Z_p(2)) =0$ for $i \le 1$.

By \cite[\S2, Th\'{e}or\`{e}me 3, p.~782]{CSS},
$\varinjlim_n H^i_\etale(X,W_n \Omega^2_{X,\log})$
is a finite group for $i=0,1$. 
By \cite[Pas $n^o$ 1, p.~783]{CSS}, the projective system
$\{H^i_\etale(X,W_n \Omega^2_{X,\log})\}_n$ satisfies the Mittag-Leffler condition.
Using the same argument used in \cite[Pas $n^o$ 4, p.~784]{CSS},
we obtain an exact sequence
\[
\begin{array}{ll}
0 \to \varprojlim_n H^i_\etale(X, W_n\Omega^2_{X, \log}) \otimes_\Z \Q_p/\Z_p
\to \varinjlim_n H^i_\etale(X, W_n\Omega_{X,\log}^2) \\
\to \varprojlim_n H^{i+1}_\etale(X, W_n\Omega_{X,\log}^2)_{\tors}
\to 0.
\end{array}
\]
We then see that
$\varprojlim_n H^i_\etale(X,W_n \Omega^2_{X,\log})$
is also finite for $i=0,1$ and is isomorphic to
$\varinjlim_n H^{i-1}_\etale(X,W_n \Omega^2_{X,\log})$.
Since the homomorphism 
$$
H^i_\Zar(X,W_n \Omega^2_{X,\log})
\to H^i_\etale(X,W_n \Omega^2_{X,\log})
$$ 
induced by the change of topology 
$\vep : X_\etale \to X_\Zar$, 
is an isomorphism for $i=0$ and is injective for $i=1$,
we observe that $\varprojlim_n H^2_\cM(X,\Z/p^n (2))$ is
zero and that both $H^2_\cM(X,\Q_p/\Z_p(2))$ and 
$\varprojlim_n H^3_\cM(X,\Z/p^n (2))$
are finite groups. This proves the claim for
$j=2$.

Suppose $j \ge 3$. 
When $i \ge 2j$, since $X$ is a surface,
the cycle complex defining the higher Chow groups 
are zero in the negative degrees,
hence the claim holds true.
Since $j \ge 3$, we have
$H^i_\cM(X,\Z/p^n(j)) 
\cong H^{i-j}_\Zar(X,W_n \Omega^j_{X,\log}) =0$.
Using the theorem of Rost and Voevodsky (Theorem~\ref{conj:BK}),  
by the theorem of Geisser-Levine \cite[Theorem 1.1, p.~56]{Ge-Le2},
the group $H^i_\cM(X,\Z/m(j))$ is isomorphic to
$H^i_\Zar(X, \tau_{\le j} R \vep_* \Z/m(j))$
if $p \nmid m$. Since any affine surface 
over $\F_q$ has $\ell$-cohomological dimension three
for any $\ell \neq p$, it follows that
$H^i_\Zar(X, \tau_{\le j} R \vep_* \Z/m(j))
\cong H^i_\etale(X, \Z/m(j))$
for all $i$. 
Hence, by \cite[Th\'{e}or\`{e}me 2, p.~780]{CSS}
and the exact sequence \cite[2.1 (29) p.~781]{CSS},
for $i \le 2j-1$,
the groups $H^i_\cM(X,\Q/\Z(j))$
and $\varprojlim_{m} H^i_\cM(X,\Z/m(j))$ 
are finite. 
This proves the claim for $j \ge 3$.
\end{proof}

\begin{lem}\label{lem:Quillen_Harder}
Let $Y$ be a scheme of dimension $d \le 1$ 
of finite type over $\Spec\, \F_q$.
Then $H^i_\cM(Y,\Z(j))$ is a torsion group
unless $0 \le j \le d$ and $j \le i \le 2j$.
\end{lem}
\begin{proof}
By taking a smooth affine open subscheme of $Y_\red$ 
whose complement is of dimension zero, 
and using the localization sequence of motivic cohomology, 
we are reduced to the case in which $Y$ is connected, affine,
and smooth over $\Spec\, \F_q$. When $d=0$ (\resp $d=1$), the
claim follows from the results of Quillen \cite[Theorem 8(i), p.~583]{Quillen} 
(\resp Harder \cite[Korollar 3.2.3, p.~175]{Harder} 
(see \cite[Theorem 0.5, p.~70]{Grayson} for the correct 
interpretation of his results)) 
on the structure of the $K$-groups of $Y$, combined 
with the Riemann-Roch theorem for higher Chow 
groups \cite[Theorem 9.1, p.~296]{Bloch}.
\end{proof}

\begin{lem}\label{lem:abelgp}
Let $\varphi :M \to M'$ be a homomorphism 
of abelian groups such that
$\Ker\,\varphi$ is finite and 
$(\Coker\, \varphi)_\divi=0$. 
If $M_\divi$ or
$M'_\divi$ is uniquely divisible,
then $\varphi$
induces 
an isomorphism $\varphi_\divi: M_\divi \xto{\cong} M'_\divi$.
\end{lem}
\begin{proof}
First, we show that $\varphi_\divi$ is surjective.
Since $(\Coker \,\varphi)_\divi=0$,
for any $a \in M'_\divi$, we have 
$\varphi(a)^{-1}\neq \emptyset$.
Let $N$ be a positive integer
such that $N(\Ker\, \varphi)=0$.
Set $M_0=\varphi^{-1}(M'_\divi)$ and 
$M_1=N M_0$.
Since $\varphi$ induces 
a surjection $M_1 \to M'_\divi$,
it suffices to show that $M_1$ is divisible.
Let $x \in M_1$ and $n \in \Z_{\ge 1}$.
Take $y \in M_0$ such that $x=Ny$;
$z \in M'_\divi$ such that $nz=\varphi(y)$;
$y' \in M_0$ such that $\varphi(y')=z$.
Set $x' =Ny' \in M_1$.
Then $\varphi(y-ny')=0$ implies 
$x-nx'=Ny-nNy'=0$.   
This shows that $x$ is divisible,
hence proving the surjectivity.

Next, let us prove the injectivity.
Suppose $M_\divi$ is uniquely divisible.
Then, since $\Ker\, \varphi$ is torsion,
$\Ker\, \varphi \cap M_\divi=0$,
and $\varphi_\divi$ is injective.
Suppose $M'_\divi$ is uniquely divisible.
Then, the torsion subgroup of $M_\divi$
is contained in $\Ker\,\varphi$.
Since a nonzero torsion divisible group is infinite,
$M_\divi$ must be uniquely divisible, hence 
we are reduced to the case above and $\varphi_\divi$
is injective.
\end{proof}

\begin{proof}[Proof of Theorem~\ref{thm:motfin}]
Without loss of generality, we may assume that
$X$ is connected. We first prove the claims 
assuming $X$ is projective. It is clear that
the group $H^i_\cM(X,\Z(j))$ is zero for $i \ge
\min(j+3,2j+1)$. 
It follows from \cite[p.~787, Proposition 4]{CSS} 
that the degree map
$H^4_\cM(X,\Z(2))=\CH_0(X) \to \Z$ 
has finite kernel and cokernel,
proving the claim for $i \ge \min(j+3,2j)$.
Next, we fix $j \ge 2$. 
For $i \le 2j-1$, the group $H^i_\cM(X,\Z(j))$ is
finite modulo a uniquely divisible subgroup
by Lemmas~\ref{lem:toritori} and \ref{lem:MSGL_1}.
The claim on the identification of $H^i_\cM(X,\Z(j))_\tors$
with the \'{e}tale cohomology follows
immediately from the argument in the proof of
Lemma~\ref{lem:MSGL_1} except for the
$p$-primary part of $H^3_\cM(X,\Z(2))$, which 
follows from Proposition~\ref{prop:dRW_main}.

To finish the proof,
it remains to prove that $H^i_\cM(X,\Z(j))_\divi$
is zero for $j\ge 3$ and $i=j+1,j+2$. It suffices to prove
that $H^i_\cM(X,\Z(j))$ is a torsion group for $j \ge 3$ and
$i \ge j+1$. Consider the limit
$$
\varinjlim_Y H^{i-2}_{\cM}(Y,\Z(j-1))
\to H^{i}_{\cM}(X,\Z(j)) \to
\varinjlim_Y H^{i}_{\cM}(X \setminus Y,\Z(j))
$$
of the localization sequence in which $Y$ runs over the
reduced closed subschemes of $X$ of pure codimension one.
$H^{i-2}_{\cM}(Y,\Z(j-1))$ is torsion
by Lemma~\ref{lem:Quillen_Harder} and we have
$\varinjlim_Y H^{i-2}_{\cM}(X \setminus Y,\Z(j-1))=0$
for dimension reasons, hence the claim follows.
This completes the proof in the case where $X$ is projective.

For a general connected surface $X$, 
take an embedding $X \inj X'$ of $X$ into
a smooth projective surface $X'$ over $\F_q$
such that $Y=X' \setminus X$ is of pure
codimension one 
in $X'$. We can show that such an $X'$ exists 
via \cite{Nagata} and a resolution of 
singularities \cite[p.~111]{Abhyankar}, \cite[p.~151]{Lipman}.
Then, the claims, except for that of the 
identification of $H^i_\cM(X,\Z(j))_\tors$
with the \'{e}tale cohomology,
easily follow from Lemma~\ref{lem:abelgp} and 
the localization sequence
$$
\begin{small}
\cdots \to H^{i-2}_{\cM}(Y,\Z(j-1))
\to H^{i}_{\cM}(X',\Z(j)) \to
H^{i}_{\cM}(X,\Z(j))
\to \cdots.
\end{small}
$$
The claim on the identification of $H^i_\cM(X,\Z(j))_\tors$
with the \'{e}tale cohomology can be obtained 
using a similar approach to that used in the proof
of Lemma~\ref{lem:MSGL_1}. This completes the proof.
\end{proof}

\subsection{A criterion for the 
finiteness of $H^3_\cM(X,\Z(2))_\tors$}\label{subsec:criterion(3,2)}

\begin{prop}\label{prop:red_loc_seq}
Let $X$ be a smooth surface over $\F_q$.
Let $X \inj X'$ be an open immersion such that 
$X'$ is smooth projective over $\F_q$ 
and $Y=X' \setminus X$ is of pure
codimension one in $X'$. 
Then, the following conditions are equivalent.
\begin{enumerate}
\item $H^3_\cM(X,\Z(2))$ is finitely generated
modulo a uniquely divisible subgroup.
\item $H^3_\cM(X,\Z(2))_\tors$ is finite.
\item The pull-back map
$H^3_\cM(X',\Z(2)) \to H^3_\cM(X,\Z(2))$ 
induces an isomorphism 
$H^3_\cM(X',\Z(2))_\divi \xto{\cong} H^3_\cM(X,\Z(2))_\divi$.
\item The kernel of the pull-back map
$H^3_\cM(X',\Z(2)) \to H^3_\cM(X,\Z(2))$ is finite.
\item The cokernel of the boundary map 
$\partial:H^2_\cM(X,\Z(2)) \to H^1_\cM(Y,\Z(1))$ is finite.
\end{enumerate}
Further, if the above equivalent conditions are
satisfied, then the torsion group
$H^3_\cM(X,\Z(2))_\tors$ is isomorphic to
the direct sum of a finite group of $p$-power order and the group
$\bigoplus_{\ell \neq p} 
H^2_\etale(X,\Q_\ell/\Z_\ell(2))^\cotors$, and
the localization sequence induces a long exact sequence
\begin{equation}\label{eqn:locseq_mod_div}
\cdots \to H^{i-2}_\cM(Y,\Z(1))
\to H^i_\cM(X',\Z(2))^\cotors 
\to H^i_\cM(X,\Z(2))^\cotors \to
\cdots 
\end{equation}
of finitely generated abelian groups.
\end{prop}

\begin{proof}
Condition (1) clearly implies condition (2).
The localization sequence shows that
conditions (4) and (5) are equivalent
and that condition (3) implies condition (1).
From the localization sequence and Lemma~\ref{lem:abelgp},
condition (4) implies condition (3).

We claim that condition (2) implies condition (4).
Assume condition (2) and suppose that 
condition (4) is not satisfied. 
We set $M = \Ker[H^3_\cM(X',\Z(2)) \to H^3_\cM(X,\Z(2))]$.
The localization sequence
shows that 
$M$ is finitely generated. 
By assumption, $M$ is not torsion.
Since $H^3_\cM(X',\Z(2))$ is finite modulo a uniquely
divisible subgroup, the intersection
$H^3_\cM(X',\Z(2))_\divi \cap M$ is a nontrivial 
free abelian group of finite rank, 
hence
$H^3_\cM(X,\Z(2))$ contains a 
group isomorphic to 
$H^3_\cM(X',\Z(2))_\divi / (H^3_\cM(X',\Z(2))_\divi \cap M)$,
which contradicts condition (2), hence
condition (2) implies condition (4). This 
completes the proof of the equivalence of conditions (1)-(5).

Suppose that conditions (1)-(5) are satisfied.
The localization sequence shows that the 
kernel (\resp the cokernel) of the pull-back
$H^i_\cM(X',\Z(2)) \to H^i_\cM(X,\Z(2))$ is a 
torsion group (\resp has no nontrivial divisible subgroup)
for any $i \in \Z$, hence by Lemma~\ref{lem:abelgp},
$H^i_\cM(X,\Z(2))_\divi$ is uniquely divisible and
sequence (\ref{eqn:locseq_mod_div}) is exact.
Condition (2) and exact sequence 
(\ref{eqn:shortexact_motcoh}) 
for $(i,j)=(3,2)$ yield an isomorphism
$H^{3}_\cM(X,\Z(2))_\tors \cong 
H^{2}_\cM(X,\Q/\Z(2))^\cotors$.
Then, the claim on the structure of
$H^{3}_\cM(X,\Z(2))_\tors$ follows from
the theorem of Geisser and Levine
\cite[Corollary 1.2, p.~56. See also Corollary 1.4]{Ge-Le2} 
and the theorem of Merkurjev and Suslin~\cite[(11.5), Theorem, p.~328]{Me-Su1}.
This completes the proof.
\end{proof}

Let $X$ be a smooth projective surface over $\F_q$.
Suppose that $X$ admits a flat, surjective, and generically
smooth morphism $f:X \to C$ to a connected 
smooth projective curve $C$ over $\F_q$. 
For each point $\wp \in C$, let $X_\wp= X \times_C \wp$ denote the
fiber of $f$ at $\wp$.

\begin{cor}\label{cor:relative_cv}
Let the notations be as above.
Let $\eta \in C$ denote the generic point.
Suppose that the cokernel of the homomorphism
$\partial : H^2_\cM(X_\eta,\Z(2)) \to 
\bigoplus_{\wp \in C_0} H^1_\cM(X_\wp,\Z(1))$,
which is the inductive limit of the 
boundary maps of the localization sequences,
is a torsion group. Then,
the group $H^i_\cM(X_\eta,\Z(2))_\divi$ is
uniquely divisible for all $i \in \Z$ and
the inductive limit of localization sequences 
induces a long exact sequence
\begin{small}
$$
\cdots \to \bigoplus_{\wp \in C_0} 
H^{i-2}_\cM(X_\wp,\Z(1))
\to H^i_\cM(X,\Z(2))^\cotors 
\to H^i_\cM(X_\eta, \Z(2))^\cotors \to \cdots.
$$
\end{small}
\end{cor}
\begin{proof}
Since $\bigoplus_{\wp \in C_0} H^i_\cM(X_\wp,\Z(1))$
has no nontrivial divisible subgroup for all
$i \in \Z$ and is torsion for $i \neq 1$ by Lemma~\ref{lem:Quillen_Harder}, 
the claim follows from Lemma~\ref{lem:abelgp}.
\end{proof}

\section{The compatibility of Chern characters and the localization sequence}
\label{sec:Compatibility}
\newcommand{\fre}{\mathfrak{e}}
\newcommand{\cV}{\mathcal{V}}
\newcommand{\naive}{\mathrm{naif}}
\newcommand{\Tot}{\mathrm{Tot}}
\newcommand{\cU}{\mathcal{U}}
We first refer the reader to the first half of Section~\ref{sec:intro Chern}
for a general overview of the contents of this section.

The aim of this section is to prove Lemma~\ref{lem:chern_comm}.
Only Lemma~\ref{lem:chern_comm} 
and Remark~\ref{rmk:chern_comm}
will be used beyond this section.

%
In this paper, for the Chern class, 
we use the Chern class
for motivic cohomology of Levine (see below)
combined with the comparison 
isomorphism between 
Levine's motivic cohomology 
groups and higher Chow groups.
We also use the Riemann-Roch theorem 
for such Chern classes
by checking 
the compatibility of the localization 
sequences with the comparison isomorphisms
(i.e., Lemmas~\ref{lem:chern_comm2} and~~\ref{lem:chern_comm3}).
%

\subsection{Main statement}
Given an essentially smooth scheme $X$
 (see Section \ref{subsec:Motcoh}
for the definition of ``essentially smooth")
over $\Spec\, \F_q$ and integers $i,j\ge 0$,
let $c_{i,j}:K_i (X) \to H^{2j-i}_\cM(X,\Z(j))$
be the Chern class map.
Several approaches to constructing the map
$c_{i,j}$ 
have been proposed, including \cite[p.\ 293]{Bloch}, 
\cite[Part I, Chapter III, 1.4.8.\ Examples. (i), p.~123]{Levine}, 
\cite[Definition 5, p.~315]{Pushin};
all of these approaches are based on Gillet's work
\cite[p.~228--229, Definition 2.22]{Gillet}.
In this paper, we adopt the definition of
Levine \cite[Part I, Chapter III, 1.4.8.\ Examples. (i), p.~123]{Levine}
in which $c_{i,j}$ is denoted by $c_X^{j,2j-i}$,
given in \cite[Part I, Chapter I, 2.2.7, p.~21]{Levine}.
The definition of the target $H^{2j-i}(X,\Z(j))= H^{2j-i}_X(X,\Z(j))$ 
of $c_X^{j,2j-i}$, which we denote by $H^{2j-i}_\cL(X,\Z(j))$,
is different from the definition of the group 
$H^{2j-i}_\cM(X,\Z(j))$;
however, by combining (ii) and (iii) of 
\cite[Part I, Chapter II, 3.6.6.\ Theorem, p.~105]{Levine}, 
we obtain a canonical isomorphism 
\begin{equation}\label{eqn:ML}
\beta^i_j : H^i_\cM(X,\Z(j)) \xto{\cong} H^i_\cL(X,\Z(j)),
\end{equation}
which is compatible with the product structures.
The precise definition of our Chern class map 
$c_{i,j}$ is the composition $c_{i,j} = (\beta^i_j)^{-1}
\circ c_X^{j,2j-i}$.

The map $c_{i,j}$ is a group homomorphism if
$i \ge 1$ or $(i,j)=(0,1)$. 
Let
$\chern_{0,0}:K_0(X)\to H^0_\cM(X,\Z(0)) 
\cong H^0_\Zar(X,\Z)$ denote the homomorphism
that sends the class of locally free $\cO_X$-module
$\cF$ to the rank of $\cF$. For $i\ge 1$ and
$a \in K_i(X)$, we put formally 
$\chern_{i,0}(a)=0$.
\begin{lem}\label{lem:chern_comm}
Let $X$ be a scheme 
which is a localization of
a smooth quasi-projective scheme over $\Spec\, \F_q$. 
Let $Y \subset X$
be a closed subscheme of pure codimension $d$,
which is essentially smooth over $\Spec\, \F_q$. 
Then for $i,j \ge 1$ or $(i,j)=(0,1)$, the diagram
\begin{equation}\label{eqn:chern_comm}
\begin{CD}
K_{i}(Y) @>{\alpha_{i,j}}>> H^{2j-i-2d}_\cM(Y,\Z(j-d)) \\
@VVV @VVV \\
K_{i}(X) @>{c_{i,j}}>> H^{2j-i}_\cM(X,\Z(j)) \\
@VVV @VVV \\
K_{i}(X\setminus Y) @>{c_{i,j}}>> H^{2j-i}_\cM(X\setminus Y,\Z(j)) \\
@VVV @VVV \\
K_{i-1}(Y) @>{\alpha_{i-1,j}}>> H^{2j-i-2d+1}_\cM(Y,\Z(j-d))
\end{CD}
\end{equation}
is commutative. 
Here, the homomorphism $\alpha_{i,j}$
is defined as follows: 
for $a \in K_i(Y)$, the element $\alpha_{i,j}(a)$ equals
$$
G_{d,j-d} (\chern_{i,0}(a), c_{i,1}(a),\ldots, c_{i,j-d}(a);
c_{0,1}(\cN), \ldots, c_{0,j-d}(\cN)),
$$
where $G_{d,j-d}$ is the universal polynomial
in \cite[Expos\'{e} 0, Appendice, Proposition 1.5, p.~37]{SGA6}, 
$\cN$
is the conormal sheaf of $Y$ in $X$, 
and the left (\resp the right) vertical sequence
is the localization sequence of $K$-theory
(\resp of higher Chow groups established in 
\cite[Corollary (0.2), p.~357]{Bloch2}).
\end{lem}
\begin{proof}
We may assume that $X$ is quasi-projective and 
smooth over $\Spec\, \F_q$.
It follows from \cite[Part I, Chapter III, 1.5.2, p.~130]{Levine} 
and the Riemann-Roch theorem without denominators 
\cite[Part I, Chapter III, 3.4.7.\ Theorem, p.~174]{Levine} that
diagram (\ref{eqn:chern_comm}) is commutative if
we replace the right vertical sequence
by Gysin sequence 
\begin{equation}\label{eqn:Gysin}
\begin{array}{c}
H^{2j-i-2d}_\cL(Y,\Z(j-d))
\to H^{2j-i}_\cL(X,\Z(j)) \\
\to H^{2j-i}_\cL(X\setminus Y,\Z(j))
\to H^{2j-i-2d+1}_\cL(Y,\Z(j-d))
\end{array}
\end{equation}
in \cite[Part I, Chapter III, 2.1, p.~132]{Levine}. It suffices to show that
Gysin sequence (\ref{eqn:Gysin}) is identified with the localization
sequence of higher Chow groups.
Here, we use the notations from \cite[Part I, Chapter I, II]{Levine}. 
Let $S=\Spec\, \F_q$ and 
$\cV$ denote 
the category of schemes 
that is essentially smooth over 
$\Spec\, \F_q$. 
Let $\cA_\mot(\cV)$ be the DG category
defined in \cite[Part I, Chapter I, 1.4.10 Definition, p.~15]{Levine}.
For an object $Z$ in $\cV$ and a morphism
$f : Z' \to Z$ in $\cV$ that admits a smooth section,
and for $j \in \Z$, we have the object 
$\Z_{Z}(j)_f$ in $\cA_\mot(\cV)$. When $f=\id_Z$
is the identity, we abbreviate $\Z_Z(j)_{\id_Z}$
by $\Z_Z(j)$.
For a closed subset $W \subset Z$, let
$\Z_{Z,W}(j)$ be the object introduced in
\cite[Part I, Chapter I, (2.1.3.1), p.~17]{Levine}; 
this is an object in the DG category 
$\mathbf{C}^b_\mot(\cV)$ of
bounded complexes in $\cA_\mot(\cV)$.
The object $\Z_{Z}(j)_f$ belongs to the
full subcategory $\cA_\mot(\cV)^*$ of $\cA_\mot(\cV)$
introduced in \cite[Part I, Chapter I, 3.1.5, p.~38]{Levine},
and the object $\Z_{Z,W}(j)$ belongs to 
the DG category $\mathbf{C}^b_\mot(\cV)^*$ 
of bounded complexes in $\cA_\mot(\cV)^*$.
For $i \in \Z$, we set $H^i_{\cL,W}(Z,\Z(j)) 
= \Hom_{\mathbf{D}^b_\mot(\cV)}(1,\Z_{Z,W}(j)[i])$
where $1$ denotes the object 
$\Z_{\Spec\, \F_q}(0)$ and $\mathbf{D}^b_\mot(\cV)$ denotes 
the category
introduced in \cite[Part I, Chapter I, 2.1.4 Definition, p.~17--18]{Levine}.

Let $X$, $Y$ be as in the statement of 
Lemma~\ref{lem:chern_comm}.
Let $\mathbf{K}^b_\mot(\cV)$ be the homotopy category
of $\mathbf{C}^b_\mot(\cV)$. Then, we have a distinguished
triangle
$$
\Z_{X,Y}(j) \to \Z_{X}(j)
\to \Z_{X\setminus Y}(j)
\xto{+1}
$$
in $\mathbf{K}^b_\mot(\cV)$. This distinguished
triangle yields a long exact sequence
\begin{equation}\label{eqn:long1}
\begin{array}{c}
\cdots \to H^i_{\cL,Y}(X,\Z(j)) \to 
H^i_\cL(X,\Z(j))  \\
\to H^i_\cL(X\setminus Y,\Z(j)) \to 
H^{i+1}_{\cL,Y}(X,\Z(j)) \to \cdots.
\end{array}
\end{equation}
In \cite[Part I, Chapter III, (2.1.2.2), p.~132]{Levine},
Levine constructs an isomorphism 
$\iota_* : \Z_Y(j-d)[-2d]
\to \Z_{X,Y}(j)$ in $\mathbf{D}^b_\mot(\cV)$.
This isomorphism induces an isomorphism
$\iota_* : H^{i-2d}_{\cL}(Y,\Z(j-d)) \xto{\cong}
H^i_{\cL,Y}(X,\Z(j))$. 
This latter isomorphism, together with long exact
sequence (\ref{eqn:long1}) yields the Gysin sequence
(\ref{eqn:Gysin}).

We set $z^j_Y(X, -\bullet)
= \Cone(z^j(X, -\bullet)\to 
z^j(X\setminus Y, -\bullet))[-1]$
and define cohomology with support $H^i_{\cM,Y}(X,\Z(j))
= H^{i-2j}(z^j_Y(X, -\bullet))$.
The distinguished triangle
$$
z^j_Y(X, -\bullet) \to
z^j(X, -\bullet) \to 
z^j(X\setminus Y, -\bullet)
\xto{+1}
$$
in the derived category of abelian groups
induces a long exact sequence
\begin{equation}\label{eqn:long2}
\begin{array}{c}
\cdots \to H^i_{\cM,Y}(X,\Z(j)) \to 
H^i_\cM(X,\Z(j))  \\
\to H^i_\cM(X\setminus Y,\Z(j)) \to 
H^{i+1}_{\cM,Y}(X,\Z(j)) \to \cdots.
\end{array}
\end{equation}
The push-forward map
$z^{j-d}(Y, -\bullet) \to z^{j}(X,-\bullet)$ 
of cycles gives a homomorphism
$z^{j-d}(Y, -\bullet) \to
z^j_Y(X, -\bullet)$
of complexes of abelian groups, which is
known to be a quasi-isomorphism 
by \cite[Theorem (0.1), p.~537]{Bloch2},
hence it induces an isomorphism
$\iota_* :H^{2j-i-2d}_\cM(Y,\Z(j-d))
\xto{\cong} H^{2j-i}_{\cM,Y}(X, \Z(j))$.
Then, the claim follows from Lemmas~\ref{lem:chern_comm2} 
and \ref{lem:chern_comm3} below.
\end{proof}
\subsection{Compatibility of localization sequences}
\begin{lem}\label{lem:chern_comm2}
Let $X$ and $Y$ be as in Lemma \ref{lem:chern_comm}.
For each $i,j\in \Z$, there exists a canonical isomorphism 
$$
\beta^i_{Y,j}: H^{i}_{\cM,Y}(X,\Z(j))
\xto{\cong} H^{i}_{\cL,Y}(X,\Z(j))
$$
such that the long exact sequence (\ref{eqn:long1})
is identified  with 
the long exact sequence (\ref{eqn:long2})
via this isomorphism and the isomorphism (\ref{eqn:ML}).
\end{lem}
\begin{proof}
First, let us recall the relation between the group 
$$H^i_\cM(X,\Z(j))
\xrightarrow[\beta^i_j]{\cong} H^i_\cL(X,\Z(j))$$ 
and the naive higher Chow group introduced by 
Levine \cite[2.3.1. Definition, p.70]{Levine}.
For an object $\Gamma$ of $\mathbf{C}^b_\mot(\cV)$, the naive
higher Chow group $\CH_\naif(\Gamma,p)$ is by definition
the cohomology group $H^{-p}(\cZ_\mot(\Gamma,*))$.
Here, $\cZ_\mot(\ , *)$ is as in
\cite[Part I, Chapter II, 2.2.4.\ Definition, p.~68]{Levine},
which is a DG functor
from the category $\mathbf{C}^b_\mot(\cV)^*$ to the category of
complexes of abelian groups bounded from below.
Since $X$ is a localization of a smooth quasi-projective scheme over $k$,
it follows from \cite[2.4.1., p.71]{Levine} that we have a natural isomorphism
$\CH_\naif(\Z_X(j)[2j],2j-i) \cong H^i_\cL(X,\Z(j))$.

Observe that the functor $\cZ_\mot$ 
\cite[Part I, Chapter I, (3.3.1.2), p.~40]{Levine} 
from the category $\mathbf{C}^b_\mot(\cV)$ to the category of
bounded complexes of abelian groups is compatible with taking cones,
hence the DG functor $\cZ_\mot(\ , *)$ is also compatible 
with taking cones. 
Since $\cZ_\mot(\Z_X(j)_{\id_X},*)$ is canonically isomorphic
to the cycle complex $z^j(X,-\bullet)$, the complex
$\cZ_\mot(\Z_{X,Y}(j)_{\id_X},*)$ is canonically isomorphic
to $z^j_Y(X,-\bullet)$.
For an object $\Gamma$
in $\mathbf{C}^b_\mot(\cV)^*$, let $\mathcal{CH}(\Gamma,p)$
be the higher Chow group defined in 
\cite[Part I, Chapter II, 2.5.2.\ Definition, p.~76]{Levine}.
From the definition of $\mathcal{CH}(\Gamma,p)$, we obtain
canonical homomorphisms $H^{2j-i}_\cM(X, \Z(j)) \to 
\mathcal{CH}(\Z_{X}(j),i)$,
$H^{2j-i}_\cM(X\setminus Y, \Z(j)) \to 
\mathcal{CH}(\Z_{X\setminus Y}(j),i)$, and
$H^{2j-i}_{\cM,Y}(X, \Z(j)) \to 
\mathcal{CH}(\Z_{X,Y}(j),i)$
such that the diagram
\begin{equation}\label{CD1}
\begin{CD}
H^{2j-i}_{\cM,Y}(X, \Z(j)) 
@>>> \mathcal{CH}(\Z_{X,Y}(j),i) \\
@VVV @VVV \\
H^{2j-i}_\cM(X, \Z(j)) @>>> 
\mathcal{CH}(\Z_{X}(j),i) \\
@VVV @VVV \\
H^{2j-i}_\cM(X\setminus Y, \Z(j)) @>>> 
\mathcal{CH}(\Z_{X\setminus Y}(j),i) \\
@VVV @VVV \\
H^{2j-i+1}_{\cM,Y}(X, \Z(j)) 
@>>> \mathcal{CH}(\Z_{X,Y}(j),i-1) 
\end{CD}
\end{equation}
is commutative.

Recall the definition of the cycle class map
$\cl(\Gamma) :\mathcal{CH}(\Gamma)=\mathcal{CH}(\Gamma,0) \to 
\Hom_{\mathbf{D}^b_\mot(\cV)}(1,\Gamma)$
\cite[p.~76]{Levine} for an object $\Gamma \in
\mathbf{C}^b_\mot(\cV)^*$.
Also recall that $\mathcal{CH}(\Gamma)
= \varinjlim_{\Gamma \to \Gamma_{\wt{\cU}}}
H^0(\cZ_\mot(\Tot\, \Gamma_{\wt{\cU}},*))$
where $\Gamma \to \Gamma_{\wt{\cU}}$ 
runs over the hyper-resolutions of $\Gamma$
\cite[Part I, Chapter II, 1.4.1.\ Definition, p.~59]{Levine}
and $\Tot : \mathbf{C}^b(\mathbf{C}^b_\mot(\cV)^*)
\to \mathbf{C}^b_\mot(\cV)^*$ 
denotes the total complex functor
in \cite[Part I, Chapter II, 1.3.2, p.~58]{Levine}. 
The homomorphism
$\cl(\Gamma)$ is defined as the inductive limit of
the composition
$$
\begin{array}{ll}
\cl_\naive(\Tot\, \Gamma_{\wt{\cU}}):& 
H^0(\cZ_\mot(\Tot\, \Gamma_{\wt{\cU}},*))
\to \Hom_{\mathbf{D}^b_\mot(\cV)}(1,\Tot\, \Gamma_{\wt{\cU}})
\\
&\xleftarrow{\cong}
\Hom_{\mathbf{D}^b_\mot(\cV)}(1,\Gamma).
\end{array}
$$
For an object $\Gamma$ in $\mathbf{C}^b_\mot(\cV)^*$,
the homomorphism $\cl_\naive(\Gamma)$ is, by definition 
\cite[Part I, Chapter II, (2.3.6.1), p.~71]{Levine}, 
equal to the composition
$$
\begin{array}{rl}
H^0(\cZ_\mot(\Gamma,*))
& \xleftarrow{\cong} H^0(\cZ_\mot(\Sigma^N(\Gamma)[N])) \\
& \xleftarrow{\cong}
\Hom_{\mathbf{K}^b_\mot(\cV)}(\fre^{\otimes a} \otimes 1,
\Sigma^N(\Gamma)[N]) \\
& \to \Hom_{\mathbf{D}^b_\mot(\cV)}(\fre^{\otimes a} \otimes 1,
\Sigma^N(\Gamma)[N]) \\
& \xleftarrow{\cong}\Hom_{\mathbf{D}^b_\mot(\cV)}(1, \Gamma)
\end{array}
$$
for sufficiently large integers $N, a \ge 0$.
Here, $\Sigma^N$ is the suspension functor 
in \cite[Part I, Chapter II, 2.2.2.\ Definition, p.~68]{Levine},
and $\fre$ is the object in \cite[Part I, Chapter I, 1.4.5, p.~13]{Levine} 
that we regard as an object in $\cA_\mot(V)$. 
Let $\Gamma \to \Gamma'$ 
be a morphism in $\mathbf{C}^b_\mot(\cV)$ and set
$\Gamma'' = \Cone(\Gamma \to \Gamma')[-1]$.
Since the functor $\Sigma^N$
is compatible with taking cones, the diagram
\begin{equation}\label{CD2}
\begin{CD}
\mathcal{CH}(\Gamma'',i)
@>{\cl(\Gamma''[-i])}>> 
\Hom_{\mathbf{D}^b_\mot(\cV)}(1,\Gamma''[-i]) \\
@VVV @VVV \\
\mathcal{CH}(\Gamma,i)
@>{\cl(\Gamma[-i])}>> 
\Hom_{\mathbf{D}^b_\mot(\cV)}(1, \Gamma[-i]) \\
@VVV @VVV \\
\mathcal{CH}(\Gamma',i)
@>{\cl(\Gamma'[-i])}>> 
\Hom_{\mathbf{D}^b_\mot(\cV)}(1, \Gamma'[-i]) \\
@VVV @VVV \\
\mathcal{CH}(\Gamma'',i-1)
@>{\cl(\Gamma''[-i+1])}>> 
\Hom_{\mathbf{D}^b_\mot(\cV)}(1,\Gamma''[-i+1]) 
\end{CD}
\end{equation}
is commutative.

The homomorphism $\beta^i_j:H^i_\cM(X, \Z(j))\to
H^{2j-i}_\cL(X,\Z(j))$ is, by definition, equal to
the composition
\[
\begin{array}{ll}
H^i_\cM(X, \Z(j)) 
\to \mathcal{CH}(\Z_X(j),2j-i) 
\\
\xto{\cl(\Z_X(j)[i-2j])}
\Hom_{\mathbf{D}^b_\mot(\cV)}(1,\Z_X(j)[i])
= H^{i}_\cL(X,\Z(j)).
\end{array}
\]
We define 
$\beta^i_{Y,j}: H^i_{\cM,Y}(X, \Z(j))\to
H^{2j-i}_{\cL,Y}(X,\Z(j))$ to be the composition
\[
\begin{array}{l}
H^i_{\cM,Y} (X, \Z(j)) 
\to \mathcal{CH}(\Z_{X,Y}(j),2j-i) 
\\
\xto{\cl(\Z_{X,Y}(j)[i-2j])}
\Hom_{\mathbf{D}^b_\mot(\cV)}(1,\Z_{X,Y}(j)[i])
= H^{i}_{\cL,Y}(X,\Z(j)).
\end{array}
\] 
By (\ref{CD1}) and (\ref{CD2}), we have a commutative diagram
$$
\begin{CD}
H^{2j-i}_{\cM,Y}(X,\Z(j)) 
@>{\beta^{2j-i}_{Y,j}}>> H^{2j-i}_{\cL,Y}(X,\Z(j)) \\
@VVV @VVV \\
H^{2j-i}_\cM(X, \Z(j)) 
@>{\beta^{2j-i}_j}>{\cong}> 
H^{2j-i}_\cL(X,\Z(j)) \\
@VVV @VVV \\
H^{2j-i}_\cM(X\setminus Y, \Z(j)) 
@>{\beta^{2j-i}_j}>{\cong}> 
H^{2j-i}_\cL(X \setminus Y,\Z(j)) \\
@VVV @VVV \\
H^{2j-i+1}_{\cM,Y}(X, \Z(j)) 
@>{\beta^{2j-i+1}_j}>>  H^{2j-i+1}_{\cL,Y}(X,\Z(j)),
\end{CD}
$$
where the right vertical arrow is the long
exact sequence (\ref{eqn:long1}),
hence $\beta^{2j-i}_{Y,j}$ is an isomorphism
and the claim follows.
\end{proof}
\subsection{Compatibility of Gysin maps}
\begin{lem}\label{lem:chern_comm3}
The diagram 
$$
\begin{CD}
H^{2j-i-2d}_{\cM}(Y,\Z(j-d)) @>{\iota_*}>{\cong}> 
H^{2j-i}_{\cM.Y}(X, \Z(j)) \\
@V{\beta^{2j-i-2d}_{j-d}}V{\cong}V 
@V{\beta^{2j-i}_{Y,j}}V{\cong}V \\
H^{2j-i-2d}_{\cL}(Y,\Z(j-d)) @>{\iota_*}>{\cong}>
H^{2j-i}_{\cL,Y}(X,\Z(j))
\end{CD}
$$
is commutative.
\end{lem}
\begin{proof}
Recall the construction of 
the upper horizontal isomorphism $\iota_*$
in \cite[Part I, Chapter III, (2.1.2.2), p.~132]{Levine}.
Let $Z$ be the blow-up of $X 
\times_{\Spec\, \F_q} \A^1_{\F_q}$
along $Y \times_{\Spec\, \F_q} \{0\}$.
Let $W$ be the proper transform of
$Y \times_{\Spec\, \F_q} \A^1_{\F_q}$
to $Z$. Then, $W$ is canonically isomorphic to
$Y \times_{\Spec\, \F_q} \A^1_{\F_q}$.
Let $P$ be the inverse image of 
$Y \times_{\Spec\, \F_q} \{0\}$ under the map
$Z \to X \times_{\Spec\, \F_q} \A^1_{\F_q}$
and let $Q= P \times_Z W$. 
We set $Z' = Z \amalg (X \times_{\Spec\, \F_q} \{1\}) 
\amalg P$ and let $f:Z' \to Z$ denote the
canonical morphism. We then have canonical morphisms
$$
\Z_{P,Q}(j) \leftarrow \Z_{Z,W}(j)_f 
\to \Z_{X \times_{\Spec\, \F_q} \{1\},
Y \times_{\Spec\, \F_q} \{1\} }(j) 
= \Z_{X,Y}(j)
$$
in $\mathbf{C}^b_\mot(\cV)^*$,
which become isomorphisms in the category
$\mathbf{D}^b_\mot(\cV)$. 

Let $g:P \to Y\times_{\Spec\, \F_q} \{0\}
\cong Y$ be the canonical morphism.
The restriction of $g$ to $Q \subset P$ is
an isomorphism,
hence giving a section 
$s : Y \to P$ to $g$. 
The cycle class
$\cl_{P,Q}^d(Q) \in H^{2d}_{Q}(P,\Z(d))$
in \cite[Part I, Chapter I, (3.5.2.7), p.~48]{Levine} 
comes from the map
$[Q]_Q : \fre \otimes 1 \to \Z_{P,Q}(d)[2d]$
in $\mathbf{C}^b_\mot(\cV)$, as defined in
\cite[Part I, Chapter I, (2.1.3.3), p.~17]{Levine}.
We then have morphisms
\begin{equation}\label{eqn:DD}
\begin{array}{rl}
& \fre \otimes \Z_P(j-d)[-2d] \to
\Z_{P,Q}(d) \otimes
\Z_P(j-d)  \\
\xto{\gamma} &
\Z_{P \times_{\Spec\,\F_q} P, Q \times_{\Spec\,\F_q} P}(j) \\
\leftarrow &
\Z_{P \times_{\Spec\,\F_q} P, Q \times_{\Spec\,\F_q} P}(j)_{f'}
\xto{\Delta_P^*} \Z_{P,Q} (j)
\end{array}
\end{equation}
in $\mathbf{C}^b_\mot(\cV)$. 
Here, $\gamma$ is the map induced from the
external products, i.e., $\boxtimes_{P,P}: \Z_P(d) 
\otimes \Z_P(j-d) \to \Z_{P \times_{\Spec\,\F_q} P}(j)$
and $\boxtimes_{Q,P}: \Z_Q(d) 
\otimes \Z_P(j-d) \to \Z_{Q \times_{\Spec\,\F_q} P}(j)$;
further, $\Delta_P:P \to P \times_{\Spec\, \F_q} P$ 
denotes the diagonal embedding and $f'$ is the morphism 
$$
f'=\id_{P\times_{\Spec\, \F_q}P}
\amalg \Delta_P : P\times_{\Spec\, \F_q}P
\amalg P \to P\times_{\Spec\, \F_q}P.
$$
The morphisms in (\ref{eqn:DD}) above induce morphism
$\delta :\Z_P(j-d)[-2d] \to \Z_{P,Q} (j)$ in
$\mathbf{D}^b_\mot(\cV)$. 
The composite morphism
$\Z_{Y}(j-d)[-2d] \xto{q^*} \Z_{P}(j-d)[-2d] 
\xto{\delta} \Z_{P,Q} (j)$ induces a homomorphism
$\delta_* : H^{i-2d}_{\cL}(Y,\Z(j-d))
\to H^i_{\cL,Q}(P,\Z(j))$ for each $i \in \Z$.

From the construction of $\delta_*$, 
we observe that the diagram
$$
\begin{CD}
H^{i-2d}_{\cM}(Y,\Z(j-d))
@>{s_*}>> H^i_{\cM,Q}(P,\Z(j)) \\
@V{\beta^{i-2d}_{j-d}}VV @V{\beta^{i}_{Q,j}}VV \\
H^{i-2d}_{\cL}(Y,\Z(j-d))
@>{\delta_*}>> H^i_{\cL,Q}(P,\Z(j))
\end{CD}
$$
is commutative. Here, the upper horizontal arrow $s_*$ 
is the homomorphism that sends the class of a cycle $V \in 
z^{j-d}(Y,2j-i)$ to the class in
$H^i_{\cM,Q}(P,\Z(j))$ of the cycle
$s(V)$, which belongs to the kernel of 
$z^j(P, 2j-i) \to z^j(P\setminus Q,2j-i)$.
Next, the isomorphism $\iota_*:H^{2j-i-2d}_{\cL}(Y,\Z(j-d)) 
\xto{\cong} H^{2j-i}_{\cL,Y}(X,\Z(j))$
equals the composition 
$$
\begin{array}{c}
H^{2j-i-2d}_{\cL}(Y,\Z(j-d))
\xto{\delta_*} H^{2j-i}_{\cL,Q}(P,\Z(j))
\leftarrow H^{2j-i}_{\cL,W}(Z,\Z(j)) \\
\to H^{2j-i}_{\cL,Y\times_{\Spec\,\F_q} \{1\}}
(X\times_{\Spec\, \F_q}\{1\},\Z(j))
= H^{2j-i}_{\cL,Y}(X,\Z(j)).
\end{array}
$$
We can easily verify that 
the isomorphism $\iota_*:H^{2j-i-2d}_{\cM}(Y,\Z(j-d)) 
\xto{\cong} H^{2j-i}_{\cM,Y}(X,\Z(j))$
equals the composition 
$$
\begin{array}{c}
H^{2j-i-2d}_{\cM}(Y,\Z(j-d))
\xto{s_*} H^{2j-i}_{\cM,Q}(P,\Z(j))
\leftarrow H^{2j-i}_{\cM,W}(Z,\Z(j)) \\
\to H^{2j-i}_{\cM,Y\times_{\Spec\,\F_q} \{1\}}
(X\times_{\Spec\, \F_q}\{1\},\Z(j))
= H^{2j-i}_{\cM,Y}(X,\Z(j)).
\end{array}
$$
Given this, the claim follows.
\end{proof}

\begin{rmk}\label{rmk:chern_comm}
For $j=d$, we have 
$\alpha_{i,d}=(-1)^{d-1}(d-1)! \cdot \chern_{i,0}$.
For $i \ge 1$ and $j=d+1$, we have
$\alpha_{i,d+1}=(-1)^d d! \cdot c_{i,1}$.

Suppose that $d=1$ and $\cN \cong \cO_Y$. 
Then, we have $\alpha_{i,1}=\chern_{i,0}$ and
$\alpha_{i,j}(a) = (-1)^{j-1} 
Q_{j-1}(c_{i,1}(a),\ldots, c_{i,j-1}(a))$ 
for $i \ge 0$, $j\ge 2$, where
$Q_{j-1}$ denotes the $(j-1)$-st 
Newton polynomial, which 
expresses the $(j-1)$-st power sum polynomial 
in terms of the elementary symmetric polynomials.
In particular, we have $\alpha_{i,2}=-c_{i,1}$ for
$i \ge 0$ and $\alpha_{i,j}=- (j-1) c_{i,j-1}$
for $i \ge 1$, $j\ge 2$.
\end{rmk}

\section{Motivic Chern 
characters for singular curves over finite fields}\label{sec:MotChern}
Before embarking on this section, we refer 
to the latter half of Section~\ref{sec:intro Chern} 
for a general overview.   
Note that the output of this section consists of 
Lemma~\ref{prop:AppB_main} and Proposition~\ref{prop:Z}.

In this section, we construct Chern 
characters of low degrees for singular curves over finite fields 
with values in the higher Chow groups in an ad hoc manner.
Bloch defines Chern characters with values in the higher Chow groups
tensored with $\Q$ in \cite[(7.4), p.~294]{Bloch}.  We restrict ourselves
to one-dimensional varieties over finite fields, but the target group
lies with coefficients in $\Z$.

\subsection{A lemma on a curve over a finite field}
Below, we first state a lemma to be used in this section and later
in Lemma~\ref{lem:coker_del43}.   
For a scheme $X$, we let $\cO(X)=H^0(X,\cO_X)$
denote the coordinate ring of $X$.
\begin{lem}\label{prop:AppB_main}
Let $X$ be a connected scheme of pure dimension one,
separated and of finite type over $\Spec\, \F_q$.
Then, the push-forward map 
$$
\alpha_X : 
H^{3}_\cM(X,\Z(2)) \to H^1_\cM(\Spec\, \cO(X),\Z(1))
$$ is an isomorphism if $X$ is proper, and 
$H^{3}_\cM(X,\Z(2))$ is zero if $X$ is
not proper.
\end{lem}
\begin{proof}
This follows from Theorem 1.1 of \cite{KY:Two}.
\end{proof}
\subsection{Integral Chern characters in low degrees}
\label{sec:def chern sing}
Let $Z$ be a scheme over $\Spec\, \F_q$ of pure dimension one,
separated and of finite type over $\Spec\, \F_q$.
We construct a canonical homomorphism
$\chern'_{i,j}:G_i(Z) \to H^{2j-i}_\cM(Z,\Z(j))$
for $(i,j)=(0,0)$, $(0,1)$, $(1,1)$, and $(1,2)$.
Next, we show in Proposition~\ref{prop:Z} that the homomorphism
\begin{equation}\label{eqn:Z}
(\chern'_{i,i},\chern'_{i,i+1}): 
G_i(Z) \to H^i_\cM(Z,\Z(i)) \oplus
H^{i+2}_\cM(Z,\Z(i+1))
\end{equation}
is an isomorphism for $i=0,1$.
Since the $G$-theory of $Z$ and
the $G$-theory of $Z_\mathrm{red}$ are isomorphic,
and the same holds for the motivic cohomology,
it suffices to treat the case where $Z$ is reduced. 

Consider a dense affine open smooth subscheme 
$Z_{(0)}\subset Z$, and let $Z_{(1)} = Z \setminus Z_{(0)}$ 
be the complement of $Z_{(0)}$ with the reduced scheme 
structure.  We define $\chern'_{0,0}$ to be the composition
$$
G_0(Z) \to K_0(Z_{(0)}) \xto{\chern_{0,0}}
H^0_{\cM}(Z_{(0)},\Z(0)) \cong H^0_\cM (Z,\Z(0)).
$$
We then use the following lemma.
\begin{lem}\label{lem:Z}
For $i=0$ (\resp $i=1$), the diagram
$$
\begin{CD}
K_{i+1}(Z_{(0)}) @>>> K_i(Z_{(1)}) \\
@V{c_{i+1,i+1}}VV @VV{\chern_{0,0}\text{ $($\resp }c_{1,1})}V \\
H^{i+1}_\cM(Z_{(0)}, \Z(i+1)) @>>> 
H^i_\cM(Z_{(1)},\Z(i)),
\end{CD}
$$
where each horizontal arrow is a part of the 
localization sequence, is commutative.
\end{lem}
\begin{proof}
Let $\wt{Z}$ denote the normalization of $Z$.
We write $\wt{Z}_{(0)} = Z_{(0)} \times_Z \wt{Z}
(\cong Z_{(0)})$ and 
$\wt{Z}_{(1)} = (Z_{(1)} \times_Z \wt{Z})_\red$.
Comparing diagrams
$$
\begin{array}{ccc}
K_{i+1}(\wt{Z}_{(0)}) & \to & K_i(\wt{Z}_{(1)}) \\
\downarrow & & \downarrow \\
K_{i+1}(Z_{(0)}) & \to & K_i(Z_{(1)})
\end{array}
\text{ and }
\begin{array}{ccc}
H^{i+1}_\cM(\wt{Z}_{(0)},\Z(i+1)) & \to 
& H^i_\cM(\wt{Z}_{(1)},\Z(i)) \\
\downarrow & & \downarrow \\
H^{i+1}_\cM(Z_{(0)},\Z(i+1)) & \to 
& H^i_\cM(Z_{(1)},\Z(i))
\end{array}
$$
reduces us to proving the same claim 
for $\wt{Z}_{(0)}$ and $\wt{Z}_{(1)}$.
This then follows from Lemma~\ref{lem:chern_comm}.
\end{proof}

We define $\chern'_{1,1}$ to be the opposite
of the composition
$$
\begin{array}{l}
G_1(Z) \to \Ker[K_1(Z_{(0)}) \to K_0(Z_{(1)})] \\
\xto{c_{1,1}}
\Ker[H^1_{\cM}(Z_{(0)},\Z(1))\to H^0_\cM(Z_{(1)},\Z(0))]
\cong H^1_\cM (Z,\Z(1)).
\end{array}
$$

Next, we define $\chern'_{1,2}$ when $Z$ is connected. 
If $Z$ is not proper, then $H^3_\cM(Z,\Z(2))$ is zero
by Lemma~\ref{prop:AppB_main}. We set $\chern'_{1,2}=0$
in this case. If $Z$ is proper, then the push-forward map
$$H^3_\cM(Z,\Z(2)) \to H^1_\cM(\Spec\, H^0(Z,\cO_Z),\Z(1))
\cong K_1(\Spec\, H^0(Z,\cO_Z))$$ 
is an isomorphism by
Lemma~\ref{prop:AppB_main}. We define
$\chern'_{1,2}$ to be $(-1)$-times the composition
$$
G_1(Z) \to K_1(\Spec\, H^0(Z,\cO_Z)) 
\cong  H^3_\cM(Z,\Z(2)).
$$
Next, we define $\chern'_{1,2}$ for non-connected $Z$ 
to be the direct sum of
$\chern'_{1,2}$ for each connected component of $Z$.

Observe that the group $G_0(Z)$ is generated by
the two subgroups 
\[M_1=\Image[K_0(Z_{(1)}) \to G_0(Z)] \text{\, and \,}
M_2=\Image[K_0(\wt{Z}) \to G_0(Z)]. 
\]
Using Lemma~\ref{lem:Z} and 
the localization sequences,
the isomorphism $\chern_{0,0}:K_0(Z_{(1)})
\xto{\cong} H^0_\cM(Z_{(1)},\Z(0))$ induces 
a homomorphism
$\chern'_{0,1} : M_1 \to H^2_\cM(Z,\Z(1))$.
The kernel of $K_0(\wt{Z}) \to G_0(Z)$
is contained in the image of 
$K_0(\wt{Z}_{(1)}) \to K_0(\wt{Z})$. 
It can then be easily verified that the composition
$$
K_0(\wt{Z}_{(1)})
\to K_0(\wt{Z}) \xto{c_{0,1}} 
H^2_\cM(\wt{Z},\Z(1)) \to H^2_\cM(Z,\Z(1))
$$
equals the composition
$$
K_0(\wt{Z}_{(1)})
\to K_0(Z_{(1)}) \surj M_1 \xto{\chern'_{0,1}}
H^2_\cM(Z,\Z(1)).
$$
Given the above, the homomorphism $c_{0,1}: K_0(\wt{Z}) \to 
H^2_\cM(\wt{Z},\Z(1))$ induces a homomorphism
$\chern'_{0,1} : M_2 \to H^2_\cM(Z,\Z(1))$ 
such that the two homomorphisms
$\chern'_{0,1}: M_i \to H^2_\cM(Z,\Z(1))$, $i=1,2$,
coincide on $M_1 \cap M_2$. Thus, we obtain 
a homomorphism $\chern'_{0,1}:G_0(Z) \to H^2_\cM(Z,\Z(1))$.

Next, we observe that the four 
homomorphisms $\chern'_{0,0}$,
$\chern'_{0,1}$, $\chern'_{1,1}$, and $\chern'_{1,2}$ 
do not depend on the choice of $Z_{(0)}$. 
\begin{prop}\label{prop:Z}
The homomorphism (\ref{eqn:Z}) for $i=0,1$ is an isomorphism.
\end{prop}
\begin{proof}
It follows from \cite[Corollary 4.3, p.~95]{BMS} that 
the Chern class map $c_{1,1}:K_1(Z_{(0)})\to H^1_{\cM}(Z_{(0)},\Z(1))$
is an isomorphism, hence by construction, 
$\chern'_{1,1}$ is surjective and its kernel equals 
the image of $K_1(Z_{(1)}) \to G_1(Z)$.
It follows from the vanishing of $K_2$ groups of finite fields
that the homomorphism $c_{2,2}:
K_2(Z_{(0)}) \to H^2_\cM(Z_{(0)},\Z(2))$ is an isomorphism. 
We then have isomorphisms
$$
\begin{array}{rl}
& \Image[K_1(Z_{(1)}) \to G_1(Z)] \\
\cong & \Image [H^1_\cM(Z_{(1)},\Z(1)) 
\to H^3_\cM(Z,\Z(2))] \cong  H^3_\cM(Z,\Z(2)),
\end{array}
$$
the first of which is by Lemma~\ref{lem:Z},  
the second by \cite[Corollary 4.3, p.~95]{BMS}.
Therefore, the composition
$\Ker\, \chern'_{1,1} \inj G_1(Z) \xto{\chern'_{1,2}} 
H^3_\cM(Z,\Z(2))$ is an isomorphism. 
This proves the claim for $G_1(Z)$.

By the construction of $\chern'_{0,1}$, the image of
$\chern'_{0,1}$ contains the image of 
$H^0_\cM(Z_{(1)},\Z(0)) \to H^2_\cM(Z,\Z(1))$,
and the composition $K_0(\wt{Z}) \to G_0(Z)
\xto{\chern'_{0,1}} H^2_\cM(Z,\Z(1)) \to 
H^2_\cM(Z_{(0)},\Z(1))$ equals the composition
\[
K_0(\wt{Z}) \to K_0(\wt{Z}_{(0)})
\cong K_0(Z_{(0)}) \xto{c_{0,1}} 
H^2_{\cM}(Z_{(0)},\Z(1)).
\]
The above implies that
$\chern'_{0,1}$ is surjective and the homomorphism
$$
\Ker\, \chern'_{0,1} \to \Ker[K_0(Z_{(0)}) 
\xto{c_{0,1}} H^2_{\cM}(Z_{(0)},\Z(1))]
$$
is an isomorphism.
This proves the claim for $G_0(Z)$.
\end{proof}

\section{$K$-groups and motivic cohomology of 
curves over a function field}\label{sec:KgpsofCurves}
Please note that 
the last paragraph of Section~\ref{sec:intro Chern} 
presents 
an overview of the contents of this section.

In this section, we focus on the following setup.
Let $C$ be a smooth projective geometrically connected
curve over a finite field $\F_q$. Let $k$ denote the function
field of $C$. Let $X$ be a smooth projective 
geometrically connected curve over $k$.
Let $\cX$ be a regular model of $X$, which is proper and
flat over $C$.

From the computations of the motivic cohomology of a surface with a fibration
(e.g., $\cX$ or $\cX$ with some fibers removed),
we deduce some results concerning the $K$-groups of low degrees of 
the generic fiber $X$.  We relate the two using the Chern class maps and 
by taking the limit.

\begin{lem}\label{lem:c12}
The map
$$
K_1(X) \xto{(c_{1,1},c_{1,2})} k^\times \oplus 
H^3_\cM(X,\Z(2))
$$
is an isomorphism. 
The group $H^4_\cM(X,\Z(3))$ is a torsion group
and there exists a canonical short exact sequence
$$
0 \to H^4_\cM(X,\Z(3)) \xto{\beta}
K_2(X) \xto{c_{2,2}}H^2_\cM(X,\Z(2)) \to 0,
$$
such that the composition $c_{2,3} \circ \beta$
equals the multiplication-by-$2$ map.
\end{lem}
\begin{proof}
Let $X_0$ denote the set of closed points of $X$.
Let $\kappa(x)$ denote the residue field at $x \in X_0$.
We construct a commutative diagram by connecting 
the localization sequence
$$
\begin{array}{rl}
& \bigoplus_{x\in X_0} K_2(\kappa(x)) 
\to K_2(X) \to 
K_2(k(X))  \\
\to & \bigoplus_{x\in X_0} 
K_1(\kappa(x)) \to K_1(X) \to K_1(k(X)) 
\to \bigoplus_{x \in X_0} K_0(x)
\end{array}
$$
with the localization sequence
\begin{small}
$$
\begin{array}{rl}
& \bigoplus_{x\in X_0} H^0_\cM(\Spec\, \kappa(x),\Z(1)) \to 
H^2_\cM(X,\Z(2)) \to H^2_\cM(\Spec\, k(X),\Z(2)) \\
\to & \bigoplus_{x\in X_0} H^1_\cM(\Spec\, \kappa(x),\Z(1)) 
\to H^3_\cM(X,\Z(2))
\to H^3_\cM(\Spec\, k(X), \Z(2)),\end{array}
$$
\end{small}using the Chern class maps. Since 
$H^0_\cM(\Spec\, \kappa(x),\Z(1))=0$ 
and the $K$-groups and motivic cohomology 
groups of fields agree in low degrees, 
the claim for $K_1(X)$ follows from diagram chasing.

It also follows from diagram chasing that
$$
K_3(k(X)) \to \bigoplus_{x\in X_0} K_2(\kappa(x)) 
\to K_2(X) \xto{c_{2,2}} H^2_\cM(X,\Z(2)) \to 0
$$
is exact. By \cite[Theorem 4.9, p.~143]{Ne-Su} and \cite[Theorem 1, p.~181]{To},
the groups $H^3_\cM(\Spec\, k(X),\Z(3))$ and $H^2_\cM(\Spec\, \kappa(x),\Z(2))$
for each $x \in X_0$ are isomorphic to
the Milnor $K$-groups $K^M_3(k(X))$ and
$K^M_2(\kappa(x))$, respectively. 
From the definition of these isomorphisms
in \cite{To}, 
the boundary map
$H^3_\cM(k(X),\Z(3))\to
H^2_\cM(\Spec\, \kappa(x),\Z(2))$
is identified under these isomorphisms 
with the boundary map $K^M_3(k(X))
\to K^M_2(\kappa(x))$.
Given the above,
by \cite[Proposition 11.11, p.~562]{Me-Su2}, 
we obtain
the first of the two isomorphisms
$$
\begin{array}{rl}
& \Coker[K_3(k(X)) \to \bigoplus_{x\in X_0} K_2(\kappa(x))] \\
\xto{\cong} &
\Coker[H^3_\cM(k(X),\Z(3)) \to 
\bigoplus_{x\in X_0} H^2_\cM(\Spec\, \kappa(x),\Z(2))] \\
\xto{\cong} &
H^4_\cM(X,\Z(3)).
\end{array}
$$
This gives us the desired short exact sequence.
The identity $c_{2,3}\circ \beta =2$ follows from
Remark~\ref{rmk:chern_comm}.
Since
$H^2_\cM(\Spec\, \kappa(x),\Z(2))$ is a torsion group for each $x \in X_0$,
the group $H^4_\cM(X,\Z(3))$ is also a torsion group.
This completes the proof.
\end{proof}

\begin{lem}\label{lem:c11}\label{lem:aux1}
Let $U \subset C$ be a non-empty open subscheme.
We use $\cX^U$ to denote the complement
$\cX \setminus \cX \times_C U$ 
with the reduced scheme structure.
Then, for $(i,j)=(0,0)$, $(0,1)$ or $(1,1)$,
the diagram
\begin{equation}\label{eqn:aux1}
\begin{CD}
K_{i+1}(X) @>>> G_i(\cX^U) \\
@V{c_{i+1,j+1}}VV 
@V{(-1)^j \chern'_{i,j}}VV \\
H^{2j -i+1}_\cM(X,\Z(j+1)) @>>> 
H^{2j -i}_\cM(\cX^U,\Z(j)),
\end{CD}
\end{equation}
where each horizontal arrow is a part of the localization sequence, 
is commutative. 
\end{lem}
\begin{proof}
Let $\cX^U_\sm \subset \cX^U$ denote the smooth locus.
The commutativity of diagram
(\ref{eqn:aux1}) for $(i,j)=(1,1)$ 
(\resp for $(i,j)=(0,0)$)
follows from the commutativity, which follows from
Lemma~\ref{lem:chern_comm}, of the diagram
$$
\begin{CD}
K_{i+1}(X) @>>> K_i(\cX^U_\sm) \\
@V{c_{i+1,j+1}}VV @VV{
\genfrac{}{}{0mm}{}{-c_{1,1}}{(\text{\resp }\chern_{0,0})}}V \\
H^{2j -i+1}_\cM(X,\Z(j+1)) @>>> 
H^{2j -i}_\cM(\cX^U_\sm,\Z(j))
\end{CD}
$$
and the injectivity of 
$H^{2j -i}_\cM(\cX^U,\Z(j))
\to H^{2j -i}_\cM(\cX^U_\sm,\Z(j))$.

By Lemma~\ref{lem:c12}, the group
$K_1(X)$ is generated by the image of
the push-forward $\bigoplus_{x \in X_0} K_1(\kappa(x))
\to K_1(X)$ and the image of the pull-back 
$K_1(k) \to K_1(X)$. Then, the commutativity
of diagram (\ref{eqn:aux1}) for $(i,j)=(0,1)$
follows from the commutativity of the diagram
$$
\begin{CD}
K_0(Y) @>>> G_0(\cX^U) \\
@V{\chern_{0,0}}VV @V{\chern'_{0,1}}VV \\
H^0_\cM(Y,\Z(0)) @>>> H^2_\cM(\cX^U,\Z(1))
\end{CD}
$$
for any reduced closed subscheme $Y \subset \cX^U$
of dimension zero, where the horizontal arrows are the push-forward
maps by closed immersion, and the fact that the composition
$K_0(C \setminus U) \xto{{f^U}^*} G_0(\cX^U)
\xto{\chern'_{0,1}} H^2_\cM(\cX^U,\Z(1))$ is zero.
Here, $f^U:\cX^U \to C \setminus U$ denotes
the morphism induced by the morphism $\cX \to C$.
\end{proof}

\begin{lem}
The diagram
\begin{equation}\label{eqn:CD_H43}
\begin{small}
\begin{array}{ccccccc}
0 \to &  H^4_\cM(X,\Z(3)) & \to & K_2(X)
& \xto{c_{2,2}} & H^2_\cM(X,\Z(2)) & \to 0 \\
& \downarrow & & \downarrow & & \downarrow & \\
0 \to & {\displaystyle \bigoplus_{\wp\in C_0}} 
H^3_\cM(\cX_\wp,\Z(2)) & \to &
{\displaystyle \bigoplus_{\wp\in C_0}} 
G_1(\cX_\wp) & \xto{-\chern'_{1,1}} &
{\displaystyle \bigoplus_{\wp\in C_0}} 
H^1_\cM(\cX_\wp,\Z(1)) & \to 0, 
\end{array}
\end{small}
\end{equation}
where the first row is as in Lemma~\ref{lem:c12},
the second row is obtained from Proposition~\ref{prop:Z},
and the vertical maps are the boundary maps in the
localization sequences, is commutative.
\end{lem}
\begin{proof}
It follows from Lemma~\ref{lem:c11} that
the right square is commutative.
For each closed point $x \in X_0$, 
let $D_x$ denote the closure of $x$ in $\cX$ 
and write $D_{x,\wp} = D_x \times_C \Spec\, \kwp$
for $\wp \in C_0$.
Then, the commutativity of 
the left square in (\ref{eqn:CD_H43}) 
follows from the commutativity of the diagram
$$
\begin{small}
\begin{array}{ccccccc}
H^4_\cM(X,\Z(3)) & \leftarrow & 
H^2_\cM(\Spec\,\kappa(x),\Z(2)) & \cong & 
K_2(\kappa(x)) & \to & K_2(X) \\
\downarrow & & \downarrow & &
\downarrow & & \downarrow \\
{\displaystyle \bigoplus_{\wp\in C_0} }
H^3_\cM(\cX_\wp,\Z(2)) & \leftarrow &
{\displaystyle \bigoplus_{\wp\in C_0} }
H^1_\cM(D_{x,\wp},\Z(1))
& \cong &
{\displaystyle \bigoplus_{\wp\in C_0} }
G_1(D_{x,\wp})
& \to & {\displaystyle \bigoplus_{\wp\in C_0} }
G_1(\cX_\wp).
\end{array}
\end{small}
$$
Here, each vertical map is a boundary map, the 
middle horizontal arrows are Chern classes, 
and the left and right horizontal arrows are the 
push-forward maps by the closed immersion.
\end{proof}

\begin{lem}\label{lem:coker_del43}
Let $U$ be an open subscheme of $C$ such that $U\neq C$.
Let $\partial: H^4_\cM(\cX \times_C U, \Z(3)) \to
H^3_\cM(\cX_U, \Z(2))$ denote the boundary map of the
localization sequence. Then, the following 
composition is an isomorphism:
$$
\alpha : \Coker\, \partial \inj 
H^5_\cM(\cX,\Z(3)) \to H^1_\cM(\Spec\, \F_q,\Z(1)) 
\cong \F_q^\times.
$$
Here the first map is induced by the push-forward map
by the closed immersion, and the second map is the 
push-forward map by the structure morphism.
\end{lem}
\begin{proof}
For each closed point $x \in X_0$, let $D_x$
denote the closure of $x$ in $\cX$.
We set $D_{x,U}=D_{x}\times_C U$.
Let $D_x^U$ denote the complement
$D_x \setminus D_{x,U}$ with the reduced scheme
structure. Let $\iota_x : D_x\ \inj \cX$,
$\iota_{x,U} : D_{x,U} \ \inj \cX_U$,
$\iota_x^U : D_x^U \ \inj \cX^U$ denote
the canonical inclusions.
Let us consider the commutative diagram
$$
\begin{CD}
H^2_\cM(D_{x,U},\Z(2)) @>>>
H^1_\cM(D_x^U,\Z(1)) @>{\beta}>>
H^3_\cM(D_x,\Z(2)) \\
@V{{\iota_{x,U}}_*}VV @V{{\iota_x^U}_*}VV @. \\
H^4_\cM(\cX \times_C U,\Z(3)) @>{\partial}>>
H^3_\cM(\cX^U,\Z(2)) @>>>
\Coker\, \partial \to 0,
\end{CD}
$$
where the first row is the localization sequence.
Since $X$ is geometrically
connected, it follows from \cite[Corollaire (4.3.12), p.~134]{EGA3}
that each fiber of $\cX \to C$ is
connected. In particular, $D_x$ intersects
every connected component of $\cX^U$,
which implies that the homomorphism ${\iota_x^U}_*$ in
the above diagram is surjective,
hence we have a surjective
homomorphism $\Image\, \beta \surj 
\Coker\, \partial$.
Let $\F(x)$ denote the finite field
$H^0(D_x,\cO_{D_x})$. Then,
the isomorphism $H^3_\cM(D_x,\Z(2)) \to 
H^1_\cM(\Spec\,\F(x),\Z(1)) \cong \F(x)^\times$ 
(see Lemma~\ref{prop:AppB_main}) gives an isomorphism
$\Image\, \beta \cong \F(x)^\times$.
Hence $|\Coker\, \partial|$ divides
$\gcd_{x \in X_0}(|\F(x)^\times|)
= q-1$, where the equality follows from
\cite[1.5.3 Lemme 1, p.~325]{Soule2}.
We can easily verify that the composition
$$
\F(x)^\times \cong \Image\,\beta \surj 
\Coker\, \partial \xto{\alpha} \F_q^\times
$$
equals the norm map $\F(x)^\times \to \F_q^\times$--
which implies $|\Coker\, \partial| \ge q-1$,
hence $|\Coker\, \partial| = q-1$ and 
the homomorphism $\alpha$ is an isomorphism.
The claim is proved.
\end{proof}

\section{Main results for $j \le 2$}\label{conseq}
The reader is referred to Sections~\ref{sec:intro BirchTate} and~\ref{sec:intro high genus} for some remarks concerning the 
contents of this section.

The objective is to prove Theorems~\ref{thm:conseq1},~\ref{thm:conseq2}, and~\ref{thm:conseq3}.
The statements give some information on the structures of $K$-groups and motivic cohomology groups of 
elliptic curves over global fields and of the (open) complements 
of some fibers of 
an elliptic surface over finite fields.
We compute the orders of some torsion groups,
in terms of the special values of $L$-functions, the torsion subgroup of (twisted) Mordell-Weil
groups, and some invariants of the base curve.  
We remark that Milne \cite{Milne} expresses the special values of zeta functions in terms of
the order of arithmetic \'{e}tale cohomology groups.   Our result is similar in spirit.

Let us give the ingredients of the proof.  
Using Theorem~\ref{surjectivity}, we deduce that
the torsion subgroups we are interested in are actually finite.  
Then the theorem of Geisser and Levine 
and the theorem of Merkurjev and Suslin 
relate the motivic cohomology groups modulo 
their uniquely divisible parts
to 
the \'{e}tale cohomology and cohomology of de Rham-Witt complexes.
We use the arguments that appear in \cite{Milne1}, \cite{CSS}, and \cite{Gr-Su} to
compute such cohomology groups  (see Section~\ref{sec:intro CSS} in which we explain our general strategy).  
The computation of the exact orders of torsion may be new.

One geometric property of an elliptic surface that makes this explicit calculation possible
is that the abelian fundamental group is isomorphic to that of the base curve.
This follows from a theorem in \cite{Shioda} for the prime-to-$p$ part.

We also use class field theory of Kato and Saito for surfaces over finite fields \cite{Ka-Sa}.   
We combine their results and Lemma~\ref{lem:appl_pi1} 
to show that the groups of zero-cycles on the elliptic surface and on the base curve are isomorphic.   

Let us give a brief outline.
Recall in Section~\ref{sec: L-function}
that the main 
part of the zeta function of $\cE$ 
is the $L$-function of $E$.  
Note that there are 
contributions from 
the $L$-function of the base curve $C$ 
and also of the singular fibers of $\cE$
(Lemma~\ref{lem:appl_LofcE}, Corollary~\ref{cor:open_L}).
We relate the $G$-groups of these singular fibers to
the motivic cohomology groups, 
then they are in
turn related to the values of the $L$-function
(Lemmas~\ref{lem:G1_isom} and~\ref{lem:G0_order}).
The technical input for these lemmas originate from our 
results in \cite{KY:Two}
and the classification of the
singular fibers of an elliptic fibration.

\subsection{Notation}\label{subsec:12_1}
Let $k$, $E$, $S_0$, $S_2$, $r$, $C$, and $\cE$ be as defined in 
Section~\ref{sec:introduction}. We also let
$S_1$ denote the set of primes
of $k$ at which $E$ has multiplicative
reduction, 
thus we have $S_0 \subset S_1 \subset S_2$.
Let $p$ denote the characteristic of $k$.
The closure of the origin of $E$ in $\cE$ gives
a section to $\cE \to C$, which we denote 
by $\iota:C \to \cE$. 
Throughout this section, we assume
that the structure morphism $f:\cE \to C$ is not smooth,
which in particular implies that it is not isotrivial.
For any scheme $X$ over $C$,
let $\cE_X$ denote the base change $\cE\times_C X$.
For any non-empty open subscheme $U\subset C$,
we use $\cE^U$ to denote the complement
$\cE \setminus \cE_U$ with the reduced scheme structure.

Let $\F_q$ denote the field of constants of $C$.
We take an algebraic closure $\Fbar_q$ of $\F_q$.
Let $\Frob \in \GFq = \Gal(\Fbar_q/\F_q)$ denote the
geometric Frobenius. For a scheme $X$ over $\Spec\, \F_q$, 
we use $\Xbar$ to denote its base change 
$\Xbar = X \times_{\Spec \, \F_q} \Spec \, \Fbar_q$ to $\Fbar_q$.
We often regard the set $\Irr(X)$ of irreducible components of $X$
as a finite \'{e}tale scheme over $\Spec\, \F_q$ corresponding to
the $\GFq$-set $\Irr(\Xbar)$.

\subsection{Results}
We set $T=E(k\otimes_{\F_q} \Fbar_q)_\tors$.
For each integer $j \in \Z$, we let
$T'_{(j)} = \bigoplus_{\ell \neq p} 
(T \otimes_\Z \Z_\ell(j))^{\GFq}$.

\begin{thm}\label{thm:conseq1}
Let the notations and assumptions be as above.
Let $L(E,s)$ denote the $L$-function of the elliptic curve
$E$ over global field $k$ (see Section~\ref{sec: L-function}).
\begin{enumerate}
\item The $\Q$-vector space $(K_2(E)^\cotors)_\Q$
is of dimension $r$.
\item The cokernel of the boundary map
$\partial_2 : K_2(E) \to \bigoplus_{\wp \in C_0}
G_1(\cE_\wp)$ is a finite group of order
$$
\frac{(q-1)^2 |L(h^0(\Irr(\cE_{S_2})),-1)|}{
|T'_{(1)}|\cdot |L(h^0(S_2), -1)|}.
$$
\item The group $K_1(E)_\divi$ is uniquely divisible.
\item The kernel of the boundary map
$\partial_1: K_1(E)^\cotors \to \bigoplus_{\wp \in C_0}G_0(\cE_\wp)$
is a finite group of order $(q-1)^2 |T'_{(1)}|\cdot |L(E,0)|$.
The cokernel of $\partial_1$ is a
finitely generated abelian group of rank
$2+|\Irr(\cE_{S_2})|-|S_2|$ whose torsion
subgroup is isomorphic to $\Jac(C)(\F_q)^{\oplus 2}$,
where $\Jac(C)$ denotes the Jacobian of $C$ 
(when the genus of $C$ is $0$, we understand it to be a point).
\end{enumerate}
\end{thm}

Let $X$ be a scheme of finite type over $\Spec\, \F_q$.
For an integer $i \in \Z$ and a prime number $\ell \neq p$,
we set $L(h^i(X),s) = 
\det(1-\Frob\cdot q^{-s};H^i_\etale(\Xbar,\Q_\ell))$.
In all cases considered in our paper, 
$L(h^i(X),s)$ does not depend on the 
choice of $\ell$.

For each non-empty open subscheme $U\subset C$, let
$T_U$ denote the torsion subgroup of the
group $\Div(\cEbar_U)/\sim_\alg$ of divisors on
$\cEbar_U$ modulo algebraic equivalence.
For each integer $j \in \Z$, 
we set $T'_{U,(j)} = \bigoplus_{\ell \neq p} 
(T_U \otimes_\Z \Z_\ell(j))^{\GFq}$.
We deduce from \cite[Theorem 1.3, p.~214]{Shioda}
that the canonical homomorphism 
$\varinjlim_{U'} T_{U'} \to T$, where the limit is taken over
the open subschemes of $C$, is an isomorphism.
We easily verify that the canonical homomorphism
$T_U \to \varinjlim_{U'} T_{U'}$ is injective
and is an isomorphism if $\cEU \to U$ is smooth.
In particular, 
we have an injection $T'_{U,(j)} \inj T'_{(j)}$, which 
is an isomorphism if $\cEU \to U$ is smooth.

In Section~\ref{subsec:appl_proofs}, 
we deduce Theorem~\ref{thm:conseq1}
from the following two theorems.

\begin{thm}\label{thm:conseq2}
Let the notations and assumptions be as above.
Let $\partial^i_{\cM,j} : 
H^i_\cM(E,\Z(j))^\cotors \to \bigoplus_{\wp \in C_0}
H^{i-1}_\cM(\cE_\wp,\Z(j-1))$ denote the homomorphism
induced by the boundary map of the localization sequence 
established in \cite[Corollary (0.2), p.~537]{Bloch2}.
\begin{enumerate} 
\item For any $i \in \Z$, the group
$H^i_\cM(E,\Z(2))_\divi$ is uniquely divisible.
\item For $i\le 0$,
the cohomology group $H^i_\cM(E,\Z(2))$ is uniquely divisible.
$H^1_\cM(E,\Z(2))$ is finite modulo a uniquely
divisible subgroup and
$H^1_\cM(E,\Z(2))_\tors$ is cyclic of order $q^2 -1$.
\item The kernel (\resp cokernel) 
of the homomorphism $\partial^{2}_{\cM,2}$
is a finite group of order $|L(h^1(C),-1)|$
(\resp of order
$$
\frac{(q-1)|L(h^0(\Irr(\cE_{S_2})),-1)|}{
|T'_{(1)}|\cdot |L(h^0(S_2),-1)|}).
$$
\item The kernel (\resp cokernel) 
of the homomorphism $\partial^{3}_{\cM,2}$
is a finite group of order 
$(q-1)|T'_{(1)}|\cdot |L(E,0)|$
(\resp is isomorphic to $\Pic(C)$).
\item For $i \ge 4$, the group $H^i_\cM(E,\Z(2))$ is zero.
\item $H^4_\cM(E,\Z(3))$ is a torsion group,
and the cokernel of the homomorphism 
$\partial^{4}_{\cM,3}$ is a finite cyclic group of order $q-1$.
\end{enumerate}
\end{thm}

\begin{thm}\label{thm:conseq3}
Let $U\subset C$ be a non-empty open subscheme. Then, we have the following.
\begin{enumerate}
\item For any $i \in \Z$, the group
$H^i_\cM(\cE_U,\Z(2))$ is finitely generated
modulo a uniquely divisible subgroup.
\item For $i \le 0$, the cohomology group $H^i_\cM(\cE_U,\Z(2))$ is 
uniquely divisible,
$H^1_\cM(\cE_U,\Z(2))$ is finite
modulo a uniquely divisible subgroup
and $H^1_\cM(\cE_U,\Z(2))_\tors$ is cyclic 
of order $q^2 -1$.
\item The rank of $H^2_\cM(\cEU,\Z(2))^\cotors$
is $|S_0 \setminus U|$.
If $U=C$ (\resp $U\neq C$), 
the torsion subgroup of $H^2_\cM(\cEU,\Z(2))^\cotors$ is 
of order $|L(h^1(C),-1)|$ (\resp of order 
$|T'_{U,(1)}|\cdot |L(h^1(C),-1)L(h^0(C\setminus U),-1)|/(q-1)$).
\item  If $U=C$ (\resp $U\neq C$), 
the cokernel of the boundary homomorphism
$H^2_\cM(\cEU,\Z(2)) \to H^1_\cM(\cE^U,\Z(1))$ 
is zero (\resp is finite of order
$$
\frac{(q-1)|L(h^0(\Irr(\cE^U)),-1)|}{
|T'_{U,(1)}| \cdot |L(h^0(C\setminus U),-1)|}).
$$
\item The rank of $H^3_\cM(\cEU,\Z(2))^\cotors$ is
$\max(|C \setminus U| -1,0)$.
If $U=C$ (\resp $U\neq C$), 
the torsion subgroup $H^3_\cM(\cEU,\Z(2))_\tors$ 
is finite of order $|L(h^2(\cE),0)|$ (\resp of order 
$$
\frac{|T'_{U,(1)}|\cdot |L(h^2(\cE),0)L^{*}(h^1(\cE^U),0)
L(h^0(C\setminus U),-1)|}{(q-1) |L(h^0(\Irr(\cE^U)),-1)|}).
$$ 
Here,
$$
L^{*} (h^1(\cE^U),0)
= \lim_{s \to 0} (s \log q)^{-|S_0\setminus U|}
L(h^1(\cE^U),s)
$$
is the leading coefficient of $L(h^1(\cE^U),s)$ at $s=0$.
\item $H^4_\cM(\cE_U,\Z(2))$ is canonically
isomorphic to $\Pic(U)$.
For $i\ge 5$, the group $H^i_\cM(\cE_U,\Z(2))$ is zero.
\end{enumerate}
\end{thm}

\subsection{Relation between $L(E,s)$ and the congruence zeta
function of $\cE$.}
\label{sec: L-function}
Let $\ell \neq p$ be a prime number. 
Take an open subset $j: U \to C$ 
such that the restriction $f_U: \cE_U \to U$ is (proper) smooth.

By the Grothendieck-Lefschetz trace formula, we have
$$
L(E,s)  = \prod_{i=0}^2 
\det(1-\Frob\cdot q^{-s};H^i_\etale(\Cbar, j_* R^1 f_{U*} \Q_\ell))^{(-1)^{i-1}}.
$$

Note that by the adjunction $\id \to j_* j^*$, 
we have a canonical morphism
\[
R^1 f_* \Q_\ell \to j_* j^* R^1 f_* \Q_\ell \cong j_* R^1 f_{U*} \Q_\ell.
\]
We will prove the following lemma:
\begin{lemma}
The map above is an isomorphism.
\end{lemma}
\begin{proof}
Let $\wp$ be a closed point of $C$, and let $S$ be
strict henselization of $C$ at $\wp$. Let $s$ be
the closed point of $S$ and let $\etabar$ denote the
spectrum of a separable closure of the function field of $S$.

We have the specialization map 
$H^1(\cE_s,\Q_\ell) \to H^1(\cE_{\etabar},\Q_\ell)$.
This is identified with
$H^1(\pi_1(\cE_{\bar{\eta}}), \Q_\ell)
\to
H^1(\pi_1(\cE_{\bar{s}}), \Q_\ell)$, and is injective,
since $\pi_1(\cE_{\bar{s}}) \to \pi_1(\cE_{\bar{\eta}})$ is 
surjective (\cite[p.270, Exp X, Corollaire 2.4]{SGA1}).

Set $V=  H^1(\cE_{\etabar},\Q_\ell)$.
For surjectivity, we prove that the
image of the specialization map is equal to
the invariant part $V^{I_s}$
under the action of the inertia group $I_s$ at $s$.
It suffices to show that
$H^1(\cE_s,\Q_\ell)$ and $V^{I_s}$
have the same dimension.

We use the Kodaira-N\'eron-Tate classification
of bad fibers (\cf \cite[10.2.1, p.~484--489]{Liu})
to compute that $H^1(\cE_s,\Q_\ell)$ is of dimension $2$, $1$, $0$ if
$E$ has good reduction, semistable bad reduction, and additive reduction
at $\wp$, respectively.
It is well known that $V^{I_s}$ is of dimension $2$
if and only if $E$ has good reduction at $p$.
If $E$ has semistable bad reduction, then
the action of $I_s$ on $V$ is nontrivial 
and unipotent. Hence $V^{I_s}$ is of dimension $1$
in this case.
Suppose that $E$ has additive reduction and 
$V^{I_s} \neq 0$. Then $V^{I_s}$ is of dimension one.
Poincare duality implies that $I_s$ acts trivially
on $V/V^{I_s}$, hence the action of $I_s$ on $V$ 
is unipotent. Therefore, $E$ must have either good or semistable reduction,
which leads to a contradiction.
\end{proof}

\begin{cor}
We have
$$
L(E,s)  = \prod_{i=0}^2 
\det(1-\Frob\cdot q^{-s};H^i_\etale(\Cbar,R^1 f_* \Q_\ell))^{(-1)^{i-1}}.
$$
\end{cor}
Below, we will work with this expression.

\begin{lem}\label{lem:h1_weight}
Let $D$ be a scheme of dimension $\le 1$
that  is proper over $\Spec\, \F_q$.
Let $\ell \neq p$ be an integer.
Then, the group $H^i_\etale(\Dbar,\Z_\ell)$
is torsion free for any $i \in \Z$ and is
zero for $i \neq 0,1,2$. 
The group 
$H^i_\etale(\Dbar,\Q_\ell)$ is pure of weight $i$
for $i \neq 1$ 
and is mixed of weight $\{ 0,1 \}$
for $i=1$. 
The group $H^1_\etale(\Dbar,\Q_\ell)$ is
pure of weight one (\resp pure of weight zero)
if $D$ is smooth (\resp every irreducible component
of $\Dbar$ is rational).
\end{lem}
\begin{proof}
We may assume that $D$ is reduced.
Let $D'$ be the normalization of $D$.
Let $\pi : D' \to D$ denote the canonical morphism.
Let $\cF_n$ denote the cokernel of the
homomorphism $\Z/\ell^n \to \pi_* (\Z/\ell^n)$
of \'{e}tale sheaves. The sheaf $\cF_n$ is supported
on the singular locus $D_\sing$ of $D$ and is
isomorphic to $i_*(\Coker[\Z/\ell^n 
\to \pi_{\sing *} (\Z/\ell^n)])$, where
$i :D_\sing \inj D$ is the canonical inclusion
and $\pi_{\sing}: D' \times_D D_\sing \to D_\sing$
is the base change of $\pi$. 
Then, the claim follows from
the long exact sequence
$$
\cdots \to H^i_\etale(\Dbar,\Z/\ell^n)
\to H^i_\etale(\Dbar',\Z/\ell^n)
\to H^i_\etale(\Dbar,\cF_n) \to \cdots.
$$
\end{proof}

\begin{lem}\label{lem:appl_vanishing}
If $i \neq 1$,
then $H^i_\etale(\Cbar,R^1 f_* \Q_\ell)=0$.
\end{lem}
\begin{proof}
For any point $x \in \Cbar(\Fbar_q)$ lying over
a closed point $\wp \in C$, the canonical
homomorphism $H^0_\etale(\Cbar,R^1 f_* \Q_\ell) \to
H^1_\etale(\cE_x,\Q_\ell)$ is injective since
$H^0_{\etale,c}(\Cbar\setminus\{x\}, R^1 f_* \Q_\ell)=0$.
By Lemma~\ref{lem:h1_weight}, 
the module $H^1_\etale(\cE_x,\Q_\ell)$ is
pure of weight one (\resp of weight zero)
if $\cE_x$ is smooth (\resp is not smooth). 
Since we have assumed that $f:\cE \to C$ is not smooth,
$H^0_\etale(\Cbar,R^1 f_* \Q_\ell)=0$.

Consider a non-empty open subscheme $U\subset C$ such that $f_U:\cE_U\to U$
is smooth. 
The group
$H^2_\etale(\Cbar,R^1 f_* \Q_\ell) \cong
H^2_{\etale,c}(\Ubar,R^1 f_{U *} \Q_\ell)$ is the dual of
$H^0_\etale(\Ubar,R^1 f_{U *} \Q_\ell(1))$ by Poincar\'e duality.
Next, assume $H^0_\etale(\Ubar,R^1 f_{U *} \Q_\ell(1)) \neq 0$.
Let $T_\ell(E)$ denote the $\ell$-adic Tate module of $E$. 
The \'{e}tale fundamental group $\pi_1(U)$ acts on $T_\ell(E)$.
By the given assumption, the $\pi_1(\Ubar)$-invariant part
$V=(T_\ell(E)\otimes \Q_\ell)^{\pi_1(\Ubar)}$ is nonzero.
Since $f$ is not smooth, $V$ is one-dimensional, 
hence we have a nonzero homomorphism
$\pi_1(\Ubar)^\ab \to \Hom(T_\ell(E)\otimes \Q_\ell/V,V)$
of $\GFq$-modules. By the weight argument, we observe that
this is impossible, hence $H^2_\etale(\Cbar,R^1 f_{*} \Q_\ell(1))= 0$.
\end{proof}

As an immediate consequence, we obtain the following corollary.
\begin{cor}
The spectral sequence
$$
E_2^{i,j}=H^i_\etale(\Cbar,R^j f_* \Q_\ell)
\Rightarrow H^{i+j}_\etale(\cEbar, \Q_\ell)
$$
is $E_2$-degenerate. \qed
\end{cor}
\begin{lem}\label{lem:appl_LofcE}
Let $U \subset C$ be a non-empty open
subscheme such that $f_U: \cE_U \to U$ is smooth.
Let $\Irr^0(\cE^U) \subset \Irr(\cE^U)$ 
denote the subset of irreducible components of $\cE^U$ 
that do not intersect $\iota(C)$.
We regard $\Irr^0(\cE^U)$ as a closed subscheme
of $\Irr(\cE^U)$ (recall the convention in Section~\ref{subsec:12_1} 
on $\Irr(\cE^U)$). Then,
$$
L(h^i(\cE),s)
= \left\{\begin{array}{ll}
(1-q^{-s}), & \text{ if }i=0,\\
L(h^1(C),s), & \text{ if }i=1,\\
(1-q^{1-s})^2 L(E,s)L(h^0(\Irr^0(\cE^U)),s-1), 
& \text{ if }i=2,\\
L(h^1(C),s-1), & \text{ if }i=3,\\
(1-q^{2-s}), & \text{ if }i=4.
\end{array}
\right.
$$
\end{lem}
\begin{proof}
We prove the lemma for $i=2$; the other cases are straightforward.
Since $R^2 f_{U *} \Q_\ell \cong \Q_\ell(-1)$,
there exists an exact sequence
$$
0 \to H^0_\etale(\Cbar,R^2 f_* \Q_\ell)
\to H^2_\etale(\overline{\cE^U},\Q_\ell) \to 
H^1_{\etale,c}(\Ubar,\Q_\ell(-1)).
$$
The map
$H^2_\etale(\overline{\cE^U},\Q_\ell) \to 
H^1_{\etale,c}(\Ubar,\Q_\ell(-1))$ decomposes as
$$
H^2_\etale(\overline{\cE^U},\Q_\ell) \to 
H^0_\etale(\Cbar \setminus \Ubar,\Q_\ell(-1))
\to H^1_{\etale,c}(\Ubar,\Q_\ell(-1)).
$$
Given the above, 
$H^0_\etale(\Cbar,R^2 f_* \Q_\ell)$ is isomorphic
to the inverse image of the image of the homomorphism
$H^0_\etale(\Cbar,\Q_\ell(-1)) \to
H^0_\etale(\Cbar \setminus \Ubar, 
\Q_\ell(-1))$ under the surjective homomorphism
$H^2_\etale(\overline{\cE^U},\Q_\ell) \to 
H^0_\etale(\Cbar \setminus \Ubar, 
\Q_\ell(-1))$.
This proves the claim. 
\end{proof}

\subsection{The fundamental group of $\cE$}

\begin{lem}\label{lem:EPchar}
For $i=0,1$, the pull-back
$H^i(C,\cO_C) \to H^i(\cE,\cO_\cE)$
is an isomorphism.
\end{lem}
\begin{proof}
The claim for $i=0$ is clear,
thus we prove the claim for $i=1$.
Let us write $\cL=R^1f_*\cO_{\cE}$.
It suffices to prove $H^0(C, \cL)=0$.
We note that $\cL$ is an invertible $\cO_C$-module
since the morphism $\cE \to C$ has no multiple fiber. 
The Leray spectral sequence
$E_2^{i,j}=H^i(C,R^j f_* \cO_\cE) \Rightarrow
H^{i+j}(\cE,\cO_\cE)$ shows that
the Euler-Poincar\'e characteristic
$\chi(\cO_\cE)$ equals $\chi(\cO_C) - \chi(\cL) = -\deg \cL$.
By the well-known inequality $\chi(\cO_\cE) > 0$ 
(\cf \cite{Ogg}, \cite{Dol}, or \cite[Theorem 2, p.~81]{Oguiso}),
we obtain $\deg \cL < 0$, which proves $H^0(C, \cL)=0$.
\end{proof}

\begin{lem}\label{lem:appl_pi1}
\begin{enumerate}
\item 
The canonical homomorphism
$\pi_1^{\ab}(\cE) \to \pi_1^{\ab}(C)$ 
between the abelian (\'{e}tale) fundamental groups 
is an isomorphism.
\item 
The canonical morphism
$\Pic^o_{C/\F_q} \to \Pic^o_{\cE/\F_q}$
between the identity components of the Picard schemes 
is an isomorphism.
\end{enumerate}
\end{lem}
\begin{proof}
The homomorphism
$\Pic^o_{C/\F_q} \to \Pic^o_{\cE/\F_q,\red}$ is
an isomorphism by \cite[Theorem 4.1, p.~219]{Shioda}.
This, combined with the cohomology 
long exact sequence of the Kummer sequence, implies that
if $p \nmid m$, then $H^1_\etale(C,\Z/m) \to H^1_\etale(\cE,\Z/m)$
is an isomorphism, hence to prove (1),
we are reduced to showing that 
$H^1_\etale(C,\Z/p^n) \to H^1_\etale(\cE,\Z/p^n)$
is an isomorphism for all $n\ge 1$.
For any scheme $X$ that is proper over $\Spec\, \F_q$, 
there exists an exact sequence
$$
0 \to \Z/p^n \Z \to W_n \cO_X \xto{1-\sigma} W_n \cO_X \to 0
$$
of \'{e}tale sheaves, where $W_n \cO_X$ 
is the sheaf of Witt vectors and 
$\sigma: W_n \cO_X \to W_n \cO_X$
is the Frobenius endomorphism. This gives rise to
the following commutative diagram with exact rows
$$
\begin{array}{ccccccc}
\cdots \xto{1-\sigma} & H^0(C,W_n \cO_C)
& \to & H^1_\etale(C,\Z/p^n) & \to &
H^1(C,W_n \cO_C) & \xto{1-\sigma} \cdots \\
 & \downarrow & & \downarrow & &
\downarrow & \\
\cdots \xto{1-\sigma} & 
H^0(\cE,W_n \cO_\cE) & \to & H^1_\etale(\cE,\Z/p^n) 
& \to & H^1(\cE, W_n \cO_\cE) 
& \xto{1-\sigma} \cdots.
\end{array}
$$
By Lemma~\ref{lem:EPchar} and induction on $n$,
we observe that the homomorphism
$H^i(C,W_n \cO_C) \to H^i(\cE,W_n \cO_\cE)$
is an isomorphism for $i=0,1$, thus the map 
$H^1_\etale(C,\Z/p^n) \to 
H^1_\etale(\cE,\Z/p^n)$ is an isomorphism.
This proves claim (1).

For (2), it suffices to prove that 
homomorphism $\Lie\, \Pic_{C/\F_q} \to \Lie\, \Pic_{\cE/\F_q}$
between the tangent spaces is an isomorphism.
Since this homomorphism is identified with
homomorphism $H^1(C,\cO_C) \to H^1(\cE,\cO_\cE)$,
claim (2) follows from Lemma~\ref{lem:EPchar}.
\end{proof}

\begin{rmk}
Using Lemma~\ref{lem:appl_pi1} (1), we can prove that 
the homomorphism $\pi_1(\cE) \to \pi_1(C)$ is 
an isomorphism. Since it is not used in this paper,
we only sketch the proof below.

Let $x \to C$ be a geometric point.
Since the morphism $f:\cE \to C$ has a section,
fiber $\cE_x$ of $f$ at $x$ has a reduced
irreducible component. 
This, together with the regularity of $\cE$ and $C$, shows that
the canonical ring homomorphism $H^0(x,\cO_x) \to H^0(Y,\cO_Y)$
is an isomorphism for any connected finite \'{e}tale covering $Y$ of $\cE_x$,
hence
by the same argument as in the proof of 
\cite[X, Proposition 1.2, Th\'{e}or\`{e}me 1.3, p.~262]{SGA1},
we have an exact sequence
$$
\pi_1(\cE_x) \to \pi_1(\cE) 
\to \pi_1(C) \to 1.
$$
In particular here, the kernel of $\pi_1(\cE) 
\to \pi_1(C)$ is abelian. Applying 
Lemma~\ref{lem:appl_pi1}(1)
to $\cE \times_C C' \to C'$ for each finite connected
\'{e}tale cover $C' \to C$, we obtain the bijectivity
of $\pi_1(\cE) \to \pi_1(C)$.

More generally, the statements in Lemma~\ref{lem:appl_pi1} and 
the statement above that the fundamental groups are isomorphic are also
valid if $\cE$ is a regular minimal elliptic 
fibration that is proper flat non-smooth over $C$ with a section,
where $C$ is a proper smooth curve over an arbitrary perfect base field.
\end{rmk}

\begin{cor}\label{cor:appl_divisible}
For any prime number $\ell \neq p$ and for
any $i \in \Z$,
the group $H^i_\etale(\cEbar,\Q_\ell/\Z_\ell)$ is
divisible.
\end{cor}

\begin{proof}
The claim for $i \neq 1,2$ is obvious. 
By Lemma~\ref{lem:appl_pi1}, we have
$H^1_\etale(\cEbar,\Q_\ell/\Z_\ell) \cong
H^1_\etale(\Cbar,\Q_\ell/\Z_\ell)$,
hence $H^1_\etale(\cEbar,\Q_\ell/\Z_\ell)$ is divisible.
The group $H^2_\etale(\cEbar,\Q_\ell/\Z_\ell)$
is divisible since 
$H^2_\etale(\cEbar,\Q_\ell/\Z_\ell)^\cotors$
is isomorphic to the Pontryagin dual of 
$H^1_\etale(\cEbar,\Q_\ell/\Z_\ell(2))^\cotors$.
\end{proof}

\begin{cor}\label{cor:appl_UC}
For $i \in \Z$, we set $M^i_j =
\bigoplus_{\ell \neq p} H^i_\etale(\cE,\Q_\ell/\Z_\ell(j))$.
For a rational number $a$, we write $|a|^{(p')}
= |a|\cdot |a|_p$.
\begin{enumerate}
\item For $i \le -1$ or $i \ge 6$, the group
$M^i_j$ is zero. 
\item For $j \neq 2$ (\resp $j=2$), 
the group $M^5_j$ is zero (\resp is
isomorphic to $\bigoplus_{\ell \neq p}\Q_\ell/\Z_\ell$).
\item For $j \neq 0$, the group $M^0_j$
is cyclic of order $q^{|j|} -1$. 
The group
$M^0_0$ is isomorphic to 
$\bigoplus_{\ell \neq p}\Q_\ell/\Z_\ell$.
\item For $j \neq 0$, the group $M^1_j$ is
finite of order $|L(h^1(C),1-j)|^{(p')}$.
\item For $j \neq 1$, the group $M^2_j$
is finite of order $|L(h^2(\cE),2-j)|^{(p')}$.
\item For $j \neq 1$, group $M^3_j$ is
finite of order $|L(h^1(C),2-j)|^{(p')}$.
\item For $j \neq 2$, the group $M^4_j$ is
cyclic of order $q^{|2-j|}-1$. The group
$M^4_2$ is isomorphic to 
$\bigoplus_{\ell \neq p}\Q_\ell/\Z_\ell$.
\end{enumerate}
\end{cor}
\begin{proof}
By Corollary~\ref{cor:appl_divisible}, if 
$i\neq 2j+1$ and $\ell \neq p$, the group
$H^i_\etale(\cE,\Q_\ell/\Z_\ell(j))$
is isomorphic to 
$H^i_\etale(\cEbar,\Q_\ell/\Z_\ell(j))^{\GFq}$.
Then, we have
$$
|H^i_\etale(\cE,\Q_\ell/\Z_\ell(j))|
= |L(h^{4-i}(\cE),2-j)|_\ell^{-1}
$$
by Poincar\'e duality for $i \neq 2j, 2j+1$,
hence the claim follows 
from Lemma~\ref{lem:appl_LofcE}.
\end{proof}

\subsection{Torsion in the \'{e}tale cohomology 
of open elliptic surfaces}

We first fix a non-empty open subscheme $U \subset C$.
\begin{lem}\label{lem:appl_coker}
Let $\ell \neq p$ be a prime number.
For $i \in \Z$, let $\gamma_i$ denote 
the pull-back
$\gamma_i:H^i_\etale (\cEbar,\Z_\ell) 
\to H^i_\etale (\cEsupUbar,\Z_\ell)$.
\begin{enumerate}
\item For $i \neq 0,2$, the homomorphism $\gamma_i$ is zero.
\item The cokernel $(\Coker\, \gamma_2)_{\Q_\ell}$ is 
isomorphic to the kernel of 
$H^0_\etale (\Cbar \setminus \Ubar,\Q_\ell(-1)) 
\to H^0_\etale(\Spec\, \Fbar_q,\Q_\ell(-1))$.
\item There exists a canonical isomorphism
$$\Hom_\Z (T_U,\Q_\ell/\Z_\ell(-1))
\cong (\Coker\, \gamma_2)_\tors.$$
\end{enumerate}
\end{lem}
\begin{proof}
By Lemma~\ref{lem:appl_pi1}, the pull-back
$H^1_\etale(\Cbar,\Z_\ell) \to H^1_\etale(\cEbar,\Z_\ell)$ 
is an isomorphism, hence the homomorphism 
$H^1_\etale (\cEbar,\Z_\ell) 
\to H^1_\etale (\cEsupUbar,\Z_\ell)$
is zero.  Claim (1) follows.

Let $\NS(\cEbar)=\Pic_{\cE/\F_q}(\Fbar_q)/
\Pic^o_{\cE/\F_q}(\Fbar_q)$ denote the N\'eron-Severi group of $\cEbar$.
For a prime number $\ell$, we set
$T_\ell M = \Hom(\Q_\ell/\Z_\ell,M)$.

We have the following exact sequence from 
Kummer theory:
\[ 
0\to \NS(\cEbar) \otimes_\Z \Z_\ell \xto{\cl_\ell} 
H^2_\etale(\cEbar,\Z_\ell(1)) \to 
T_\ell H^2_\etale(\cEbar,\Gm) \to 0. \]
We note that $T_\ell H^2_\etale(\cEbar,\Gm)$ is torsion free.
For $D \in \Irr(\cEsupUbar)$, let
$[D] \in \NS(\cEbar)$ denote the class of the Weil divisor $D_\red$ 
on $\cEbar$. By~\cite[Cycle, Definition~2.3.2, p.~145]{SGA412}, the
$D$-component of the homomorphism 
$\gamma_2:H^2_\etale (\cEbar,\Z_\ell) 
\to H^2_\etale (\cEsupUbar,\Z_\ell)
\cong \Map(\Irr(\cEsupUbar),\Z_\ell(-1))$
is identified with the homomorphism
$$
H^2_\etale (\cEbar,\Z_\ell)
\xto{\cup \cl_\ell([D])}
H^4_\etale (\cEbar,\Z_\ell(1)) \cong
\Z_\ell(-1).
$$
Let $M \subset \NS(\cEbar)$ denote the subgroup
generated by $\{[D]\ |\ D \in 
\Irr(\cEsupUbar) \}$.
By Corollary~\ref{cor:appl_divisible}, the cup-product 
$$
H^2_\etale(\cEbar,\Z_\ell(1)) \times
H^2_\etale(\cEbar,\Z_\ell(1)) \to 
H^4_\etale(\cEbar,\Z_\ell(2)) \cong \Z_\ell
$$ 
is a perfect pairing, hence the image of $\gamma_2$
is identified with the image of the composition
$$
\begin{array}{l}
\Hom_{\Z_\ell}(H^2_\etale(\cEbar,\Z_\ell(1)),\Z_\ell(-1))
\xto{\alpha^*} \Hom_{\Z}(M,\Z_\ell(-1)) \\
\xto{\beta^*} \Map(\Irr(\cEsupUbar),\Z_\ell(-1))
\cong H^2_\etale (\cEsupUbar,\Z_\ell),
\end{array}
$$
where $\alpha^*$ is the homomorphism 
induced by the restriction 
$\alpha: M \otimes_{\Z} \Z_\ell \inj 
H^2_\etale(\cEbar,\Z_\ell(1))$
of the cycle class map $\cl_\ell$ to $M$,
and the homomorphism $\beta^*$ is 
induced by the canonical surjection
$\beta: \bigoplus_{D\in\Irr(\cEsupUbar)} \Z
\surj M$.
Since $\beta$ is surjective, the homomorphism
$\beta^*$ is injective and we have an exact sequence
$$
0 \to \Coker\, \alpha^* \to \Coker\, \gamma_2
\to \Coker\, \beta^* \to 0.
$$
Since $\alpha$ is a homomorphism
of finitely generated $\Z_\ell$-modules that is injective,
the cokernel $\Coker\, \alpha^*$ is a finite
group.
Further, the group $M$ is a free abelian group with basis
$\Irr^0(\cEsupUbar)
 \cup \{D'\}$, where $D'$ is an arbitrary
element in $\Irr(\cEsupUbar)
\setminus \Irr^0(\cEsupUbar)$,
hence $\Coker\, \beta^*$ is isomorphic
to the group $\Hom_{\Z}(\Ker \, \beta, \Z_\ell(-1))$.
This proves (2).

The torsion part of
$\Coker\, \gamma_2$ is identified
with the group $\Coker\, \alpha^*$.
The homomorphism $\alpha^*$ is the composition of
the two homomorphisms
$$\begin{array}{ll}
\Hom_{\Z_\ell}(H^2_\etale(\cEbar,\Z_\ell(1)),\Z_\ell(-1))
\\ 
\xto{\cl_\ell^*}
\Hom_{\Z}(\NS(\cEbar),\Z_\ell(-1))
\xto{\iota^*} \Hom_{\Z}(M,\Z_\ell(-1)),
\end{array}
$$
where the first (\resp second) homomorphism
is that induced by the cycle class map 
$\cl_\ell:\NS(\cEbar) \otimes_\Z \Z_\ell \inj
H^2_\etale(\cEbar,\Z_\ell(1))$
(\resp the inclusion $M \otimes_\Z \Z_\ell \inj \NS(\cEbar)$).
Since $\Coker\, \cl_\ell$ is
torsion free, 
as noted above, the homomorphism
$\cl_\ell^*$ is surjective,
hence we have isomorphisms
$$
\begin{array}{l}
\Coker\, \alpha^*
\cong \Coker\, \iota^*
\cong \Ext^1_{\Z}(\NS(\cEbar)/M, \Z_\ell(-1)) \\
\cong \Hom_{\Z}(\NS(\cEbar)/M, \Q_\ell/\Z_\ell(-1))
= \Hom_{\Z}(\Div(\cEbar_U)/\sim_\alg , \Q_\ell/\Z_\ell(-1)).
\end{array}
$$
Given the above, we have claim (3).
\end{proof}

\begin{cor}\label{cor:open_L}
For $i \neq 3$, the group 
$H^i_{c,\etale}(\cEbar_U,\Z_\ell)$ is torsion free
and the group $H^3_{c,\etale}(\cEbar_U,\Z_\ell)_\tors$ 
is canonically isomorphic to the group 
$\Hom_{\Z}(T_U,\Q_\ell/\Z_\ell(-1))$.
We set
$$
L(h^i_{c,\ell}(\cEU),s)
= \det(1-\Frob\cdot q^{-s}; H^i_{c,\etale}(\cEbar_U,\Q_\ell)).
$$
Then if $U \neq C$, we have
$$
L(h^i_{c,\ell}(\cEU),s) = \left\{
\begin{array}{ll}
1, & \text{ if }i\le 0\text{ or }i\ge 5, \\
\frac{L(h^1(C),s) L(h^0(C\setminus U),s)}{1-q^{-s}} & 
\text{ if }i=1, \\
\frac{L(h^2(\cE),s)L(h^1(\cE^U),s)L(h^0(C\setminus U),s-1)}{
(1-q^{1-s}) L(h^2(\cE^U),s)} & \text{ if }i=2, \\
\frac{L(h^1(C),s-1)L(h^0(C\setminus U),s-1)}{
1-q^{1-s}}& \text{ if }i=3, \\
1-q^{2-s}, & \text{ if }i=4. \\
\end{array}
\right.
$$
\end{cor}
\begin{proof}
The above follows from Lemmas~\ref{lem:appl_LofcE} 
and \ref{lem:appl_coker}, as well as 
the long exact sequence
$$
\cdots \to H^{i}_{\etale,c}(\cEbar_U,\Z_\ell)
\to H^{i}_{\etale}(\cEbar,\Z_\ell)
\to H^{i}_{\etale}(\overline{\cE^U},\Z_\ell)
\to \cdots.
$$
\end{proof}
\begin{rmk}\label{rmk:ell_indep}
Corollary~\ref{cor:open_L} in particular shows that
the function $L(h^i_{c,\ell}(\cEU),s)$ is independent of
$\ell \neq p$. We can show the
$\ell$-independence of $L(h^i_{c,\ell}(X),s)
= \det(1-\Frob\cdot q^{-s}; H^i_{c,\etale}(X,\Q_\ell))$
for any normal surface $X$ over $\F_q$, which is
not necessarily proper.  Since we will not need this further, 
we only show a sketch here.  
There is a proper smooth surface $X'$ and  
a closed subset $D\subset X'$ of pure codimension one such
that $X=X' \setminus D$.  We can
express the cokernel and kernel of 
the restriction map
$H_\etale^1(\Xbar', \Q_l) \to H_\etale^1(\Dbar, \Q_l)$
in terms of $\Pic_{X'/\F_q}$ and the Jacobian of the normalization of 
each irreducible component of $D$.  Then, we apply the same method 
as above to obtain the result.
\end{rmk}

\begin{cor}\label{cor:etaleU_1}
Suppose that $U \neq C$. Then, we have the following.
\begin{enumerate}
\item $H^i_{\etale}(\cEU,\Q_\ell/\Z_\ell(j))$ is zero for
$i \le -1$ or $i\ge 5$.
\item For $j\ne 0$, the group 
$H^0_\etale(\cEU,\Q_\ell/\Z_\ell(j))$ is isomorphic to
$\Z_\ell/(q^j-1)$, and 
$H^0_\etale(\cEU,\Q_\ell/\Z_\ell(0)) = \Q_\ell/\Z_\ell$.
\item For $j \neq 0,1$, the group
$H^1_\etale(\cEU,\Q_\ell/\Z_\ell(j))$ is finite of
order 
$$
\frac{|T'_{U,(j-1)}|_\ell^{-1} 
\cdot |L(h^1(C),1-j)L(h^0(C\setminus U),1-j)|_\ell^{-1}}{
|q^{j-1}-1|_\ell^{-1}}.
$$
The group $H^1_\etale(\cEU,\Q_\ell/\Z_\ell(0))$ 
is isomorphic to the direct sum of $\Q_\ell/\Z_\ell$
and a finite group of order 
$$
\frac{|T'_{U,(-1)}|_\ell^{-1} 
\cdot |L(h^1(C),1)L(h^0(C\setminus U),1)|_\ell^{-1}}{
|q-1|_\ell^{-1}}.
$$
\item For $j \neq 1,2$, the cohomology group
$H^2_\etale(\cEU,\Q_\ell/\Z_\ell(j))$ is finite of
order 
\begin{small}
$$
\frac{|T'_{U,(j-1)}|_\ell^{-1} \cdot 
|L(h^2(\cE),2-j) L(h^1(\cE^U),2-j) 
L(h^0(C\setminus U),1-j)|_\ell^{-1}}{
|(q^{j-1}-1) L(h^2(\cE^U),2-j)|_\ell^{-1}}.
$$
\end{small}
\item For $j \neq 1,2$, the group
$H^3_\etale(\cEU,\Q_\ell/\Z_\ell(j))$ is finite of
order 
$$
\frac{|L(h^1(C),2-j)L(h^0(C\setminus U),2-j)|_\ell^{-1}}{
|q^{j-2}-1|_\ell^{-1}}.
$$
The cohomology 
group $H^3_\etale(\cEU,\Q_\ell/\Z_\ell(1))$ is isomorphic to
the direct sum of $(\Q_\ell/\Z_\ell)^{\oplus |C\setminus U|-1}$
and a finite group of order 
$$
\frac{|L(h^1(C),1)L(h^0(C\setminus U),1)|_\ell^{-1}}{
|q-1|_\ell^{-1}}.
$$
\item For $j \neq 2$ (\resp $j=2$), the group 
$H^4_\etale(\cEU,\Q_\ell/\Z_\ell(j))$ is zero
(\resp is isomorphic to 
$(\Q_\ell/\Z_\ell)^{\oplus |C \setminus U| -1}$).
\end{enumerate}
\end{cor}
\begin{proof}
The cohomology group $H^i_{\etale}(\cE_U,\Q_\ell/\Z_\ell(j))$
is the Pontryagin dual of the group
$H^{5-i}_{\etale,c}(\cE_U,\Z_\ell(2-j))$. The claims
follow from 
Corollary~\ref{cor:open_L} and the short exact sequence
$$
\begin{array}{l}
0 \to H^{4-i}_{c,\etale}(\cEbar_U,\Z_\ell(2-j))_{\GFq}
\to H^{5-i}_{c,\etale}(\cE_U,\Z_\ell(2-j)) \\
\to H^{5-i}_{c,\etale}(\cEbar_U,\Z_\ell(2-j))^{\GFq}
\to 0.
\end{array}
$$
\end{proof}

\begin{lem}\label{lem:etaleU_2}
Suppose that $U \neq C$. 
Then, $H^2_{\etale}(\cEU,\Q_\ell/\Z_\ell(2))^\cotors$
is finite of order 
$$
\frac{|T'_{U,(1)}|_\ell^{-1} \cdot |L(h^2(\cE),0) 
L^{*} (h^1(\cE^U),0) L(h^0(C \setminus U),-1)|_\ell^{-1}}{
|(q-1) L(h^0(\Irr(\cE^U)),-1)|_\ell^{-1}}. 
$$
\end{lem}
\begin{proof}
We note that the group 
$H^2_{\etale}(\cE_U,\Q_\ell/\Z_\ell(2))^\cotors$
is 
canonically isomorphic to the group
$H^3_{\etale}(\cE_U,\Z_\ell(2))_\tors$.
Let us consider the long exact sequence
$$
\cdots \to H^i_{\cE^U,\etale}(\cE,\Z_\ell(2)) 
\xto{\mu_i} H^i_\etale(\cE,\Z_\ell(2))
\to H^i_\etale(\cEU,\Z_\ell(2)) \to \cdots.
$$

The group $\Ker\, \mu_4$
is isomorphic to the Pontryagin dual of
the cokernel of the homomorphism
$H^1_\etale(\cE,\Q_\ell/\Z_\ell) \to 
H^1_\etale(\cE^U,\Q_\ell/\Z_\ell)$. 
By Lemma~\ref{lem:appl_pi1}, this
homomorphism factors through 
$H^1_\etale(C\setminus U,\Q_\ell/\Z_\ell) 
\to H^1_\etale(\cE^U,\Q_\ell/\Z_\ell)$.
In particular, the group $(\Ker\, \mu_4)_\tors$ is 
isomorphic to the Pontryagin dual of 
$(H^1_\etale(\overline{\cE^U},
\Q_\ell/\Z_\ell)^{\GFq})^\red$.
By the weight argument, we observe that $\Coker\, \mu_3$ 
is a finite group. It then follows that
$$
|H^3_{\etale}(\cE_U,\Z_\ell(2))_\tors|
= |L^{*}(h^1(\cE^U),0)| \cdot |\Coker\, \mu_3|.
$$ 
Let $\mu'$ denote
the homomorphism
$H^2_{\overline{\cE^U},\etale}(\cEbar,
\Z_\ell(2)) \to H^2_\etale(\cEbar,\Z_\ell(2))$. 
We have an exact sequence
\begin{equation}\label{eqn:keyexactseq}
\begin{small}
\Ker\, \mu_3 \to H^3_{\overline{\cE^U},\etale}(\cEbar,
\Z_\ell(2))^{\GFq} \to
(\Coker\, \mu')_{\GFq} 
\to \Coker\, \mu_3 \to 0.
\end{small}
\end{equation}
Since $\Ker\, \mu_3 \cong
\Coker[H^2_\etale (\cE,\Z_\ell(2)) \to 
H^2_\etale (\cEU,\Z_\ell(2))]$,
the cokernel of $\Ker\, \mu_3 \to 
H^3_{\overline{\cE^U},\etale}(\cEbar,
\Z_\ell(2))^{\GFq}$ is isomorphic to the
cokernel of the homomorphism
$$
\nu' : H^2_\etale (\cEbar_U,\Z_\ell(2))^{\GFq}
\to H^3_{\overline{\cE^U},\etale}(\cEbar,
\Z_\ell(2))^{\GFq}.
$$
Consider the following diagram with exact rows
$$
\begin{CD}
0 @>>> \Coker\, \mu' @>>> 
H^2_\etale(\cEbar_U,\Z_\ell(2)) @>{\nu}>>
H^3_{\overline{\cE^U},\etale}(\cEbar,
\Z_\ell(2)) \\
@. @V{1-\Frob}VV @V{1-\Frob}VV @V{1-\Frob}VV \\
0 @>>> \Coker\, \mu' @>>> 
H^2_\etale(\cEbar_U,\Z_\ell(2)) @>{\nu}>>
H^3_{\overline{\cE^U},\etale}(\cEbar,
\Z_\ell(2)).
\end{CD}
$$
Since $(\Coker\, \nu)^{\GFq} \subset 
H^3_\etale(\cEbar,\Z_\ell(2))^{\GFq} =0$,
$\Coker\, \nu'$ is isomorphic to the kernel of
$(\Coker\, \mu')_{\GFq} \to 
H^2_{\etale}(\cEbar_U,\Z_\ell(2))_{\GFq}$,
hence by (\ref{eqn:keyexactseq}), $|\Coker\, \mu_3|$ 
equals the order of 
$$
\begin{array}{rl}
M'' & = \Image[
(\Coker\, \mu')_{\GFq} \to 
H^2_{\etale}(\cEbar_U,\Z_\ell(2))_{\GFq}] \\
& = \Image[H^2_{\etale}(\cEbar,\Z_\ell(2))_{\GFq}
\to H^2_{\etale}(\cEbar_U,\Z_\ell(2))_{\GFq}].
\end{array}
$$
Next, we set $M'= \Image[H^2_{\etale}(\cEbar,\Z_\ell(2))
\to H^2_{\etale}(\cEbar_U,\Z_\ell(2))]$.
From the commutative diagram with exact rows
\begin{equation}\label{eqn:NS}
\begin{footnotesize}
\begin{array}{rcccccl}
0 \to & \NS(\cEbar)\otimes_\Z \Z_\ell
& \to & H^2_\etale(\cEbar,\Z_\ell(1)) 
& \to & T_\ell H^2_\etale (\cEbar,\Gm)
& \to  0 \\
& \downarrow & & \downarrow & & \downarrow & \\
0 \to & (\Div(\cEbar_U)/\sim_\alg)\otimes_\Z \Z_\ell
& \to & H^2_\etale(\cEbar_U,\Z_\ell(1)) 
& \to & T_\ell H^2_\etale (\cEbar_U,\Gm) 
& \to 0
\end{array}
\end{footnotesize}
\end{equation}
and the exact sequence 
$$
0 \to H^2_\etale(\cEbar,\Gm) \to 
H^2_\etale(\cEbar_U,\Gm) \to H^1_\etale(\overline{\cE^U},\Q/\Z),
$$
we obtain an exact sequence
$$
0 \to M' \to
H^2_\etale (\cEbar_U,\Z_\ell(2)) \to
T_\ell H^1_\etale (\overline{\cE^U},\Q_\ell/\Z_\ell(1)).
$$
By the weight argument,
we obtain $(T_\ell H^1_\etale (\overline{\cE^U},
\Q_\ell/\Z_\ell(1)))^{\GFq} = 0$,
hence the canonical surjection $M'_{G_{\F_q}} \to M''$ 
is an isomorphism.
From (\ref{eqn:NS}), we have an exact sequence
$$
0 \to (\Div(\cEbar_U)/\sim_\alg)\otimes_\Z \Z_\ell(1)
\to M' \to T_\ell H^2_\etale(\cEbar,\Gm)(1)
\to 0.
$$
By the weight argument, 
we yield $(T_\ell H^2_\etale(\cEbar,
\Gm)(1))^{\GFq} =0$,
hence 
$$
\begin{array}{rl}
0 & \to 
((\Div(\cEbar_U)/\sim_\alg)\otimes_\Z 
\Z_\ell(1))_{\GFq}
\to M'_{\GFq} \\
& \to
(T_\ell H^2_\etale(\cEbar, \Gm)(1))_{\GFq} \to 0
\end{array}
$$
is exact. Therefore, $|\Coker\, \mu_3|=|M'_{\GFq}|$
equals
$$
\frac{|(T_U \otimes_\Z \Z_\ell(1))_{\GFq}|
\cdot |\det(1-\Frob;\, 
H^2_\etale(\cEbar,\Q_\ell(2)))|_\ell^{-1}}{
|\det(1-\Frob;\, 
\Ker[\NS(\cEbar)\to \Div(\cEbar_U)/\sim_\alg] 
\otimes_\Z \Q_\ell(1))|_\ell^{-1}}.
$$
This proves the claim.
\end{proof}

\subsection{Cohomology of the fibers}
Fix a non-empty open subscheme $U \subset C$.
Let $f^U :\cE^U \to C \setminus U$
denote the structure morphism and let
$\iota^U : C\setminus U \to \cE^U$ 
denote the morphism induced from $\iota:C\to \cE$.

\begin{lem}\label{lem:G1_isom}
The homomorphism
$$
(\chern'_{1,1},f^U_{*}): G_1(\cE^U) \to
H^1_\cM(\cE^U,\Z(1)) \oplus K_1(C \setminus U)
$$ 
is an isomorphism.
\end{lem}
\begin{proof}
The morphism $f^U :\cE^U \to C \setminus U$ has
connected fibers, hence the claim follows 
from Proposition~\ref{prop:Z} and 
the construction of $\chern'_{1,2}$.
\end{proof}

\begin{lem}\label{lem:G0_order}
The group $H^2_\cM(\cE^U,\Z(1))$ is 
finitely generated of rank $|C \setminus U|$. 
Moreover, 
$H^2_\cM(\cE^U,\Z(1))_\tors$ is of order $|L^{*}(h^1(\cE^U),0)|$.
\end{lem}
\begin{proof}
It suffices to prove the following claim: 
if $E$ has good reduction (\resp non-split multiplicative
reduction, \resp split multiplicative or additive reduction) 
at $\wp \in C_0$, then $H^2_\cM(\cE_\wp,\Z(1))$ is
a finitely generated abelian group of
rank one, and $|H^2_\cM(\cE_\wp,\Z(1))_\tors|$ 
equals $|\cE_\wp(\kappa(\wp))|$ (\resp rank two, \resp rank one).
We set $\cE_{\wp,(0)} = 
(\cE_{\wp,\red})_\sm \setminus \iota(\wp)$ and 
$\cE_{\wp,(1)} = \cE^U \setminus \cE_{\wp,(0)}$.
We then have an exact sequence
$$
\begin{array}{l}
H^1_\cM(\cE_{\wp,(0)},\Z(1)) \to H^0_\cM(\cE_{\wp,(1)},\Z(0)) \\
\to H^2_\cM(\cE_\wp,\Z(1)) \to \Pic(\cE_{\wp,(0)}) \to 0.
\end{array}
$$
First, suppose that $E$ does not have non-split 
multiplicative reduction at $\wp$ or
that $E$ has non-split multiplicative
reduction at $\wp$ and $\cE_\wp \otimes_{\kwp}\Fbar_q$ has
an even number of irreducible components. 
Then, using the classification of Kodaira, N\'eron, and Tate
(\cf \cite[10.2.1, p.~484--489]{Liu}) of singular 
fibers of $\cE \to C$,
we can verify the equality
$$
\begin{array}{rl}
& \Image [H^0_\cM(\cE_{\wp,(1)},\Z(0)) \to H^2_\cM(\cE_\wp,\Z(1))] \\
= & \Image [\iota_*: H^0_\cM(\Spec\, \kwp,\Z(0)) 
\to H^2_\cM(\cE_\wp,\Z(1))].
\end{array}
$$
This shows that the group $H^2_\cM(\cE_\wp,\Z(1))$ is isomorphic
to the direct sum of Picard group $\Pic(\cE_{\wp,(0)})$ and 
$H^0_\cM(\Spec\, \kwp,\Z(0)) \cong \Z$. 
In particular, we have
$H^2_\cM(\cE_\wp,\Z(1))_\tors \cong \Pic(\cE_{\wp,(0)})$,
from which we easily deduce the claim.

Next, suppose that $E$ has non-split multiplicative
reduction at $\wp$ and $\cE_\wp \otimes_{\kwp}\Fbar_q$ has
an odd number of irreducible components. In this case,
we can directly verify that the image of
$H^0_\cM(\cE_{\wp,(1)},\Z(0)) \to H^2_\cM(\cE_\wp,\Z(1))$
is isomorphic to $\Z \oplus \Z/2$ and
$\Pic(\cE_{\wp,(0)}) =0$.  The claim in this case follows.
\end{proof}

\begin{lem}\label{lem:aux2}
The diagram
$$
\begin{CD}
K_1(E) @>>> G_0(\cE^U)\\
@V{\iota^*}VV @V{\iota^{U *}}VV \\
K_1(k) @>>> K_0(C \setminus U)
\end{CD}
$$
is commutative.
\end{lem}
\begin{proof}
The group $K_1(E)$ is generated by the
image of $f^* :K_1(k) \to K_1(E)$ and the
image of $\bigoplus_{x \in E_0} K_1(\kappa(x))
\to K_1(E)$. The claim follows from the fact
that the localization sequence
in $G$-theory commutes with flat pull-backs
and finite push-forwards.
\end{proof}

\subsection{Proofs of Theorems \ref{thm:conseq1}, 
\ref{thm:conseq2}, and \ref{thm:conseq3}}
\label{subsec:appl_proofs}

\begin{lem}\label{lem:coker_partial}
For any non-empty open subscheme $U \subset C$, the cokernel
of the boundary map $\partial_U:H^2_\cM(\cEU,\Z(2))
\to H^1_\cM(\cE^U,\Z(1))$ is finite.
\end{lem}
\begin{proof}
If suffices to prove the claim for sufficiently
small $U$, hence we assume that $\cE_U \to U$
is smooth. Since $K_2(\cEU)_\Q \to K_2(E)_\Q$ is an 
isomorphism in this case, the claim follows from 
Theorem~\ref{surjectivity} and Lemma~\ref{lem:aux1}.
\end{proof}

\begin{proof}[Proof of Theorem~\ref{thm:conseq3}]
Claims (1) and (2) follow from
Theorem~\ref{thm:motfin}, 
Proposition~\ref{prop:red_loc_seq}, and Lemma~\ref{lem:coker_partial}.
Proposition~\ref{prop:red_loc_seq}
gives an exact sequence
$$
\begin{array}{rl}
& 0 \to H^2_\cM(\cE,\Z(2))_\tors 
\to H^2_\cM(\cEU,\Z(2))^\cotors  \\
\xto{\partial^2_U} & H^1_\cM(\cE^U,\Z(1))
\to H^3_\cM(\cE,\Z(2))_\tors 
\to H^3_\cM(\cEU,\Z(2))^\cotors \\
\xto{\partial^3_U} & H^2_\cM(\cE^U,\Z(1))
\to \CH_0(\cE)
\to \CH_0(\cEU) \to 0.
\end{array}
$$
From Lemma~\ref{lem:coker_partial},
it follows that $\Coker\, \partial^2_{U}$ is a finite group,
which implies that the group $H^2_\cM(\cEU,\Z(2))^\cotors$
is of rank $|S_0 \setminus U|$.
By Theorem~\ref{thm:motfin},
$|H^2_\cM(\cEU, \Z(2))_\tors|$ equals
$$
\prod_{\ell \neq p} |H^1_\etale(\cEU,\Q_\ell/\Z_\ell(2))|.
$$
By Corollaries~\ref{cor:appl_UC} and \ref{cor:etaleU_1}, 
this equals 
$$
|T'_{U,(1)}| \cdot |L(h^1(C),-1)L(h^0(C\setminus U,-1)/(q-1)|.
$$ 
The above proves claim (3).

As noted in the proof of Theorem~\ref{thm:motfin} (1),
the group $\CH_0(\cE)$ is a finitely
generated abelian group of rank one and 
$\CH_0(\cEU)$ is finite if $U \neq C$.
By Lemma~\ref{lem:G0_order},
$H^2_\cM(\cE^U,\Z(1))$ is a finitely generated abelian group 
of rank $|C\setminus U|$,
hence the rank of $H^3_\cM(\cEU,\Z(2))^\cotors$ 
equals $\max(|C\setminus U|-1,0)$.

From the class field theory of varieties over finite fields
\cite[Theorem 1, p.~242]{Ka-Sa} (see also \cite[p.~283--284]{Co-Ra})
and Lemma~\ref{lem:appl_pi1}, it follows that
the push-forward map $\CH_0(\cE) \to \Pic(C)$ is an
isomorphism,
hence the homomorphism
$H^2_\cM(\cE^U,\Z(1)) \to \CH_0(\cE) \cong \Pic(C)$
factors through the push-forward map
$f^U_*:H^2_\cM(\cE^U,\Z(1)) \to H^0_\cM(C\setminus U,\Z(0))$.
By the surjectivity of $f^U_*$, we have isomorphisms
$$
\CH_0(\cE^U) \cong \Coker[H^0_\cM(C\setminus U,\Z(0)) \to \Pic(C)]
\cong \Pic(U),
$$ 
which proves claim (6).
Since the group $H^0_\cM(C\setminus U,\Z(0))$ is torsion free,
the image of $H^2_\cM(\cE^U,\Z(1))_\tors$ in
$\CH_0(\cE)$ is zero, thus we have an exact sequence
$$
\begin{array}{l}
0 \to \Coker\, \partial^2_U \to H^3_\cM(\cE,\Z(2))_\tors \\
\to H^3_\cM(\cEU,\Z(2))_\tors \to H^2_\cM(\cE^U,\Z(1))_\tors
\to 0.
\end{array}
$$
By Proposition~\ref{prop:red_loc_seq} and
Lemma~\ref{lem:etaleU_2}, the group 
$H^3_\cM(\cEU,\Z(2))_\tors$ is finite of order
$$
\frac{p^m |T'_{U,(1)}| \cdot |L(h^2(\cE),0) L^*(h^1(\cE^U),0)
L(h^0(C \setminus U),-1)|}{(q-1) |L(h^0(\Irr(\cE^U)),-1)|}
$$
for some $m \in \Z$. By Lemma~\ref{lem:G0_order},
the group 
$H^2_\cM(\cE^U,\Z(1))_\tors$ is finite of order
$|L^{*}(h^1(\cE^U),0)|$.
By Lemma~\ref{lem:appl_pi1}, the Picard scheme
$\Pic^o_{\cE/\F_q}$ is an abelian variety
and, 
in particular, 
$\Hom(\Pic^o_{\cE/\F_q},\Gm)= 0$,
hence by Theorem~\ref{thm:motfin} and 
Corollary~\ref{cor:appl_UC}, 
the group $H^3_\cM(\cE,\Z(2))_\tors$ is of order
$|L(h^2(\cE),0)|$.
Therefore, 
$$
\begin{array}{rl}
|\Coker\, \partial^2_U |
& = \frac{|H^3_\cM(\cE,\Z(2))_\tors|
\cdot |H^2_\cM(\cE^U,\Z(1))_\tors|}{
|H^3_\cM(\cEU,\Z(2))_\tors|} \\
& = \frac{p^{-m} (q-1) |L(h^0(\Irr(\cE^U)),-1)|}{
|T'_{U,(1)}| \cdot |L(h^0(C \setminus U),-1)|}.
\end{array}
$$
Since $|\Coker\, \partial^2_U |$ is prime to $p$,
we have $m=0$.
This proves claims (4) and (5) and 
completes the proof.
\end{proof}

\begin{proof}[Proof of Theorem~\ref{thm:conseq2}]
Claim (5) is clear. Claim (1) follows from
Corollary~\ref{cor:relative_cv} and Theorem~\ref{surjectivity}. 
We easily verify that 
$H^i_\cM(\cE^U,\Z(1))$ is zero for $i \le 0$.
By the localization sequence of higher Chow groups
(\cf \cite[Corollary (0.2), p.~537]{Bloch2}), we have 
$H^i_\cM(\cE,\Z(2)) \cong H^i_\cM(\cE_U,\Z(2))$
for $i \le 1$. Taking the inductive limit with respect to $U$,
we obtain claim (2).

By Corollary~\ref{cor:relative_cv},
we have an exact sequence
\begin{equation}
\label{eqn:longexact_conseq2}
\begin{array}{rl}
& 0 \to H^2_\cM(\cE,\Z(2))_\tors 
\xto{\alpha} H^2_\cM(E,\Z(2))^\cotors \\
\xto{\partial^2_{\cM,2}} &
 {\displaystyle \bigoplus_{\wp \in C_0}}
H^1_\cM(\cE_\wp,\Z(1)) \to H^3_\cM(\cE,\Z(2))_\tors
\to H^3_\cM(E,\Z(2))^\cotors  \\
\xto{\partial^3_{\cM,2}} &
{\displaystyle \bigoplus_{\wp \in C_0}} 
H^2_\cM(\cE_\wp,\Z(1))
\to \Pic(C) \to 0.
\end{array}
\end{equation}
Hence by Theorem~\ref{thm:motfin} and 
Corollary~\ref{cor:appl_UC}, the group 
$\Ker\, \partial^2_{\cM,2}$ is finite of order
$|L(h^{1}(C),-1)|$.
For a non-empty open subscheme $U \subset C$,
consider the group $\Coker\, \partial^2_U$
in the proof of Theorem~\ref{thm:conseq3}.
For two non-empty open subschemes $U',U \subset C$ 
with
$U' \subset U$, the homomorphism 
$\Coker\, \partial^2_U \to \Coker\, 
\partial^2_{U'}$ is injective since
both $\Coker\, \partial^2_U$ and
$\Coker\, \partial^2_{U'}$ canonically 
inject into $H^3_\cM(\cE,\Z(2))_\tors$.
Claim (3) follows from claim
(4) of Theorem~\ref{thm:conseq3} by taking the 
inductive limit.
Claim (4) follows from 
exact sequence (\ref{eqn:longexact_conseq2}) 
and Lemma~\ref{lem:appl_LofcE}.

From the localization sequence, it follows that the
push-forward homomorphism
$\bigoplus_{x \in E_0} H^2_\cM(\Spec\, \kappa(x),\Z(2))
\to H^4_\cM(E,\Z(3))$
is surjective, 
hence $H^4_\cM(E,\Z(3))$ is a torsion group
and claim (6) follows from Lemma~\ref{lem:coker_del43}.
This completes the proof.
\end{proof}

\begin{proof}[Proof of Theorem~\ref{thm:conseq1}]
Consider the restriction 
$\gamma: \Ker\, c_{2,3} \to H^2_\cM(E,\Z(2))$ of $c_{2,2}$
to $\Ker\, c_{2,3}$. By Lemma~\ref{lem:c12},
both $\Ker\, \gamma$ and $\Coker\, \gamma$ are
annihilated by the multiplication-by-$2$ map, 
which implies that the image of $\gamma$
contains $H^2_\cM(E,\Z(2))_\divi$ and that 
$\Ext^1_\Z(H^2_\cM(E,\Z(2))_\divi, \Ker\, \gamma)$
is zero. From this, it follows that 
the homomorphism $\gamma$ induces 
an isomorphism $(\Ker\, c_{2,3})_\divi \xto{\cong} 
H^2_\cM(E,\Z(2))_\divi$,
which shows that the homomorphism
$K_2(E)^\cotors \to H^2_\cM(E,\Z(2))^\cotors$ 
induced by $c_{2,2}$ is surjective with torsion kernel,
thus claim (1) follows from Theorem~\ref{thm:conseq2} (3).

Claim (3) follows from
Theorem~\ref{thm:conseq2} (1) and Lemma~\ref{lem:c12}.

For $\wp \in C_0$, let
$\iota_\wp :\Spec\, \kwp \to \cE_\wp$ 
denote the reduction at $\wp$ of the morphism 
$\iota:C\to \cE$. 
Diagram (\ref{eqn:CD_H43}) gives
an exact sequence
$$
\Coker\, \partial^4_{\cM,3}
\to \Coker \, \partial_2
\to \Coker \, \partial^2_{\cM,2} \to 0.
$$
By Lemma~\ref{lem:coker_del43}, we have
an isomorphism $\Coker\, \partial^4_{\cM,3}
\cong \F_q^\times$. 
By the construction of this isomorphism,
we see that the composition
$$
\F_q^\times \cong \Coker\, \partial^4_{\cM,3} 
\to \Coker \,\partial_2 \inj K_1(\cE)
\to K_1(\Spec\, \F_q) \cong \F_q^\times
$$
equals the identity, hence the map
$\Coker\, \partial^4_{\cM,3} \to \Coker \,\partial_2$ 
is injective. Then, claim (2) follows from
Theorem~\ref{thm:conseq2} (3).

From Proposition~\ref{prop:Z} and Lemmas~\ref{lem:c12},
\ref{lem:aux1}, and \ref{lem:aux2}, 
it follows that the homomorphism
$\partial_1 :K_1(E)^\cotors \to 
\bigoplus_{\wp \in C_0} G_0(\cE_\wp)$
is identified with the direct sum of the map
$\partial'_1: k^\times \to \bigoplus_{\wp \in C_0} 
H^0_\cM(\Spec\, \kwp, \Z(0)) \to 
\bigoplus_{\wp} H^0_\cM(\cE_\wp, \Z(0))$ and the map
$\partial^3_{\cM,2} : H^3_\cM(E,\Z(2))^\cotors \to
\bigoplus_{\wp} H^2_\cM(\cE_\wp,\Z(1))$.
We then have isomorphisms 
$$
\begin{array}{l}
\Ker\, \partial'_1 \cong \F_q^\times,\ 
\Coker\, \partial'_1 \cong \Pic(C) \oplus 
\bigoplus_{\wp} \Z^{|\Irr(\cE_\wp)|-1}, \\
\Ker\, \partial^3_{\cM,2} \cong
H^3_\cM(\cE,\Z(2))_\tors /\Coker\, \partial^2_{\cM,2},\ 
\Coker\, \partial^3_{\cM,2} \cong \Pic(C).
\end{array}
$$  
Claim (4) follows,
which completes the proof 
of Theorem~\ref{thm:conseq1}.
\end{proof}
\section{Results for $j \ge 3$}
\label{sec:BlKa}
In this section, we obtain results for $j \ge 3$, 
generalizing the theorems of Section~\ref{conseq}.
The proofs here are simpler than those of Section~\ref{conseq}
in that we do not use 
tools such as the class field theory of Kato-Saito~\cite{Ka-Sa} or Theorem~\ref{surjectivity}.
We also refer the reader to Section~\ref{sec:intro high genus} for 
remarks concerning the contents of this section.
Finally, we note that the notation we use is 
as in Section \ref{conseq}. 

\subsection{Statements}
For integers $i,j$, consider the boundary map
$$
\partial^i_{\cM,j}: 
H^{i}_\cM(E,\Z(j))^\cotors \to \bigoplus_{\wp \in C_0} 
H^{i-1}_\cM(\cE_\wp,\Z(j-1)).
$$

\begin{thm}
\label{thm:conseq5}
Let $j \ge 3$ be an integer. 
\begin{enumerate}
\item For any $i \in \Z$, both
$\Ker\, \partial^i_{\cM,j}$ and $\Coker\, \partial^i_{\cM,j}$
are finite groups.
\item We have
$$
|\Ker \partial^i_{\cM,j}|=
\left\{ \begin{array}{ll} 1, & 
\text{ if }i \le 0 \text{ or }i \ge 5,\\
q^j-1, & \text{ if } i=1, \\
|L(h^1(C),1-j)|, & \text{ if } i=2, \\
\frac{|T'_{(j-1)}|\cdot |L(h^2(\cE),2-j)|}{q^{j-1}-1},
& \text{ if } i=3, \\
|L(h^1(C),2-j)|, & \text{ if }i=4.
\end{array}\right.
$$
Further, 
the group $\Ker \,\partial^1_{\cM,j}$ is cyclic 
of order $q^j-1$.
\item We have
$$
|\Coker\, \partial^i_{\cM,j}|=
\left\{ \begin{array}{ll} 1, & \text{ if }
i \le 1,\ i=3,\text{ or }i \ge 5,\\
\frac{q^{j-1}-1}{|T'_{(j-1)}|}, & 
\text{ if } i=2, \\
q^{j-2}-1
& \text{ if }i=4.
\end{array}\right.
$$
\item Let $U \subset C$ be a non-empty
open subscheme. Then, the group
$H^i_\cM(\cE_U,\Z(j))$ is finite modulo 
a uniquely divisible subgroup for any $i\in \Z$.
The group $H^i_\cM(\cE_U,\Z(j))$ is zero if $i\ge \max(6,j)$
and is finite for $(i,j)=(4,3),(5,3)$, $(4,4)$, $(5,4)$, or $(5,5)$.
\item $H^i_\cM(\cE_U,\Z(j))$ 
is uniquely divisible for $i \le 0$ or $6 \le i \le j$,
and $H^1_\cM(\cE_U,\Z(j))_\tors$ 
is cyclic of order $q^j -1$.
\item Suppose that $U=C$ (\resp $U \neq C$).
The group $H^2_\cM(\cE_U,\Z(j))_\tors$ 
is of order $|L(h^1(C),1-j)|$ (\resp of order
$$
\frac{|T'_{U,(j-1)}| 
\cdot |L(h^1(C),1-j)L(h^0(C\setminus U),1-j)|}{
q^{j-1}-1}).
$$
The group $H^3_\cM(\cE_U,\Z(j))_\tors$ 
is of order $|L(h^2(\cE),2-j)|$ (\resp of order
$$
\frac{|T'_{U,(j-1)}| 
\cdot |L(h^2(\cE),2-j) L(h^1(\cE^U),2-j) 
L(h^0(C\setminus U),1-j)|}{
(q^{j-1}-1) |L(h^0(\Irr(\cE^U)),1-j)|}).
$$
The group $H^4_\cM(\cE_U,\Z(j))_\tors$ 
is of order $|L(h^1(C),2-j)|$ (\resp of order
$$
\frac{|L(h^1(C),2-j)L(h^0(C\setminus U),2-j)|}{
q^{j-2}-1}).
$$
The group $H^5_\cM(\cE_U,\Z(j))_\tors$ 
is cyclic of order $q^{j-2}-1$ (\resp is zero).
\end{enumerate}
\end{thm}

\begin{thm}\label{thm:conseq6}
The following statements hold.
\begin{enumerate}
\item 
The group $K_2(E)_\divi$ 
is uniquely divisible and
$c_{2,2}$ induces
an isomorphism
$K_2(E)_\divi \cong H^2_\cM(E,\Z(2))_\divi$.
\item 
The kernel of the boundary map
$\partial: K_2(E)^\red \to \bigoplus_{\wp \in C_0}
G_1(\cE_\wp)$ 
is a finite group
of order 
$|L(h^1(C),-1)|^2$.
\end{enumerate}
\end{thm}

\subsection{Lemmas}
\begin{lem}\label{lem:Milnor_K}
Let $X$ be a smooth projective geometrically 
connected curve over a global field $k'$. 
Let $k'(X)$ denote the function field of $X$.
Then, the Milnor $K$-group $K_n^M(k'(X))$
is torsion for $n\ge 2 + \gon(X)$, and
is of exponent two (\resp zero) for
$n \ge 3 + \gon(X)$ if $\chara(k')=0$ (\resp $\chara(k')>0$).
Here, $\gon(X)$ denotes the gonality of $X$, i.e.,
the minimal degree of morphisms from $X$ to $\mathbb{P}^1_{k'}$.
\end{lem}
\begin{proof}
The field $k'(X)$ is an extension of degree $\gon(X)$
of a subfield $K$ of the form $K=k'(t)$.
From the split exact sequence
$$
0 \to K_n^M(k') \to K_n^M(K) \to 
\bigoplus_{P} K_{n-1}^M(k'[t]/P) \to 0
$$
in \cite[Theorem 2.3, p.~325]{Milnor},
where $P$ runs over the irreducible monic polynomials
in $k'[t]$, and using \cite[Chapter II, (2.1), p.~396]{Bass-Tate},
we observe that $K_n^M(K)$ is torsion for $n\ge 3$ and
is of exponent two (\resp zero) for
$n \ge 4$ if $\chara(k')=0$ (\resp $\chara(k')>0$).

Next, consider a flag $K=V_1 \subset V_2 \subset \cdots 
\subset V_{\gon(X)}=k'(X)$
of $K$-subspaces of $k'(X)$ with $\dim_K V_i =i$.
For each $i$, we set $V^*_i = V_i \setminus \{0\}$.
Suppose $i \ge 2$ and consider two elements 
$\alpha,\beta \in V_i \setminus V_{i-1}$. 
Then, there exist $a,b \in K^\times$ such that
$\gamma = a \alpha + b \beta \in V_{i-1}$. 
If $\gamma =0$ (\resp $\gamma \neq 0$), 
then $\{a\alpha, b \beta \}=0$ 
(\resp $\{a\alpha /\gamma, b \beta /\gamma \}=0$) 
in $K_2^M(k'(X))$.
Expanding this equality, we obtain 
an expression for $\{\alpha, \beta\}$.
We observe that 
$\{\beta,\gamma \}$ belongs to the
subgroup of $K_2^M(k'(X))$ generated by $\{V^*_i, V^*_{i-1}\}$, 
hence for $n \ge \gon(X) -1$, 
the group $K_n^M(k'(X))$ is
generated by the image of
$\{V^*_{\gon(X)},\ldots,V^*_2 \} \times K^M_{n-\gon(X)+1}(K)$.
This proves the claim.
\end{proof}

\begin{lem}\label{lem:22_zero}
The push-forward homomorphism
$H^2_\cM(\cE^U,\Z(j-1))
\to H^4_\cM(\cE,\Z(j))$ is zero.
\end{lem}
\begin{proof}
Consider the composition
$$
H^2_\cM(\cE^U,\Z(j-1))
\to H^4_\cM(\cE,\Z(j)) \xto{f_*}
H^2_\cM(C,\Z(j-1))
$$ of push-forwards. 
This is the zero map since this factors through 
the group $H^0_\cM(C\setminus U,\Z(j-2))$, 
which is zero by the theorem presented by 
Geisser and Levine
 \cite[Corollary 1.2, p.~56]{Ge-Le2}. 
By Lemma~\ref{lem:Quillen_Harder}, 
the group $H^2_\cM(\cE^U,\Z(j-1))$ is torsion,
hence it suffices to show that
the homomorphism $f_{*,\tors}:
H^4_\cM(\cE,\Z(j))_\tors \to H^2_\cM(C,\Z(j-1))_\tors$ 
induced by $f_*$ is an isomorphism.
Next, consider the commutative diagram
$$
\begin{CD}
H^4_\cM(\cE,\Z(j))_\tors @>{f_{*,\tors}}>>
H^2_\cM(C,\Z(j-1))_\tors \\
@A{\cong}AA @A{\cong}AA \\
H^3_\cM(\cE,\Q/\Z(j)) @>>> 
H^1_\cM(C,\Q/\Z(j-1)) \\
@V{\cong}VV @V{\cong}VV \\
{\displaystyle \bigoplus_{\ell \neq p}}
H^3_\etale(\cEbar,\Q_\ell/\Z_\ell(j))^{\GFq} 
@>>>
{\displaystyle \bigoplus_{\ell \neq p}}
H^1_\etale(\Cbar,\Q_\ell/\Z_\ell(j-1))^{\GFq}.
\end{CD}
$$
Here, the horizontal arrows are push-forward maps,
the upper vertical arrows are boundary maps
obtained from the exact sequence
$0 \to \Z \to \Q \to \Q/\Z \to 0$, and the
lower vertical arrows are those obtained from
Theorem~\ref{thm:motfin}(2)(b)
(the same argument also applies to curves)
further using the weight argument.
The homomorphism at the bottom
is an isomorphism by Lemma~\ref{lem:appl_pi1},
hence $f_{*,\tors}$ is an isomorphism, as desired.
\end{proof}

\subsection{Proofs of theorems}
\begin{proof}[Proof of Theorem~\ref{thm:conseq5}]
From Theorem~\ref{thm:motfin}(2) and Lemma~\ref{lem:Milnor_K},
claims (4) and (5) follow. 
Claim (6)
follows from Theorem~\ref{thm:motfin}(2) and
Corollary~\ref{cor:etaleU_1}.
Using an approach similar to that of the proof of
Corollary~\ref{cor:relative_cv}, we are able to show that
the pull-back map induces an isomorphism 
$H^i_\cM(\cE,\Z(j))_\divi \cong  H^i_\cM(E,\Z(j))_\divi$
for all $i \in \Z$, and that the localization sequence
induces a long exact sequence
\begin{equation}\label{eqn:last_loc_seq}
\begin{array}{rl}
\cdots & \to {\displaystyle \bigoplus_{\wp \in C_0}}
H^{i-2}_\cM(\cE_\wp,\Z(j-1)) \\
& \to H^i_\cM(\cE,\Z(j))_\tors \to  H^i_\cM(E,\Z(j))_\tors
\to \cdots.
\end{array}
\end{equation}
Using the theorem introduced by 
Rost and Voevodsky (i.e., Theorem~\ref{conj:BK} above), 
we observe that for any $\wp \in C_0$,
even if $\cE_\wp$ is singular,
the group $H^i_\cM(\cE_\wp,\Z(j-1))$
is finite for all $i$, is zero for $i \le 0$ or
$i \ge 4$, and is cyclic of order $q_\wp^{j-1}-1$,
where $q_\wp=|\kappa(\wp)|$ is the cardinality
of the residue field at $\wp$, 
for $i=1$.
From the exact sequence (\ref{eqn:last_loc_seq}) and 
Lemma~\ref{lem:22_zero}, we can deduce 
claims (1), (2), and (3) from claims 
(4), (5), and (6), thus completing the proof.
\end{proof}

\begin{proof}[Proof of Theorem~\ref{thm:conseq6}]
Let $U \neq C$. 
Then, by Lemmas~\ref{lem:coker_del43}
and \ref{lem:22_zero}, 
the following sequence is exact:
$$
0 \to H^4_\cM(\cE,\Z(3)) \to
H^4_\cM(\cE_U,\Z(3)) \xto{\partial}
H^3_\cM(\cE^U,\Z(2)) \xto{\alpha} \F_q^\times 
\to 1.
$$
Here,  $\partial$ and $\alpha$ are as in Lemma~\ref{lem:coker_del43},
and the second map is the pull-back.
By taking the inductive limit,
we obtain an exact sequence
\begin{equation}\label{eqn:exactseq_43}
\begin{array}{rl}
0 & \to H^4_\cM(\cE,\Z(3)) \to
H^4_\cM(E,\Z(3)) \\ 
& \xto{\partial^4_{\cM,3}}
\bigoplus_{\wp \in C_0} 
H^3_\cM(\cE_\wp,\Z(2)) \to \F_q^\times 
\to 1.
\end{array}
\end{equation}

By Theorem~\ref{thm:motfin}
and Corollary~\ref{cor:relative_cv}, 
the group $H^4_\cM(E,\Z(3))_\divi$ is zero,
hence using Lemma~\ref{lem:c12}, 
we obtain $K_2(E)_\divi \subset \Ker\, c_{2,3}$.
From the proof of Theorem~\ref{thm:conseq1},
we saw that map $c_{2,2}$ induces
an isomorphism
$(\Ker\, c_{2,3})_\divi \xto{\cong} H^2_\cM(E,\Z(2))_\divi$,
hence 
$c_{2,2}$ induces an isomorphism
$K_2(E)_\divi \cong H^2_\cM(E,\Z(2))_\divi$, which proves
claim (1). Claim (2) follows from
Theorems~\ref{thm:conseq1} and \ref{thm:conseq2},
the commutative diagram (\ref{eqn:CD_H43}),
and exact sequence (\ref{eqn:exactseq_43}).
This completes the proof.
\end{proof}

\appendix
\section{A proposition on the $p$-part}\label{Appendix p-part}
The aim of this Appendix is to provide a proof of 
Proposition~\ref{prop:dRW_main} below, 
which we used in the proof of Theorem~\ref{thm:motfin}.
Nothing in this Appendix is new except for 
the definition of the Frobenius map on the inductive 
limit (not on the inverse limit) given in Section~\ref{mod. Frob}.  
A similar presentation has already been provided 
in the work of 
Milne \cite{Milne1} and Nygaard \cite{Nygaard}. 
\begin{prop}\label{prop:dRW_main}
Let $X$ be a smooth projective geometrically
connected surface over a finite field $\F_q$ of cardinal $q$ 
of characteristic $p$.
Let $\dRWlog{n}{i}$ denote the
logarithmic de Rham-Witt sheaf (\cf \cite[I, 5.7, p.~596]{Illusie}).
Then, the inductive limit 
$\varinjlim_n H^{0}_\etale(X,\dRWlog{n}{2})$ with 
respect to multiplication-by-$p$ is finite 
of order 
$|\Hom(\Pic^o_{X/\F_q},\Gm)|_p^{-1} 
\cdot |L(h^2(X),0)|_p^{-1}$.
Here, 
$\Hom(\Pic^o_{X/\F_q},\Gm)$ 
denotes the set of homomorphisms 
$\Pic^o_{X/\F_q}\to \Gm$ of $\F_q$-group schemes,
and $L(h^2(X),s)$ is the (Hasse-Weil) $L$-function
of $h^2(X)$. 
\end{prop}

\subsection{The de Rham-Witt complex}
In this Appendix, let $k$ be a perfect field 
of characteristic $p$.
Let $X$ be a scheme of dimension $\delta$
that is proper over $\Spec\, k$.
For $i,n \in \Z$, let $\dRW{n}{\bullet}$ 
denote the de Rham-Witt complex 
(\cf \cite[I, 1.12, p.~548]{Illusie}) of the ringed topos 
of schemes over $X$ with Zariski topology.
We let $R :\dRW{n}{i} \to \dRW{n-1}{i}$,
$F :\dRW{n}{i} \to \dRW{n-1}{i}$,
and $V :\dRW{n}{i} \to \dRW{n+1}{i}$
denote the restriction, the Frobenius, and
the Verschiebung, respectively.
For each $i \in \Z$, the sheaf
$\dRW{n}{i}$ has a canonical structure of coherent 
$W_n \cO_X$-module,
which enables us to regard $\dRW{n}{i}$ as
an \'{e}tale sheaf. 
Therefore, in this Appendix, 
we focus on the category of \'{e}tale sheaves
on schemes over $X$.
\subsection{Logarithmic de Rham-Witt sheaves}
For $n \in \Z$, let 
$\dRWlog{n}{i} \subset \dRW{n}{i}$ denote the
logarithmic de Rham-Witt sheaf (\cf \cite[I, 5.7, p~.596]{Illusie}).
%
\begin{lem}
The homomorphism $V:\dRW{n}{i} \to
\dRW{n+1}{i}$ sends $\dRWlog{n}{i}$ into
$\dRWlog{n+1}{i}$.
\end{lem}
\begin{proof}
Let $x \in \dRWlog{n}{i}$ be an \'{e}tale local section.
By the definition of $\dRWlog{n}{i}$, there exists
an \'{e}tale local section $y\in \dRWlog{n+1}{i}$ such that $x=Ry$.
We observe that $Ry=Fy$, hence $Vx=VRy=VFy=py \in
\dRWlog{n+1}{i}$.
\end{proof}
Let $\dRCW{i}$ denote the inductive limit
$\dRCW{i} = \varinjlim_{n,\, V} \dRW{n}{i}$ with respect
to $V$. The above lemma enables us to define 
the inductive limit 
$\dRCWlog{i} = \varinjlim_{n,\, V} \dRWlog{n}{i}$.

\subsection{Modified Frobenius operator}\label{mod. Frob}
In this subsection we define an operator 
$F': \dRCW{i} \to \dRCW{i}$ such that
the sequence 
\begin{equation}\label{eqn:dRW_exact}
0 \to \dRCWlog{i} \to \dRCW{i} \xto{1-F'}
\dRCW{i} \to 0
\end{equation}
is exact.

For $n \ge 0$, let $\dRWt{n}{i}$ denote the
cokernel 
of the homomorphism 
$V^{n}: \Omega^i_X =\dRW{1}{i} \to \dRW{n+1}{i}$.
The homomorphisms $R$, $F$ and $V$ on $\dRW{n+1}{i}$ 
induce homomorphisms on $\dRWt{n}{i}$, which we denote
using the same notation. 
If $n \ge 1$, the homomorphisms 
$R, F:\dRW{n+1}{i} \to \dRW{n}{i}$
factor through the canonical surjection
$\dRW{n+1}{i} \to \dRWt{n}{i}$.
%
We let $\wt{R}, \wt{F} :\dRWt{n}{i} \to \dRW{n}{i}$ 
denote the induced homomorphisms. Then, 
both $\wt{R}$ and $\wt{F}$ commute with $R$, $F$ and $V$.
For $n\ge 0$, we let $\dRWtlog{n}{i}$ denote 
the image of $\dRWlog{n+1}{i}$ by the canonical 
surjection $\dRW{n+1}{i} \to \dRWt{n}{i}$.
The restriction of $\wt{R}:\dRWt{n}{i} \to \dRW{n}{i}$
to $\dRWtlog{n}{i}$ gives a surjective homomorphism
$\wt{R}_{\log} : \dRWtlog{n}{i} \to \dRWlog{n}{i}$.

\begin{lem}\label{dRW_lem1}
The homomorphisms $\wt{R}$, $\wt{R}_{\log}$
induce isomorphisms 
$$
\varinjlim_{n\ge 0,\, V}
\dRWt{n}{i} \cong \dRCW{i},\ 
\varinjlim_{n\ge 0,\, V} \dRWtlog{n}{i} 
\cong \dRCWlog{i}.
$$
\end{lem}
\begin{proof}
Surjectivity here is clear.
From \cite[I, Proposition 3.2, p.~568]{Illusie},
it follows that the kernel of $\wt{R}$ equals the image of
the composition
$\dRW{1}{i} \xto{dV^n} \dRW{n+1}{i}
\surj \dRWt{n}{i}$. Since $Vd=pdV$,
we have $V(\Ker\, \wt{R})=0$. This
proves the injectivity.
\end{proof}

We observe that $\dRWtlog{n}{i}$
is contained in the kernel of
$\wt{R}-\wt{F}:\dRWt{n}{i} \to \dRW{n}{i}$,
hence
\begin{equation}\label{eqn:dRW_complex}
0 \to \dRWtlog{n}{i} \to 
\dRWt{n}{i} \xto{\wt{R}-\wt{F}} \dRW{n}{i}
\to 0
\end{equation}
is a complex.
\begin{lem}\label{dRW_lem2}
The inductive limit 
$$
0 \to \varinjlim_{n \ge 0,\, V} 
\dRWtlog{n}{i} \to \varinjlim_{n \ge 0,\, V} 
\dRWt{n}{i} \to \dRCW{i} \to 0
$$
of (\ref{eqn:dRW_complex}) with respect to $V$
is exact.
\end{lem}
\begin{proof}
The argument in the proof of 
\cite[I, Th\'{e}or\`{e}me 5.7.2, p.~597]{Illusie}
shows that the kernel of $R-F:\dRW{n+1}{i}
\to \dRW{n}{i}$ is contained in
$\dRWlog{n+1}{i} + \Ker\, R$, hence the
claim follows from Lemma~\ref{dRW_lem1}.
\end{proof}

The inductive limit of $\wt{F}:\dRWt{n}{i}
\to \dRWt{n+1}{i}$ gives the endomorphism
$F' : \dRCW{i} \cong
\varprojlim_{n \ge 1,\, V}
\dRWt{n}{i} \to \dRCW{i}$. By Lemmas~\ref{dRW_lem1}
and~\ref{dRW_lem2}, we have a canonical
exact sequence (\ref{eqn:dRW_exact}).

\subsection{The duality}
Let $H^*(X,\dRW{n}{i})$ denote the cohomology groups
of $\dRW{n}{i}$ with respect to the Zariski topology.

The trace map $\Tr: H^\delta(X,\dRW{n}{\delta})\cong W_n(\F_q)$
is defined in \cite[4.1.3, p.~49]{Illusie2}. This commutes with 
homomorphisms $R$, $F$ and $V$.
For $0 \le i,j \le \delta$, the product 
$m:\dRW{n}{i} \times \dRW{n}{\delta-i} \to
\dRW{n}{\delta}$ gives a $W_n(k)$-bilinear 
paring 
$$
\begin{small}
(\ ,\ ): H^j(X,\dRW{n}{i}) \times
H^{\delta-j}(X,\dRW{n}{\delta-i}) 
\to H^\delta(X,\dRW{n}{\delta}) \xto{\Tr}W_n(k).
\end{small}
$$
By \cite[Corollary 4.2.2, p.~51]{Illusie2}, this pairing is perfect.

Since $m \circ(\id \otimes V) = 
V \circ m \circ (F \otimes \id)$,
the diagram
$$
\begin{CD}
\dRW{n+1}{i} @. \times @. \dRW{n+1}{\delta -i} @>>> W_{n+1}(k) \\
@V{F}VV @. @A{V}AA @A{V}AA \\
\dRW{n}{i} @. \times @. \dRW{n}{\delta -i} @>>> W_{n}(k) \\
\end{CD}
$$
is commutative, hence this induces an isomorphism
\begin{equation}\label{eqn:dRW_isom}
H^{\delta -j}(X,\dRCW{\delta -i}) \cong
\varinjlim_n \Hom_{W_n(k)}(H^j(X,\dRW{n}{i}),W_n(k)),
\end{equation}
where the transition map 
in the inductive limit of the right hand side is
given by $f \mapsto V\circ f\circ F$.
We endow each $H^j(X,\dRW{n}{i})$ with the
discrete topology. We set
$H^j(X,W'\Omega^i_X) = \varprojlim_{n,\, F}
H^j(X, \dRW{n}{i})$ and endow it with the induced topology.
Next, we turn $H^j(X,W'\Omega^i_X)$ into a $W(k)$-module
by letting $a\cdot (b_n) = (\sigma^{-n}(a) b_n)$
for $a\in W(k)$, $b_n \in H^j(X,\dRW{n}{i})$.
We set $D= \varinjlim_{n,\, V}W_{n}(k)$ and
endow it with the discrete topology.
We turn $D$ into a $W(k)$-module
by letting $a\cdot c_n = \sigma^{-n}(a) c_n$
for $a\in W(k)$, $c_n \in W_n(k)$.
Then, the right hand side of (\ref{eqn:dRW_isom}) equals
$\Hom_{W(k),\cont}(H^j(X,W'\Omega^i_X),D)$.
The homomorphism $R:H^j(X, \dRW{n}{i}) \to H^j(X, \dRW{n-1}{i})$
induces an endomorphism 
$R' :H^j(X,W'\Omega^i_X) \to H^j(X,W'\Omega^i_X)$.
The Frobenius endomorphism $\sigma:W_n(k) \to W_n(k)$
induces an endomorphism $\sigma : D \to D$.

\begin{lem}\label{lem:A5}
Under the isomorphism (\ref{eqn:dRW_isom}),
the endomorphism 
\[
F' :H^{\delta -j}(X,\dRCW{\delta -i})\to
H^{\delta -j}(X,\dRCW{\delta -i})\]
is identified with
the endomorphism of
$\Hom_{W(k),\cont}(H^j(X,W'\Omega^i_X),D)$,
that sends a homomorphism 
$f: H^j(X,W'\Omega^i_X) \to D$ to 
the homomorphism $\sigma \circ f \circ R'$.
\end{lem}
\begin{proof}
The proof here is immediate from the definition of
the isomorphism (\ref{eqn:dRW_isom})
and the module $D$.
\end{proof}
\subsection{The degree zero case}\label{sec:A5}
We are primarily concerned with the case in which $i=0$.
We denote $H^j(X,W'\Omega^0_X)$ by $H^j(X,W'\cO_X)$.
Recall that $F:\dRW{n}{0} \to \dRW{n-1}{0}$ equals
the composition $W_n \cO_X \xto{\sigma} W_n \cO_X \xto{R}
W_{n-1}\cO_X$. From \cite[II, Proposition 2.1, p.~607]{Illusie},
it follows that 
$H^j(X,W\Omega^i_X) \to \varprojlim_{n,\, R} H^j(X,\dRW{n}{i})$ 
is an isomorphism, hence $H^j(X,W'\cO_X)$ is isomorphic
to the projective limit
$$
\wt{H}^j(X,W\cO_X) =
\varprojlim[\cdots \xto{\sigma} 
H^j(X,W\cO_X) \xto{\sigma} H^j(X,W\cO_X)].
$$
The endomorphism $\sigma : H^j(X,W\cO_X) \to H^j(X,W\cO_X)$
induces an automorphism $\sigma:\wt{H}^j(X,W\cO_X) \xto{\cong}
\wt{H}^j(X,W\cO_X)$. We observe here that
the endomorphism $R'$ on $H^j(X,W'\cO_X)$ corresponds to
the endomorphism $\sigma^{-1}$ on $\wt{H}^j(X,W\cO_X)$.

Let $K =\Frac\, W(k)$ denote the field of fractions of $W(k)$.
The homomorphism $\sigma^n/p^n : W_n(k) \to K/W(k)$
for each $n \ge 1$ 
induces a canonical isomorphism $D \cong K/W(k)$
of $W(k)$-modules that commutes with the action of $\sigma$.

\subsection{Proof of Proposition~\ref{prop:dRW_main}}
Suppose that $k= \F_q$. Then 
by Lemma~\ref{lem:A5}, $H^{0}(X,\dRCW{d})$
is isomorphic to the Pontryagin dual of $\wt{H}^{\delta}(X,W\cO_X)$,
hence the group 
$$
H^{0}(X,\dRCWlog{\delta})
\cong \Ker[H^{0}(X,\dRCW{\delta})\xto{1-F'} H^{0}(X,\dRCW{\delta})]
$$
is isomorphic to the 
Pontryagin dual of the cokernel of 
$1-\sigma^{-1}$ on $\wt{H}^{\delta}(X,W\cO_X)$.

\begin{prop}\label{prop:dRW_sub}
Let $k=\F_q$ be a finite field and $X$ 
be a scheme of dimension $\delta$
that is smooth and projective over $\Spec\, k$. 
Suppose that the $V$-torsion part $T$ of 
$H^{\delta}(X,W\cO_X)$ is finite.
Then, $H^{0}(X,\dRCW{\delta})$ is a finite group of
order $|T^{\sigma}| \cdot |L(h^\delta (X),0)|_p^{-1}$.
Here, $T^\sigma$ denotes the $\sigma$-invariant part of
$T$.
\end{prop}
\begin{proof}
By the argument above, the order of
$H^{0}(X,\dRCW{\delta})$ equals the order of the
cokernel of $1-\sigma$ on $\wt{H}^{\delta}(X,W\cO_X)$
if it is finite.
The torsion subgroup of $\wt{H}^{\delta}(X,W\cO_X)$ is finite
since it injects into $T$.
By \cite[II, Corollaire 3.5, p.~616]{Illusie},
$\wt{H}^{\delta}(X,W\cO_X)\otimes_{\Z_p} \Q_p$
is isomorphic to the slope-zero part of
$H^{\delta}_{\cris} (X/W(k))\otimes_{\Z_p} \Q_p$,
hence the claim follows.
\end{proof}

\begin{proof}[Proof of Proposition~\ref{prop:dRW_main}]
Let the notation be as above and suppose that $\delta=2$. 
Then, by \cite[II, Remarque 6.4, p.~641]{Illusie}, 
the module $T$ in the above proposition
is canonically isomorphic to the group
$$
\Hom_{W(\F_q)}
(M(\Pic_{X/\F_q}^o /\Pic_{X/\F_q,\red}^o),K/W(\F_q)),
$$
where $M(\ )$ 
denotes the contravariant Dieudonn\'{e} module functor.
In particular, $T$ is a finite group.
Let $T_\sigma$ denote the $\sigma$-coinvariants of $T$.
Then, by Dieudonn\'{e} theory (\cf \cite{Demazure}),
$\Hom_{W(\F_q)}(T_\sigma, K/W(\F_q))$ is canonically
isomorphic to $\Hom(\Pic_{X/\F_q}^o,\Gm)$,
hence the claim follows from 
Proposition~\ref{prop:dRW_sub}.
\end{proof}

\nocite{*}
\bibliographystyle{cdraifplain}
\bibliography{K1K2bib01}
\end{document}